\documentclass[preprint,1p,times]{elsarticle}

\usepackage{amsmath,amssymb,amsfonts,amsthm}
\usepackage{mathtools}
\usepackage{graphicx}
\usepackage[all,dvips]{xy}
\usepackage{bbm}
\usepackage{hyperref}
\usepackage{breakurl}

\newcommand{\term}[1]{{\it#1}}

\newcommand{\kk}{\Bbbk}
\renewcommand{\kk}{k}
\newcommand{\NN}{\mathbb{N}}
\newcommand{\ZZ}{\mathbb{Z}}
\newcommand{\CC}{\mathbb{C}}
\newcommand{\unit}{\mathbbm{1}}

\newcommand{\cA}{\mathcal{A}}

\newcommand{\cC}{\mathcal{C}}
\newcommand{\cD}{\mathcal{D}}
\newcommand{\cE}{\mathcal{E}}

\newcommand{\cM}{\mathcal{M}}
\newcommand{\cP}{\mathcal{P}}
\newcommand{\cS}{\mathcal{S}}
\newcommand{\cT}{\mathcal{T}}
\newcommand{\cW}{\mathcal{W}}
\newcommand{\TL}{\mathcal{TL}}

\newcommand{\pW}{\mathcal{W}'}

\newcommand{\fS}{\mathfrak{S}}
\newcommand{\fB}{\mathfrak{B}}

\DeclareMathOperator{\Hom}{Hom}
\DeclareMathOperator{\End}{End}
\DeclareMathOperator{\Ind}{Ind}
\DeclareMathOperator{\Res}{Res}
\DeclareMathOperator{\id}{id}
\DeclareMathOperator{\Id}{Id}
\DeclareMathOperator{\Tr}{Tr}
\DeclareMathOperator{\Add}{\mathcal{A}\mathit{dd}}
\DeclareMathOperator{\Kar}{\mathcal{K}\mathit{ar}}
\DeclareMathOperator{\Ps}{\mathcal{P}\mathit{s}}

\DeclareMathOperator{\HOM}{\mathcal{H}\mathit{om}}

\newcommand{\Triv}{{\mathcal{T}\mathit{riv}}}
\newcommand{\Cat}{{\mathbf{Cat}}}

\newcommand{\PsCat}{{\mathbf{PsCat}}}
\newcommand{\tCat}{{{\otimes}\text{-}\Cat}}

\newcommand{\tPsCat}{{{\otimes}\text{-}\PsCat}}
\newcommand{\uCat}{{\unit\text{-}\Cat}}
\newcommand{\uPsCat}{{\unit\text{-}\PsCat}}

\newcommand{\Mod}{{\mathcal{M}\mathit{od}}}
\newcommand{\Rep}{\mathcal{R}\mathit{ep}}

\DeclareMathOperator{\ev}{ev}
\DeclareMathOperator{\coev}{coev}
\newcommand{\op}{\mathrm{op}}

\newcommand{\br}[2][d]{[{#2}]_{#1}}
\newcommand{\bigbr}[2][d]{\bigl[{#2}\bigr]_{#1}}
\newcommand{\ag}[2][t]{\langle{#2}\rangle_{#1}}
\newcommand{\aag}[2][t]{\langle\mspace{-3mu}\langle{#2}\rangle\mspace{-3mu}\rangle_{#1}}
\newcommand{\bigag}[2][t]{\bigl\langle{#2}\bigl\rangle_{#1}}
\newcommand{\Bigag}[2][t]{\Bigl\langle{#2}\Bigl\rangle_{#1}}
\newcommand{\bigaag}[2][t]{\bigl\langle\mspace{-3mu}\bigl\langle{#2}\bigl\rangle\mspace{-3mu}\bigl\rangle_{#1}}

\newcommand{\eqmap}{\stackrel{\hspace{-2pt}\sim}{\rule{0pt}{2pt}\smash\to}}
\newcommand{\bijection}{\stackrel{1:1}{\rule{0pt}{2.5pt}\smash\longleftrightarrow}}
\newcommand{\bijectivemap}{\stackrel{1:1}{\rule{0pt}{2.5pt}\smash\longrightarrow}}

\newcommand{\inbox}[1]{{\xybox{*+{#1}*\frm{-}}}}

\theoremstyle{plain}
  \newtheorem{theorem}{Theorem}[section]
  \newtheorem{proposition}[theorem]{Proposition}
  \newtheorem{lemma}[theorem]{Lemma}
  \newtheorem{corollary}[theorem]{Corollary}

\theoremstyle{definition}
  \newtheorem{definition}[theorem]{Definition}
  \newtheorem{notation}[theorem]{Notation}
  \newtheorem{assumption}[theorem]{Assumption}

  \newtheorem{remark}[theorem]{Remark}
  
  \newtheorem{example}[theorem]{Example}

\let\origboxtimes=\boxtimes
\def\boxtimes{\DOTSB\mathbin{\mathchoice{\dboxtimes}{\dboxtimes}{\sboxtimes}{\ssboxtimes}}\DOTSB}
\def\dboxtimes{\raise-.92pt\hbox{$\origboxtimes$}}
\def\sboxtimes{\raise-.89pt\hbox{$\scriptstyle\origboxtimes$}}
\def\ssboxtimes{\raise-.89pt\hbox{$\scriptscriptstyle\origboxtimes$}}

\def\barboxtimes{\DOTSB\mathbin{\overline{\mathchoice{\dboxtimes}{\dboxtimes}{\sboxtimes}{\ssboxtimes}}}\DOTSB}

\DeclareMathOperator*{\bigboxtimes}{\raise-2.45pt\hbox{\scalebox{2}{$\boxtimes$}}}
\DeclareMathOperator*{\bigbarboxtimes}{\overline{\raise-2.45pt\hbox{\scalebox{2}{$\boxtimes$}}}}

\let\Kar\relax \DeclareMathOperator{\Kar}{\mathcal{K}\mspace{-2mu}\mathit{ar}}
\let\Ps\relax \DeclareMathOperator{\Ps}{\mathcal{P}\mspace{-2mu}\mathit{s}}

\renewcommand{\Triv}{{\mathcal{T}\mspace{-4mu}\mathit{riv}}}
\renewcommand{\TL}{\mathcal{T\!\!\!\;L}}

\let\HOM\relax \DeclareMathOperator{\HOM}{\mathcal{H}\mspace{-2mu}\mathit{om}}

\journal{arXiv.org}

\begin{document}

\begin{frontmatter}



\title{On representation categories of wreath products\\in non-integral rank}


\author[utms]{Masaki Mori}
\ead{mori@ms.u-tokyo.ac.jp}

\address[utms]{Graduate School of Mathematical Sciences, The University of Tokyo, Tokyo 153, Japan}

\begin{abstract}
For an arbitrary commutative ring $k$ and $t\in k$,
we construct a 2-functor $\mathcal{S}_t$ which sends a tensor category to a new tensor category.
By applying it to the representation category of a bialgebra
we obtain a family of categories which interpolates
the representation categories of the wreath products of the bialgebra.
This generalizes the construction of Deligne's category $\mathrm{Rep}(S_t,k)$
for representation categories of symmetric groups.
\end{abstract}


\end{frontmatter}



\section{Introduction}
Let $\kk$ be a commutative ring.
In \cite{Deligne:2007}, Deligne introduced a tensor category
$\mathrm{Rep}(S_t,\kk)$ for an arbitrary $t\in\kk$,
``the category of representations of the symmetric group of rank $t$ over $k$''
in some sense.
This category is consisting of objects which imitate some classes of
representations of the symmetric group of indefinite rank.
If the rank $t$ is a natural number,
the usual representation category of the symmetric group will be
restored by taking a quotient of Deligne's category.

Generalizations of Deligne's category are considered by many authors,
e.g.\ Knop \cite{Knop:2006,Knop:2007}, Etingof~\cite{Etingof:2009} and Mathew~\cite{Mathew:2010}.
In this paper we give another generalization:
we extend Deligne's construction
to a 2-functor $\cS_t$ which sends a tensor category to another tensor category.
In other words, for each tensor category $\cC$ the 2-functor $\cS_t$ provides a new tensor category $\cS_t(\cC)$.
Using this 2-functor,
Deligne's category is obtained by applying it to the trivial tensor category
consisting of only one object.
Moreover if we apply $\cS_t$ to a representation category of some bialgebra,
we will get a family of new tensor categories which interpolates the representation categories
of the wreath products of the bialgebra.
For a finite group $G$, Knop's interpolation $\mathrm{Rep}(G\wr S_t,\kk)$
is essentially the same as ours
but in general either construction does not include the other.
For example, in Knop's category $\cT(\cA,\delta)$, the tensor product is always symmetric
and every object has its dual;
however our $\cS_t(\cC)$ satisfies neither of them unless the base category $\cC$ does.

The 2-functor $\cS_t$ naturally preserves various
structures of categories such as
duals, braidings (symmetric or not), twists, traces and so on (see \ref{sec:cat_with_str}).
In particular, if $\cC$ is a braided tensor category then
so is $\cS_t(\cC)$.
In this case, we can represent and calculate morphisms in $\cS_t(\cC)$
by string diagrams.
These diagrams are generalizations of those used for partition algebras~\cite{Jones:1994,Martin:1994}
and can be regarded as ``$\cC$-colored'' variants of them.
For example, there is a morphism in $\cS_t(\cC)$
represented by a diagram
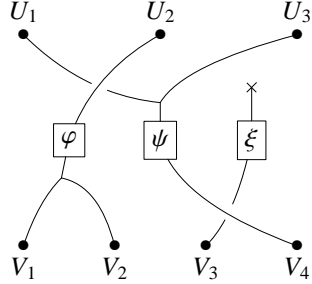
\begin{gather*}
\begin{xy}
(-12,0)*+{\varphi}*\frm{-}="p",
(0,0)*+{\psi}*\frm{-}="q",
(12,0)*+{\xi}*\frm{-}="r",
(-18,14)*={\bullet}="u1"+(0,3)*{U_1},
(0,14)*={\bullet}="u2"+(0,3)*{U_2},
(18,14)*={\bullet}="u3"+(0,3)*{U_3},
(-18,-14)*={\bullet}="v1"+(0,-3)*{V_1},
(-6,-14)*={\bullet}="v2"+(0,-3)*{V_2},
(6,-14)*={\bullet}="v3"+(0,-3)*{V_3},
(18,-14)*={\bullet}="v4"+(0,-3)*{V_4},
(-8,7)*{\hole}="x",
(8.5,-9.5)*{\hole}="y",
"p"+(-1,-5)="p-",
"q"+(0,5)="q+",
"r"+(0,7)="r+",
"p";"p-" **\dir{-},
"q";"q+" **\dir{-},
"r";"r+" **\dir{-} ?>*\dir{x},
"p";"u2" **\crv{(-9,9)},
"q+"."x";"u1" **\crv{(-10,6)}, "q+";"u1"."x" **\crv{(-10,6)},
"q+";"u3" **\crv{(3,10)},
"p-";"v1" **\crv{(-16,-8)},
"p-";"v2" **\crv{(-8,-6)},
"r"."y";"v3" **\crv{(10,-10)}, "r";"v3"."y" **\crv{(10,-10)},
"q";"v4" **\crv{(2,-8)},
\end{xy}
\end{gather*}
where $U_1$, $U_2$, $U_3$, $V_1$, $V_2$, $V_3$ and $V_4$ are objects of $\cC$
and $\varphi$, $\psi$ and $\xi$ are suitable morphisms in $\cC$.
Composition of such morphisms is expressed by vertical connection of diagrams
and tensor product by horizontal arrangement.
By Theorem~\ref{thm:universality} we also prove that
$\cS_t(\cC)$ can be described in terms of generators (i.e.\ pieces of diagrams)
and relations (i.e.\ local transformation of diagrams).
In fact it has a universal property which says that
it is the smallest braided tensor category which satisfies these relations.
Its generators and the relations are listed in Proposition~\ref{prop:diagram_transform}.

In the rest of the paper we extend the result of Comes and Ostrik~\cite{ComesOstrik:2011}
which describes the structure of Deligne's category.
Assume that $\kk$ is a field of characteristic zero
and let $\cC$ be an abelian semisimple tensor category
whose every simple object $U$ satisfies $\End_\cC(U)\simeq\kk$.
In this case, we can completely describe the structure of the category $\cS_t(\cC)$;
we classify the indecomposable objects, simple objects and blocks.
We parameterize them using sequences of Young diagrams indexed by the simple objects of $\cC$. 
See Theorem~\ref{thm:blocks} for details.
In fact, ignoring the structure of tensor product, this category is equivalent to the direct sum of some copies of
Deligne's category $\mathrm{Rep}(S_{t-m},\kk)$ as $m\in\NN$ varies.
In particular, if $t\not\in\NN$ then $\cS_t(\cC)$ is also abelian semisimple
and we can produce a large number of new abelian semisimple tensor categories
which can not be realized as representation category of algebraic structure.

I would like to thank Hisayosi Matumoto who taught me
about representation theory from the basics for a long time.
I am also grateful to my colleagues, especially
to Hideaki Hosaka and Hironori Kitagawa for many useful suggestions.

\subsection{Conventions and Notations}
In this paper, a ring means an associative ring with unit
and ring homomorphisms preserve the unit.
A module over a ring is always a left module and unital.
We use the symbol $\kk$ to denote
a commutative ring and
for $\kk$-modules $U$ and $V$,
we write $U\otimes V$ instead of $U\otimes_\kk V$ for short.

For a category $\cC$ the notation $U\in\cC$ means that $U$ is an object of $\cC$.
For $U,V\in\cC$, we denote by $\Hom_\cC(U,V)$ the set of morphisms from $U$ to $V$.
If $U=V$ we also denote it by $\End_\cC(U)$.
For a natural transformation $\eta\colon F\to G$
between two functors $F,G\colon\cC\to\cD$,
we denote its component at an object $U\in\cC$ by
$\eta(U)\colon F(U)\to G(U)$.
We do not ask the meanings of the terms ``small'' and ``large''
about sizes of categories;
some readers may interpret them with class theory while others prefer to use Grothendieck universes.

We include zero in the set of natural numbers, so $\NN=\{0,1,2,\dots\}$.

\section{The Language of Linear Categories}
In this section we quickly review the theory of linear categories.

\subsection{Definition and Properties}
\begin{definition}
\label{def:lincat}
\begin{enumerate}
\item A category $\cC$ is called a {\itshape $\kk$-linear category}
if for each objects $U,V\in\cC$, $\Hom_\cC(U,V)$ is endowed with a structure of $\kk$-module
and the composition of morphisms is $\kk$-bilinear.
\item A functor $F\colon\cC\to\cD$ between
two $\kk$-linear categories is called {\itshape $\kk$-linear}
if for any $U,V\in\cC$ the map $F\colon\Hom_\cC(U,V)\to\Hom_\cD(F(U),F(V))$ is $\kk$-linear.
We define \term{$\kk$-multilinear functor} in the same way.
\item A \term{$\kk$-linear transformation} is just a natural transformation between two $\kk$-linear functors.
\end{enumerate}
\end{definition}
Some authors call a $\kk$-linear category a {\itshape $\kk$-preadditive category} or
a {\itshape $\kk$-category}.
These below are examples of $\kk$-linear categories which we use later.

\begin{definition}
\label{def:lincat_ex}
\begin{enumerate}
\item We denote by $\Triv_\kk$ the \term{trivial $\kk$-linear category} consisting of a single object 
$\unit\in\Triv_\kk$ which satisfies $\End_{\Triv_\kk}(\unit)\simeq\kk$.
\item For a $\kk$-algebra $A$, we denote by $\Mod(A)$ the category consisting of $A$-modules and $A$-homomorphisms,
and $\Rep(A)$ the full subcategory of $\Mod(A)$ consisting of
$A$-modules which are finitely generated and projective over $\kk$.
\item For two $\kk$-linear categories $\cC$ and $\cD$,
we denote by $\HOM_\kk(\cC,\cD)$ the category consisting of $\kk$-linear functors from $\cC$ to $\cD$
and $\kk$-linear transformations between them.
\end{enumerate}
\end{definition}

In a $\kk$-linear category finite product and finite coproduct
coincide and both are called direct sum.
$\kk$-linear functors and transformations are automatically compatible with taking direct sum.

\begin{definition}
\label{def:addkarps}
Let $\cC$ be a $\kk$-linear category.
\begin{enumerate}
\item $\cC$ is called \term{additive} if for any $U_1,\dots,U_m\in\cC$ there exists
their direct sum $U_1\oplus\dots\oplus U_m\in\cC$ (including zero object for $m=0$).
\item $\cC$ is called \term{Karoubian} (or \term{idempotent complete}) if for any $U\in\cC$
and any idempotent $e=e^2\in\End_\cC(U)$
there exists its image $e U\in\cC$.
In other words, $\cC$ is Karoubian if every idempotent $e\in\End_\cC(U)$
admits a direct sum decomposition $U\simeq eU\oplus(1-e)U$.
\item $\cC$ is called \term{pseudo-abelian} if it is additive and Karoubian.
\end{enumerate}
\end{definition}
For example, $\Mod(A)$ and $\Rep(A)$ are both pseudo-abelian $\kk$-linear categories.
The category $\HOM_\kk(\cC,\cD)$ of $\kk$-linear functors is additive or Karoubian
if the target category $\cD$ is.

\begin{definition}
\label{def:homfin}
A $\kk$-linear category
$\cC$ is called \term{hom-finite} (resp.\ \term{projective})
if $\Hom_\cC(U,V)$ is finitely generated (resp.\ projective) over $\kk$ for every $U,V\in\cC$.
\end{definition}
For example, $\Rep(\kk)$ is clearly hom-finite and projective.
If $\kk$ is Noetherian $\Rep(A)$ is also hom-finite
for any $\kk$-algebra $A$ since $\Hom_A(U,V)\subset\Hom_\kk(U,V)$.
Similarly if $\kk$ is a hereditary ring $\Rep(A)$ is automatically projective.

\begin{definition}
Let $\cC$ be a pseudo-abelian $\kk$-linear category.
An \term{indecomposable object} in $\cC$ is an object $U$ such that
$U\simeq U_1\oplus U_2$ implies either $U_1\simeq0$ or $U_2\simeq0$.
$\cC$ is called a \term{Krull--Schmidt category}
if it satisfies the following two conditions:
\begin{enumerate}
\item every object in $\cC$ is a finite direct sum of indecomposable objects,
\item the endomorphism ring of each indecomposable object in $\cC$ is a local ring.
\end{enumerate}
\end{definition}
It is clear that every hom-finite pseudo-abelian linear category over a field
is a Krull--Schmidt category.
In such a category, the factors in the indecomposable decomposition of an object is uniquely determined.
\begin{theorem}
\label{thm:krull_schmidt}
Let $\cC$ be a Krull--Schmidt category.
Let $U\simeq V_1\oplus\dots\oplus V_m\simeq W_1\oplus\dots\oplus W_n\in\cC$
be two decompositions of an object into indecomposable objects.
Then $m=n$ and $V_i\simeq W_i$ after reordering if necessary.
\end{theorem}
This is a generalization of the usual Krull--Schmidt theorem for modules over a ring,
and the proof of them are same.
See e.g.\ \cite{Benson:1991}.
So to describe the structure of a Krull--Schmidt category
all we need is the classification of indecomposable objects
and morphisms between them.

\subsection{Envelopes}

A $\kk$-linear category is not necessarily additive nor Karoubian in general;
so the direct sum of objects or the image of an idempotent does not necessarily exist.
But we can formally add the results of these operations into our category
to make a new category including them.

\begin{definition}
\label{def:envelopes}
Let $\cC$ be a $\kk$-linear category.
\begin{enumerate}
\item Define the $\kk$-linear category $\Add(\cC)$ as follows:
\begin{description}
\item[Object] A finite tuple $(U_1,\dots,U_m)$ of objects in $\cC$.
\item[Morphism] $\Hom_{\Add(\cC)}((U_1,\dots,U_m),(V_1,\dots,V_n))\coloneqq\bigoplus_{i,j}\Hom_\cC(U_i,V_j)$
and the composition of morphisms is same as the product of matrices.
\end{description}
We simply denote $(U)\in\Add(\cC)$ by $U$, then $(U_1,\dots,U_m)\simeq U_1\oplus\dots\oplus U_m$
and the empty tuple $()$ is a zero object.
$\Add(\cC)$ is called the \term{additive envelope} of $\cC$.

\item Define the $\kk$-linear category $\Kar(\cC)$ as follows:
\begin{description}
\item[Object] A pair $(U,e)$ of an object $U\in\cC$ and an idempotent $e=e^2\in\End_\cC(U)$.
\item[Morphism] $\Hom_{\Kar(\cC)}((U,e),(V,f))\coloneqq f\circ\Hom_\cC(U,V)\circ e$.
\end{description}
We denote $(U,\id_U)\in\Kar(\cC)$ by $U$, then $(U,e)\simeq e U$.
$\Kar(\cC)$ is called the \term{Karoubian envelope} (or the \term{idempotent completion}) of $\cC$.

\item $\Ps(\cC)\coloneqq\Kar(\Add(\cC))$ is called the \term{pseudo-abelian envelope} of $\cC$.
\end{enumerate}
\end{definition}

Clearly $\Add(\cC)$ is additive and $\Kar(\cC)$ is Karoubian.
$\Ps(\cC)$ is pseudo-abelian
since $\Kar(\cC)$ is additive when $\cC$ is: $(U,e)\oplus(V,f)\simeq(U\oplus V,e\oplus f)$.
The base category $\cC$ is embedded in $\Add(\cC)$ (resp.\ $\Kar(\cC)$, $\Ps(\cC)$) as a full subcategory
and this embedding is a category equivalence if and only if $\cC$ is additive (resp.\ Karoubian, pseudo-abelian).

\begin{example}
\begin{align*}
\Add(\Triv_\kk)&\simeq\text{(The category of finitely generated free $\kk$-modules)},\\
\Ps(\Triv_\kk)&\simeq\text{(The category of finitely generated projective $\kk$-modules)}\\&=\Rep(\kk).
\end{align*}
\end{example}

To describe the universal properties of the operation $\Ps$
we should use the notions of \term{2-categories} and \term{2-functors}.
For their definitions, see e.g.\ \cite{Leinster:2004}.
Let us denote by $\Cat_\kk$ the 2-category consisting of (small) $\kk$-linear categories,
functors and transformations,
and by $\PsCat_\kk$
the full sub-2-category of $\Cat_\kk$ consisting of pseudo-abelian
$\kk$-linear categories.
For a $\kk$-linear functor $F\colon \cC\to\cD$, we can extend it to
the functor $\Ps(F)\colon\Ps(\cC)\to\Ps(\cD)$ between the envelopes
in the obvious manner.
Moreover, for a $\kk$-linear transformation $\eta\colon F\to G$
we can also define the transformation $\Ps(\eta)\colon\Ps(F)\to\Ps(G)$.
So the operation $\Ps\colon\Cat_\kk\to\PsCat_\kk$ is actually a 2-functor between these 2-categories.
This is the left adjoint of the embedding $\PsCat_\kk\hookrightarrow\Cat_\kk$
in the 2-categorical sense;
that is, if $\cD$ is pseudo-abelian then the restriction of functors
induces a category equivalence
\begin{gather*}
\HOM_\kk(\Ps(\cC),\cD)\eqmap\HOM_\kk(\cC,\cD).
\end{gather*}


We say a pseudo-abelian $\kk$-linear category $\cC$
\term{is generated by} a full subcategory $\cC'\subset\cC$
if every object in $\cC$ is isomorphic to some direct summand
of a direct sum of objects in $\cC'$,
or equivalently, $\Ps(\cC')\simeq\cC$.
When this condition is satisfied we also say
objects in $\cC'$ generate $\cC$.

\subsection{Tensor categories}

A tensor category is a kind of generalization of categories
which have binary ``product'', associative and unital up to isomorphism,
such as the category of vector spaces with tensor product.

\begin{definition}
\label{def:tencat}
\begin{enumerate}
\item A \term{$\kk$-tensor category} is a $\kk$-linear category $\cC$ equipped with
a $\kk$-bilinear functor $\otimes:\cC\times\cC\to\cC$ called the \term{tensor product} and
a functorial isomorphism $\alpha_\cC$
called the \term{associativity constraint}
with components
$\alpha_\cC(U,V,W)\colon(U\otimes V)\otimes W\eqmap U\otimes (V\otimes W)$
such that the diagram below commutes:
\begin{align*}
\xymatrix @C 0pt{
&(U\otimes V)\otimes(W\otimes X) \ar[rd]^-{\quad\alpha_\cC(U,V,W\otimes X)}\\
((U\otimes V)\otimes W)\otimes X \ar[ru]^-{\alpha_\cC(U\otimes V,W,X)\quad} \ar[d]_-{\alpha_\cC(U,V,W)\otimes\id_X}&&
U\otimes(V\otimes(W\otimes X))\rlap{.}\\
(U\otimes(V\otimes W))\otimes X \ar[rr]_-{\alpha_\cC(U,V\otimes W,X)}&&
U\otimes((V\otimes W)\otimes X) \ar[u]_-{\id_U\otimes\alpha_\cC(V,W,X)}
}
\end{align*}

\item A \term{unit object} of a $\kk$-tensor category $\cC$ is an object $\unit_\cC\in\cC$ equipped with
two functorial isomorphisms
$\lambda_\cC(U)\colon \unit_\cC\otimes U\eqmap U$ and
$\rho_\cC(U)\colon U\otimes\unit_\cC\eqmap U$
called the \term{unit constraints}
such that the diagram below commutes:
\begin{align*}
\xymatrix @C 10pt{
(U\otimes\unit_\cC)\otimes V \ar[rr]^-{\alpha_\cC(U,\unit_\cC,V)} \ar[rd]_-{\rho_\cC(U)\otimes\id_V~~}&&
U\otimes(\unit_\cC\otimes V) \ar[ld]^-{~~\id_U\otimes\lambda_\cC(V)}\\&
U\otimes V\rlap{.}
}
\end{align*}
\end{enumerate}
\end{definition}

Since the equality $(U\otimes V)\otimes W=U\otimes(V\otimes W)$ is too strict in category theory,
we need a functorial isomorphism instead.
However, Mac~Lane's coherence theorem~\cite{Mac-Lane:1998} allows us to define the $m$-fold tensor product
$U_1\otimes\dots\otimes U_m$ for multiple objects $U_1,\dots,U_m\in\cC$ since
it does not depend on the order of taking tensor product up to a unique isomorphism.
Similarly for an object $U_1\otimes\dots\otimes U_m$
we can freely insert or remove
tensor products of unit objects.

Remark that a unit object is unique up to a unique isomorphism if exists.
If $\cC$ has a unit object $\unit_\cC$ then $\End_\cC(\unit_\cC)$ is
also a commutative ring
and $\cC$ has two (possibly different) structures of $\End_\cC(\unit_\cC)$-linear category
induced by $\lambda_\cC$ and $\rho_\cC$.

\begin{assumption}
In this paper we do not treat tensor categories without units.
We always assume that each $\kk$-tensor category $\cC$ is endowed with a fixed unit object $\unit_\cC\in\cC$.
In addition, we require that
the unit object $\unit_\cC$ satisfies $\End_\cC(\unit_\cC)\simeq\kk$.
\end{assumption}

In the rest of this paper, we omit writing
the isomorphisms $\alpha_\cC$, $\lambda_\cC$ and $\rho_\cC$ explicitly for
a $\kk$-tensor category $\cC$ since
the reader can complete them easily if needed.

\begin{example}
$\Triv_\kk$ has the unique structure of $\kk$-tensor category.
$\Mod(\kk)$ and $\Rep(\kk)$ are $\kk$-tensor categories with usual tensor products of modules.
More generally, for a bialgebra $A$ over $\kk$, the $\kk$-linear categories
$\Mod(A)$ and $\Rep(A)$ have structures of $\kk$-tensor category.
For $A$-modules $U$ and $V$, the $\kk$-module $U\otimes V$ becomes an $A$-module
via the coproduct of $A$, $\Delta_A\colon A\to A\otimes A$.
The unit object $\unit_A$ is defined to be $\kk$ as a $\kk$-module and
the action of $A$ is the scalar multiplication by the counit of $A$, $\epsilon_A\colon A\to\kk$.
\end{example}

Next we define the corresponding structures on functors and transformations.
Again we need functorial isomorphisms to avoid using equations.
\begin{definition}
\label{def:tenfun}
\begin{enumerate}
\item A $\kk$-tensor functor $F\colon\cC\to\cD$ between $\kk$-tensor categories
is a $\kk$-linear functor equipped with
functorial isomorphisms
$\mu_F(U,V)\colon F(U)\otimes F(V)\eqmap F(U\otimes V)$ and
$\iota_F\colon \unit_\cD\eqmap F(\unit_\cC)$
such that the diagrams below commute:
\begin{align*}
\xymatrix @C 40pt{
F(U)\otimes F(V)\otimes F(W) \ar[r]^-{\id_{F(U)}\otimes\mu_F(V,W)} \ar[d]_-{\mu_F(U,V)\otimes\id_{F(W)}}&
F(U)\otimes F(V\otimes W) \ar[d]^-{\mu_F(U,V\otimes W)}\\
F(U\otimes V)\otimes F(W) \ar[r]_-{\mu_F(U\otimes V,W)}&
F(U\otimes V\otimes W)\rlap{,}\\
}&&
\xymatrix @C 4pt{
F(U) \ar[r]^-{\id_{F(U)}\otimes\iota_F} \ar[d]_-{\iota_F\otimes\id_{F(U)}} \ar@{=}[rd]&
F(U)\otimes F(\unit_\cC) \ar[d]^-{\mu_F(U,\unit_\cC)}\\
F(\unit_\cC)\otimes F(U) \ar[r]_-{\mu_F(\unit_\cC,U)}&
F(U)\rlap{.}
}
\end{align*}
In other words, the isomorphisms $\mu_F$ and $\iota_F$ must be associative and unital.
\item A $\kk$-tensor transformation $\eta\colon F\to G$ between $\kk$-tensor functors
is a $\kk$-linear transformation 
such that the diagrams below commute:
\begin{align*}
\xymatrix{
F(U)\otimes F(V) \ar[r]^-{\mu_F(U,V)} \ar[d]_-{\eta(U)\otimes\eta(V)}&
F(U\otimes V) \ar[d]^-{\eta(U\otimes V)}\\
G(U)\otimes G(V) \ar[r]_-{\mu_G(U,V)}&
G(U\otimes V)\rlap{,}
}&&
\xymatrix{
\unit_\cD \ar[r]^-{\iota_F} \ar@{=}[d]&
F(\unit_\cC) \ar[d]^-{\eta(\unit_\cC)}\\
\unit_\cD \ar[r]_-{\iota_G}&
G(\unit_\cC)\rlap{.}
}
\end{align*}
In other words, a $\kk$-tensor transformation $\eta$ must satisfy that
$\eta(U\otimes V)=\eta(U)\otimes\eta(V)$ and $\eta(\unit_\cC)=\id_{\unit_\cD}$.
\end{enumerate}
\end{definition}

Beware that the category $\HOM_\kk^\otimes(\cC,\cD)$ consisting of
$\kk$-tensor functors and transformations is no longer $\kk$-linear.

\subsection{Braided tensor categories}

A braided tensor category is a tensor category equipped with
a functorial isomorphism called \term{braiding},
which allows us to swap two objects in a tensor product $U\otimes V$.

\begin{definition}
\label{def:braid}
\begin{enumerate}
\item A \term{braiding} (also called a \term{commutativity constraint})
on a $\kk$-tensor category $\cC$ is a functorial isomorphism
$\sigma_\cC(U,V)\colon U\otimes V\eqmap V\otimes U$
such that the diagrams below commute:
\begin{align*}
\xymatrix @C 0pt{
U\otimes V\otimes W \ar[rr]^-{\sigma_\cC(U,V\otimes W)} \ar[rd]_-{\sigma_\cC(U,V)\otimes\id_W\quad}&&
V\otimes W\otimes U\rlap{,} \\&
V\otimes U\otimes W \ar[ru]_-{\quad\id_V\otimes\sigma_\cC(U,W)}
}&&
\xymatrix @C 0pt{
&U\otimes W\otimes V \ar[rd]^-{\quad\sigma_\cC(U,W)\otimes\id_V}\\
U\otimes V\otimes W \ar[rr]_-{\sigma_\cC(U\otimes V, W)} \ar[ru]^-{\id_U\otimes\sigma_\cC(V,W)\quad}&&
W\otimes U\otimes V\rlap{.}
}
\end{align*}
The \term{inverse} of the braiding $\sigma_\cC$ is defined by $\sigma_\cC^{-1}(V,W)\coloneqq\sigma_\cC(W,V)^{-1}$.
A braiding $\sigma_\cC$ is called \term{symmetric} if $\sigma_\cC=\sigma_\cC^{-1}$.
\item A $\kk$-tensor category $\cC$ equipped with a braiding $\sigma_\cC$
is called a \term{$\kk$-braided tensor category}.
If the braiding is symmetric, we call it a \term{$\kk$-symmetric tensor category}.
\item A \term{$\kk$-braided tensor functor} $F\colon\cC\to\cD$ between $\kk$-braided tensor categories
is a $\kk$-tensor functor such that the diagram below commutes:
\begin{gather*}
\xymatrix{
F(U)\otimes F(V) \ar[r]^-{\mu_F(U,V)} \ar[d]_-{\sigma_\cD(F(U),F(V))}&
F(U\otimes V) \ar[d]^-{F(\sigma_\cC(U,V))}\\
F(V)\otimes F(U) \ar[r]_-{\mu_F(V,U)}&
F(V\otimes U)\rlap{.}
}
\end{gather*}
\item A \term{$\kk$-braided tensor transformation} is just a $\kk$-tensor transformation
between two $\kk$-braided tensor functors.
\end{enumerate}
\end{definition}

The axiom says that
the braiding $\sigma_\cC(U_1\otimes\dots\otimes U_m,V_1\otimes\dots\otimes V_n)$
between tensor products
is determined by $\sigma_\cC(U_i,V_j)$ at each terms $U_i$ and $V_j$.
It also indicates that
for each $g\in\fB_m$, where $\fB_m$ is the braid group of order $m$,
there is a well-defined functorial isomorphism
\begin{align*}
\sigma_\cC^g(U_1,\dots,U_m)\colon
U_1\otimes\dots\otimes U_m\eqmap U_{g^{-1}(1)}\otimes\dots\otimes U_{g^{-1}(m)}
\end{align*}
which permutes the terms of tensor products along $g$ using the braiding $\sigma_\cC$.
When the braiding is symmetric then
$\sigma_\cC^g$ is well-defined for $g\in\fS_m$,
an element of the symmetric group.

\begin{example}
If $A$ is a cocommutative bialgebra then
the transposition map
$U\otimes V\eqmap V\otimes U;u\otimes v\mapsto v\otimes u$
for $U,V\in\Mod(A)$ is an $A$-homomorphism.
Thus this functorial isomorphism defines a structure of $\kk$-symmetric tensor category
on $\Mod(A)$.
On the other hand, the quantum enveloping algebra $U_q(\mathfrak{g})$ over $\kk=\CC(q)$
is not cocommutative, but the category of finite dimensional $\mathfrak{h}$-semisimple
$U_q(\mathfrak{g})$-modules has a non-symmetric braiding
introduced by an $R$-matrix.
\end{example}

\section{Representation Category of Wreath Product}

Let $d\in\NN$.
For each $\kk$-algebra $A$, we can construct a new algebra
$A\wr\fS_d$ called the \term{wreath product} of $A$ of rank $d$
following the two steps below:
\begin{gather*}
A\longmapsto A^{\otimes d}\longmapsto A\wr\fS_d.
\end{gather*}
\begin{enumerate}
\item Create the $d$-fold tensor product algebra
$A^{\otimes d}=A\otimes\dots\otimes A$ from the base algebra $A$.
Then the symmetric group $\fS_d$ of rank $d$ naturally acts on
$A^{\otimes d}$ by permutation of terms.
\item Create the semidirect product algebra $A\wr\fS_d=A^{\otimes d}\rtimes\fS_d$
by twisting the product via the action $\fS_d\curvearrowright A^{\otimes d}$.
\end{enumerate}
For these three algebras
we have corresponding representation categories
\begin{gather*}
\Rep(A)\longmapsto\Rep(A^{\otimes d})\longmapsto\Rep(A\wr\fS_d).
\end{gather*}
One of the remarkable facts is, under suitable conditions, that
we can proceed these steps using the categorical language only
and create these representation categories
without the information about the base algebra $A$ itself.
This operation can be applied to an arbitrary $\kk$-linear category $\cC$
which is not of the form of representation category of algebra.
The procedure for this construction is as follows:
\begin{gather*}
\cC\longmapsto\cC^{\barboxtimes d}\longmapsto(\cC^{\barboxtimes d})^{\fS_d}.
\end{gather*}
\begin{enumerate}
\item Create the $d$-fold tensor product category
$\cC^{\barboxtimes d}=\cC\barboxtimes\dots\barboxtimes\cC$ from the base category $\cC$.
Then the symmetric group $\fS_d$ naturally acts on it.
\item Take the category $(\cC^{\barboxtimes d})^{\fS_d}$ of $\fS_d$-invariants
in $\fS_d\curvearrowright\cC^{\barboxtimes d}$.
\end{enumerate}
We denote the result above by $\cW_d(\cC)\coloneqq(\cC^{\barboxtimes d})^{\fS_d}$.
In this section we see how this process works.
Actually the category $\cS_t(\cC)$ for $t\in\kk$, which is our main product in this paper,
interpolates the family of categories $\cW_d(\cC)$ for $d\in\NN$.

\subsection{Tensor product of Categories}
First we study the tensor product of $\kk$-linear categories.
Recall that if $A$ and $B$ are both $\kk$-algebras then so is $A\otimes B$ naturally.
We see that tensor product of algebras in representation theory
corresponds to that of categories in category theory.

\begin{definition}
\label{def:catten}
Let $\cC,\cD$ be $\kk$-linear categories.
Their \term{tensor product} $\cC\boxtimes\cD$ is the $\kk$-linear category defined as follows:
\begin{description}
\item[Object] a symbol $U\boxtimes V$ for a pair of objects $U\in\cC$ and $V\in\cD$.
\item[Morphism] $\Hom_{\cC\boxtimes\cD}(U\boxtimes V,U'\boxtimes V')\coloneqq\Hom_\cC(U,U')\otimes\Hom_\cD(V,V')$
and composition of morphisms is diagonal.
We denote a morphism by $f\boxtimes g$
instead of $f\otimes g$ for $f\in\Hom_\cC(U,U')$ and $g\in\Hom_\cD(V,V')$.
\end{description}
This operation naturally defines a 2-bifunctor $\boxtimes\colon\Cat_\kk\times\Cat_\kk\to\Cat_\kk$.
For $\kk$-linear functors $F\colon\cC\to\cC'$ and $G\colon\cD\to\cD'$,
the $\kk$-linear functor $F\boxtimes G\colon\cC\boxtimes\cD\to\cC'\boxtimes\cD'$
acts on objects and morphisms diagonally.
For $\kk$-linear transformations $\eta\colon F\to F'$ and $\kappa\colon G\to G'$,
the $\kk$-linear transformation
$\eta\boxtimes\kappa\colon F\boxtimes G\to F'\boxtimes G'$
is defined by
\begin{gather*}
(\eta\boxtimes\kappa)(U\boxtimes V)\coloneqq\eta(U)\boxtimes\kappa(V)
\colon F(U)\boxtimes G(V)\to F'(U)\boxtimes G'(V)
\end{gather*}
at each $U\in\cC$ and $V\in\cD$.
\end{definition}

The product $\boxtimes$ is associative and commutative up to equivalence,
so we can write $\cC_1\boxtimes\dots\boxtimes\cC_d$ without any confusions.
If all terms are equal to $\cC$, we denote it by
$\cC^{\boxtimes d}\coloneqq\cC\boxtimes\dots\boxtimes\cC$.
It is convenient to set $\cC^{\boxtimes 0}\coloneqq\Triv_\kk$, the unit with respect to $\boxtimes$.
The operation $\cC\mapsto\cC^{\boxtimes d}$ also defines a 2-functor
$\Cat_\kk\to\Cat_\kk$.

One of the purpose of considering the tensor product of categories is
to create a universal object related to $\kk$-bilinear functors:
the category of $\kk$-bilinear functors $\cC\times\cD\to\cE$
is equivalent to the category of $\kk$-linear functors $\cC\boxtimes\cD\to\cE$.
It is equivalent to say that the natural functor
\begin{gather*}
\HOM_\kk(\cC\boxtimes\cD,\cE)\eqmap
\HOM_\kk(\cC,\HOM_\kk(\cD,\cE))
\end{gather*}
is a category equivalence (recall that the category $\HOM_\kk(\cD,\cE)$ is again $\kk$-linear).
For pseudo-abelian categories, it is natural to define
the tensor product by $\cC\barboxtimes\cD\coloneqq\Ps(\cC\boxtimes\cD)$.
It satisfies the same universality as above in the 2-category $\PsCat_\kk$.
The unit for $\barboxtimes$ is $\Ps(\Triv_\kk)\simeq\Rep(\kk)$.



Now let us pay attention to its representation-theoretic properties listed in the next proposition.
Recall that for a $\kk$-algebra $A$,
$\Mod(A)$ is the category of all $A$-modules
and $\Rep(A)$ is the category of $A$-modules which are finitely generated and projective over $\kk$.

\begin{proposition}
\label{prop:mimic_tensor}
Let $A$ and $B$ be $\kk$-algebras.
\begin{enumerate}
\item There is a canonical functor
$\Mod(A)\barboxtimes\Mod(B)\to\Mod(A\otimes B)$ which sends an object $U\boxtimes V$ to the $(A\otimes B)$-module $U\otimes V$
on which $A\otimes B$ acts diagonally.
\item If $\Rep(A)$ is hom-finite and projective,
the restriction $\Rep(A)\barboxtimes\Rep(B)\to\Rep(A\otimes B)$
of this functor is fully faithful.
\item Suppose that $\kk$ is a field. If $A$ and $B$ are separable $\kk$-algebras,
the restricted functor above gives a category equivalence.
\end{enumerate}
\end{proposition}
\begin{proof}
(1) Obvious.

(2) Let $U,U'\in\Rep(A)$ and $V,V'\in\Rep(B)$.
By the assumptions $V'$ and $\Hom_A(U,U')$ are finitely generated and projective over $\kk$,
thus we get
\begin{align*}
\Hom_{A\otimes B}(U\otimes V,U'\otimes V')
&\simeq\Hom_B(V,\Hom_A(U,U'\otimes V'))\\
&\simeq\Hom_B(V,\Hom_A(U,U')\otimes V')\\
&\simeq\Hom_A(U,U')\otimes\Hom_B(V,V').
\end{align*}

(3) Since the functor is fully faithful by (2),
it suffices to prove that the functor is essentially surjective.
For a separable $\kk$-algebra $C$, let $I(C)$ be the set of all
finite dimensional irreducible $C$-modules up to isomorphism.
Since $A\otimes B$ is also separable,
it suffices to show that the image of the functor contains $I(A\otimes B)$.
If $\kk$ is algebraically closed
the statement follows from the well known fact
\begin{align*}
I(A\otimes B)=\{U\otimes V\;|\;U\in I(A),V\in I(B)\}.
\end{align*}

For a general field $\kk$, let $\overline{\kk}$ be the algebraic closure of $\kk$
and let us denote a field extension ${\bullet}\otimes\overline{\kk}$ by $\overline{\bullet}$.
We use the next fact to prove the statement.
The proof is easy and we omit it.
\begin{lemma}
Let $C$ be a separable $\kk$-algebra.
Then for each $L\in I(\overline{C})$ there exists unique $L'\in I(C)$
such that $L$ appears in the irreducible components of $\overline{L'}$.
\end{lemma}
By the lemma for $A$ and $B$ we get that
each object in $ I(\overline{A\otimes B})$ is a direct summand of
$\overline{U\otimes V}$
for some $U\in I(A)$ and $V\in I(B)$.
Using the lemma for $A\otimes B$ again,
we conclude the statement.
\end{proof}

We interpret these results as follows.
Using the data of representation categories of $A$ and $B$
we can imitate that of $A\otimes B$ to some extent,
even if we do not know about the base algebras $A$ and $B$ themselves.
So we regard $\Rep(A)\barboxtimes\Rep(B)$
as a replica of $\Rep(A\otimes B)$ for any $A$ and $B$.

\subsection{Group action on Category}
Suppose that a group $G$ acts on a $\kk$-algebra $A$
by $\kk$-linear automorphisms of algebra.
For the consistency of notations we denote the action
of $g\in G$ by conjugation $a\in A\mapsto gag^{-1}\in A$.
Then for each $g\in G$ and an $A$-module $U$,
we can define the twisted $A$-module
\begin{gather*}
g\cdot U\coloneqq\{\text{symbol } g\cdot u\;|\;u\in U\}
\end{gather*}
whose $A$-action is defined by $a(g\cdot u)\coloneqq g\cdot(g^{-1}ag)u$.
This defines a \term{$G$-action} on the $\kk$-linear category $\Mod(A)$
described below.

\begin{definition}
Let $G$ be a group and $\cC$ a $\kk$-linear category.
An \term{action} $\cM\colon G\curvearrowright\cC$ is a collection of
$\kk$-linear endofunctors $\cM_g\colon U\mapsto g\cdot U$ on $\cC$ for all $g\in G$
equipped with functorial isomorphisms
$\mu_\cM^{g,h}(U)\colon g\cdot(h\cdot U)\simeq gh\cdot U$ for each $g,h\in G$ and
$\iota_\cM(U)\colon U\simeq 1\cdot U$ for the unit element $1\in G$
such that the diagrams below commute:
\begin{align*}
\xymatrix{
g\cdot(h\cdot(k\cdot U)) \ar[r]^-{g\cdot\mu_\cM^{h,k}(U)} \ar[d]_-{\mu_\cM^{g,h}(k\cdot U)}&
g\cdot(hk\cdot U) \ar[d]^-{\mu_\cM^{g,hk}(U)}\\
gh\cdot(k\cdot U) \ar[r]_-{\mu_\cM^{gh,k}(U)}&
ghk\cdot U\rlap{,}\\
}&&
\xymatrix{
g\cdot U \ar[r]^-{g\cdot\iota_\cM(U)} \ar[d]_-{\iota_\cM(g\cdot U)} \ar@{=}[rd]&
g\cdot(1\cdot U) \ar[d]^-{\mu_\cM^{g,1}(U)}\\
1\cdot(g\cdot U) \ar[r]_-{\mu_\cM^{1,g}(U)}&
g\cdot U\rlap{.}
}
\end{align*}
\end{definition}

For example, on any $\kk$-linear category $\cC$
we can define the \term{trivial action} of $G$ by $\cM_g\coloneqq\Id_\cC$.
If groups $G$ and $H$ act on $\kk$-linear categories $\cC$ and $\cD$ respectively,
$G^\op\times H$ and $G\times H$ naturally act on
$\HOM_\kk(\cC,\cD)$ and $\cC\boxtimes\cD$ respectively.

\begin{definition}
Let $G$ be a group and $\cC$ be a $\kk$-linear category
on which $G$ acts.
\begin{enumerate}
\item A \term{$G$-invariant object} $U\in\cC$ is an object equipped with
a collection of isomorphisms $\kappa_U^g\colon g\cdot U\simeq U$ for all $g\in G$
such that the diagrams below commute:
\begin{align*}
\xymatrix{
g\cdot(h\cdot U) \ar[r]^-{g\cdot\kappa_U^h} \ar[d]_-{\mu_\cM^{g,h}(U)}&
g\cdot U \ar[r]^-{\kappa_U^g}&
U \ar@{=}[d]\\
gh\cdot U \ar[rr]_-{\kappa_U^{gh}}&&
U\rlap{,}\\
}&&
\xymatrix{
U \ar[d]_-{\iota_\cM(U)} \ar@{=}[rd]\\
1\cdot U \ar[r]_-{\kappa_U^{1}}&
U\rlap{.}
}
\end{align*}
\item A \term{$G$-invariant morphism} $\varphi\colon U\to V$ between $G$-invariant objects
is a morphism such that the diagram below commutes:
\begin{gather*}
\xymatrix{
U \ar[r]^-{\varphi} \ar[d]_-{\kappa_V^g}&
V \ar[d]^-{\kappa_U^g}\\
g\cdot U \ar[r]_-{g\cdot\varphi}&
g\cdot V\rlap{.}
}
\end{gather*}
\item We denote by $\cC^G$ the $\kk$-linear category consisting of $G$-invariant objects and morphisms.
\end{enumerate}
\end{definition}


\begin{remark}
Although we do not use it explicitly in this paper,
one can easily define the 2-category $G\text{-}\Cat_\kk$
consisting of $\kk$-linear categories with $G$-actions
along with \term{$G$-equivalent functors} and \term{$G$-equivalent transformations}.
Taking invariants $\cC\mapsto\cC^G$ is a 2-functor
$G\text{-}\Cat_\kk\to\Cat_\kk$ and this is the right adjoint of
the 2-functor which attaches
the trivial $G$-action to a given category.
\end{remark}

Now let $G$ be a group acts on a $\kk$-algebra $A$.
Recall that the semidirect product $A\rtimes G$ of $A$ and $G$
is a $\kk$-algebra which is isomorphic to $A\otimes\kk[G]$ as $\kk$-module
and its product is defined by $(a\otimes g)(b\otimes h)\coloneqq a(gbg^{-1})\otimes gh$
for $a,b\in A$ and $g,h\in G$.
We see here that making the semidirect product of an algebra
is exactly taking the invariants of a category.

\begin{proposition}
\label{prop:mimic_semi}
For $G$ and $A$ as above,
there are equivalences $\Mod(A)^G\eqmap\Mod(A\rtimes G)$
and $\Rep(A)^G\eqmap\Rep(A\rtimes G)$.
\end{proposition}
\begin{proof}
For each $G$-invariant $A$-module $U$,
using isomorphisms $\kappa_U^g\colon g\cdot U\simeq U$,
we can define a $A\rtimes G$ action on it by
$(a\otimes g)u\coloneqq a\:\kappa_U^g(g\cdot u)$.
On the other hand, for each ($A\rtimes G$)-module $U$,
there are natural $A$-module isomorphisms $g\cdot U\simeq U;g\cdot u\mapsto(1\otimes g)u$.
It is easy to check that they are well-defined and two functors above are inverse to each other.
\end{proof}

Now suppose that a group $G$ acts on a $\kk$-linear category $\cC$.
To create $G$-invariant objects in $\cC$, we can use
the technique of \term{restriction} and \term{induction} as we do
for ordinary representations of groups.
\begin{definition}
Let $H\subset G$ be a group and its subgroup and
$\cC$ a $\kk$-linear category on which $G$ acts.
We denote by $\Res^G_H\colon\cC^G\to\cC^H$
the obvious forgetful functor and call it the \term{restriction functor}.
If it has the left adjoint, we denote it by $\Ind^G_H\colon\cC^H\to\cC^G$
and call it the \term{induction functor}.
\end{definition}
\begin{proposition}
\label{prop:ind_functor}
Let $G,H$ and $\cC$ be as above.
If $\#(G/H)<\infty$ and $\cC$ is additive,
then the induction functor exists.
In this case, $\Ind^G_H$ is also the right adjoint of $\Res^G_H$.
\end{proposition}
\begin{proof}
First let us choose representatives of the left coset $G/H$, namely
$G/H=\{g_1,\dots,g_l\}$.
For $U\in\cC^H$ we define an object $\Ind^G_H(U)\in\cC$ by
\begin{gather*}
\Ind^G_H(U)\coloneqq\bigoplus_{i=1,\dots,l}g_i\cdot U.
\end{gather*}
Take any $h\in G$.
For each $i\in\{1,\dots,l\}$, there exist unique $h^{(i)}\in H$ and $i'\in\{1,\dots,l\}$
such that $hg_i=g_{i'}h^{(i)}$.
Thus there is an isomorphism
\begin{gather*}
h\cdot\Ind^G_H(U)
\simeq \bigoplus_{i=1,\dots,l}h\cdot(g_i\cdot U)
\simeq \bigoplus_{i=1,\dots,l}g_{i'}\cdot(h^{(i)}\cdot U)
\simeq \bigoplus_{i=1,\dots,l}g_{i'}\cdot U
\simeq \Ind^G_H(U).
\end{gather*}
These isomorphisms define a structure of $G$-invariant object on $\Ind^G_H(U)$.
It is easy to check that this construction is functorial and gives the left adjoint of $\Res^G_H$.
The last statement follows from considering the opposite category.
\end{proof}
\begin{corollary}
\label{cor:ind_generator}
Let $\cC$ be a $\kk$-linear category on which a group $G$ acts.
Suppose that $\#G<\infty$ and $\#G\in\kk$ is invertible.
Then all objects of the form $\Ind^G_{\{1\}}(U)$ for $U\in\cC$
generate a pseudo-abelian category $\Ps(\cC)^G$.
\end{corollary}
\begin{proof}
Take an arbitrary $U\in\Ps(\cC)^G$. There are morphisms in $\Ps(\cC)^G$
\begin{gather*}
U\xrightarrow{i}\Ind^G_{\{1\}}\Res^G_{\{1\}}(U)\xrightarrow{p} U
\end{gather*}
induced by $g\cdot U\simeq g$ for all $g\in G$.
Since $p\circ i=(\#G)\id_U$, the idempotent $(\#G)^{-1}i\circ p$
has its image in $\Ind^G_{\{1\}}\Res^G_{\{1\}}(U)$ isomorphic to $U$.
\end{proof}

\subsection{Wreath Product of Algebra}
Now we consider the main topic of this section,
representation categories of wreath products.

\begin{notation}
Let $X$ be a finite set.
We denote by $P(X)$ the set of
all equivalence relations on $X$,
and for $p\in P(X)$ we write $x\sim_p y$
if $x$ and $y$ are equivalent with respect to $p$.
There is a natural bijection from the set of partitions of $X$ to $P(X)$:
\begin{align*}
X=X_1\sqcup\dots\sqcup X_l
\quad\bijection\quad
x\sim_p y\iff\text{$x,y$ are in the same $X_i$.}
\end{align*}
So we call $p\in P(X)$ a \term{partition} and represent by
$p=\{X_1,\dots,X_l\}$
that each $X_i$ is an equivalence class of $X$ by $p$.
We denote by $\#p$ the number of its equivalence classes
and call it the \term{length} of $p$.

$P(X)$ is partially ordered with respect to strength of relations.
For two partitions $p,q\in P(X)$ we write $p\leq q$ if $x\sim_q y$ implies $x\sim_p y$.
We also say that the partition $q$ is a \term{refinement} of $p$ when $p\leq q$.
The \term{common refinement} $p\wedge q\in P(X)$ of
two partitions $p,q\in P(X)$ is defined by
\begin{align*}
x\sim_{p\wedge q}y \iff x\sim_p y~~\text{and}~~x\sim_q y.
\end{align*}
Beware that it is the \term{least upper bound} of $p$ and $q$,
not the \term{greatest lower bound} in the language of partially ordered set.

We denote by $\fS_X$ the group of all bijections from $X$ to $X$ itself
and call the \term{symmetric group} on $X$.
For $p=\{X_1,\dots,X_l\}\in P(X)$, we define the subgroup $\fS_p\subset\fS_X$ by
\begin{align*}
\fS_p&\coloneqq\{g\in\fS_X\;|\;x\sim_p g(x)\text{ for all } x\in X\}\\
&\simeq\fS_{X_1}\times\dots\times\fS_{X_l}.
\end{align*}
It is called a \term{Young subgroup} of $\fS_X$.
$\fS_X$ also acts on $P(X)$ as follows:
for $g\in\fS_X$ and $p\in P(X)$, $g(p)\in P(X)$ is a partition such that
\begin{align*}
x\sim_{g(p)}y \iff g^{-1}(x)\sim_p g^{-1}(y).
\end{align*}
If $d\in\NN$ and $X=\{1,\dots,d\}$, we simply denote $P(X)$ and $\fS_X$
by $P(d)$ and $\fS_d$ respectively.
\end{notation}

\begin{definition}
For a $\kk$-algebra $A$ and $d\in\NN$,
the \term{wreath product} $A\wr\fS_d$ of $A$ by $\fS_d$
is the semidirect product $A^{\otimes d}\rtimes\fS_d$
where the symmetric group $\fS_d$ acts on the $d$-fold tensor product $A^{\otimes d}$
by permutation of terms.
More explicitly, $A\wr\fS_d$ is the $\kk$-algebra
which is isomorphic to $A^{\otimes d}\otimes\kk[\fS_d]$ as $\kk$-module
and its product is defined by
\begin{align*}
(a_1\otimes\dots\otimes a_d\otimes g)(b_1\otimes\dots\otimes b_d\otimes h)
=(a_1b_{g^{-1}(1)})\otimes\dots\otimes(a_db_{g^{-1}(d)})\otimes g h
\end{align*}
for $a_1,\dots,a_d,b_1,\dots,b_d\in A$ and $g,h\in\fS_d$.
For $p\in P(d)$, let $A\wr\fS_p\coloneqq A^{\otimes d}\rtimes\fS_p$.
Obviously it is a $\kk$-subalgebra of $A\wr\fS_d$.
\end{definition}

Let us create representation categories of wreath products of algebras
in the language of categories.
We already know what should it be by the preceding arguments.
\begin{definition}
Let $d\in\NN$ and $\cC$ be a $\kk$-linear category.
We denote by $\cW_d(\cC)\coloneqq(\cC^{\barboxtimes d})^{\fS_d}$
the category of $\fS_d$-invariants in
the $d$-fold tensor product category $\cC^{\barboxtimes d}$
where the symmetric group $\fS_d$ acts on it by permutation of terms.
This induces a 2-functor $\cW_d\colon\Cat_\kk\to\PsCat_\kk$.
\end{definition}
Note that the $\fS_d$-action on $\Mod(A^{\otimes d})$ induced by
$\fS_d\curvearrowright A^{\otimes d}$ coincides with that we used
in the definition above.
Combining Propositions~\ref{prop:mimic_tensor} and \ref{prop:mimic_semi},
we obtain the next results.
\begin{proposition}
\label{prop:mimic_wr}
Let $A$ be a $\kk$-algebra.
\begin{enumerate}
\item There is a canonical functor $\cW_d(\Mod(A))\to\Mod(A\wr\fS_d)$.
\item If $\Rep(A)$ is hom-finite and projective, then
the restriction $\cW_d(\Rep(A))\to\Rep(A\wr\fS_d)$
is fully faithful.
\item Suppose that $\kk$ is a field.
If $A$ is a separable $\kk$-algebra,
the restricted functor above gives a category equivalence.
\end{enumerate}
\end{proposition}

It is not hard to check that when $\cC$ is a $\kk$-tensor category
our category $\cW_d(\cC)$ also has a canonical structure of
$\kk$-tensor category induced from that of $\cC$.
We have an enriched 2-functor $\cW_d\colon\tCat_\kk\to\tPsCat_\kk$
where $\tCat_\kk$ is the 2-category of $\kk$-tensor categories,
functors and transformations,
and $\tPsCat_\kk$ is its full sub-2-category consisting of pseudo-abelian ones.
On the other hand, if $A$ is a $\kk$-bialgebra
then the coproduct $\Delta_A$ and the counit $\epsilon_A$ of $A$
will be lifted to those of $A\wr\fS_d$:
for $a_1,\dots,a_d\in A$ and $g\in\fS_d$,
\begin{align*}
\Delta_{A_d}(a_1\otimes\dots\otimes a_d\otimes g) &= \sum (a_1^{(1)}\otimes\dots\otimes a_d^{(1)}\otimes g)\otimes(a_1^{(2)}\otimes\dots\otimes a_d^{(2)}\otimes g),\\
\epsilon_{A_d}(a_1\otimes\dots\otimes a_d\otimes g) &= \epsilon_A(a_1)\dotsm\epsilon_A(a_d)
\end{align*}
so $A\wr\fS_d$ is also a $\kk$-bialgebra.
Here we use the Sweedler notation $\Delta_A(a)=\sum a^{(1)}\otimes a^{(2)}$ to write coproducts.
These structures are of course compatible and
$\cW_d(\Mod(A))\to\Mod(A\wr\fS_d)$ induces a $\kk$-tensor functor.
The same holds for $\kk$-braided tensor categories.

\subsection{Induced objects from Young subgroups}
\label{sec:replica}
For an object $U\in\cC$,
its $d$-fold tensor product $U^{\boxtimes d}\in\cC^{\boxtimes d}$ is clearly $\fS_d$-invariant.
More generally,
let $p\in P(d)$ and take $U_1,\dots,U_d\in\cC$ such that $U_i=U_j$ whenever $i\sim_p j$.
Then the object $U_1\boxtimes\dots\boxtimes U_d$ is $\fS_p$-invariant
and we can induce this object to the $\fS_d$-invariant object
\begin{gather*}
\Ind_p(U_1\boxtimes\dots\boxtimes U_d)\coloneqq
\Ind_{\fS_p}^{\fS_d}(U_1\boxtimes\dots\boxtimes U_d)\in\cW_d(\cC).
\end{gather*}
In this subsection we study the pseudo-abelian full subcategory $\pW_d(\cC)$ of $\cW_d(\cC)$
generated by objects of this form.
That is, an object in $\pW_d(\cC)$ is a direct summand of a direct sum of objects
$\Ind_p(U_1\boxtimes\dots\boxtimes U_d)$.
Note that if $\#\fS_d=d!$ is invertible in $\kk$, $\pW_d(\cC)$ coincides with
the whole category $\cW_d(\cC)$ by Corollary~\ref{cor:ind_generator}.

By its definition in the proof of Proposition~\ref{prop:ind_functor},
\begin{gather*}
\Ind_p(U_1\boxtimes\dots\boxtimes U_d)\simeq\!\!
\bigoplus_{g\in\fS_d/\fS_p}\!\!U_{g^{-1}(1)}\boxtimes\dots\boxtimes U_{g^{-1}(d)}
\end{gather*}
as object in $\cC^{\barboxtimes d}$, so
\begin{multline*}
\Hom_{\cC^{\barboxtimes d}}
(\Ind_p(U_1\boxtimes\dots\boxtimes U_d),\Ind_q(V_1\boxtimes\dots\boxtimes V_d))\\
\simeq\;\bigoplus_{\substack{g\in\fS_d/\fS_p\\h\in\fS_d/\fS_q}}\mspace{-10mu}
\Hom_\cC(U_{g^{-1}(1)},V_{h^{-1}(1)})\otimes\dots\otimes\Hom_\cC(U_{g^{-1}(d)},V_{h^{-1}(d)}).
\end{multline*}
The symmetric group $\fS_d$ acts on the space of $\cC^{\barboxtimes d}$-morphisms above
by permutation
and $\cW_d(\cC)$-morphisms are exactly the fixed points of this action.
To describe them more precisely,
we first study the diagonal action $\fS_d\curvearrowright\fS_d/\fS_p\times\fS_d/\fS_q$.
It is clear that the map
\begin{align*}
\fS_d/\fS_p\times\fS_d/\fS_q&\to\fS_p\backslash\fS_d/\fS_q\\
(g,h)&\mapsto g^{-1}h
\end{align*}
induces a bijection $\fS_d\backslash(\fS_d/\fS_p\times\fS_d/\fS_q)\bijectivemap\fS_p\backslash\fS_d/\fS_q$,
and the stabilizer subgroup of each $(g,h)\in\fS_d/\fS_p\times\fS_d/\fS_q$ is
\begin{align*}
g\fS_p g^{-1}\cap h\fS_q h^{-1}
=\fS_{g(p)}\cap\fS_{h(q)}
=\fS_{g(p)\wedge h(q)}.
\end{align*}
Thus the orbit decomposition gives a bijection
\begin{align*}
\bigsqcup_{k\in\fS_p\backslash\fS_d/\fS_q}\!\!\!\!\!\!
\fS_d/\fS_{p\wedge k(q)}&\bijectivemap\fS_d/\fS_p\times\fS_d/\fS_q\\
(k;g)&\longmapsto(g,gk).
\end{align*}
Here the notation $k\in\fS_p\backslash\fS_d/\fS_q$ means that
$k$ runs over the representatives of the $\fS_p$-orbits of $\fS_d/\fS_q$.
If we choose another representative $xk\in\fS_d/\fS_q$
for $x\in\fS_p$, there are canonical isomorphisms
\begin{align*}
\fS_{p\wedge k(q)}&\eqmap\fS_{p\wedge xk(q)}&
\fS_d/\fS_{p\wedge k(q)}&\eqmap\fS_d/\fS_{p\wedge xk(q)}\\
y&\mapsto xyx^{-1},&
g&\mapsto gx^{-1},
\end{align*}
so the bijection above is well-defined.
This gives us an isomorphism
\begin{align*}
&\mspace{24mu}\Hom_{\cW_d(\cC)}
(\Ind_p(U_1\boxtimes\dots\boxtimes U_d),\Ind_q(V_1\boxtimes\dots\boxtimes V_d))\\
&\simeq\Biggl(\;\bigoplus_{\substack{g\in\fS_d/\fS_p\\h\in\fS_d/\fS_q}}\!\!\!
\Hom_\cC(U_{g^{-1}(1)},V_{h^{-1}(1)})\otimes\dots\otimes\Hom_\cC(U_{g^{-1}(d)},V_{h^{-1}(d)})\Biggr)^{\fS_d}\\
&\simeq\!\!\bigoplus_{k\in\fS_p\backslash\fS_d/\fS_q}\!\Biggl(\;\bigoplus_{g\in\fS_d/\fS_{p\wedge k(q)}}\!\!\!\!\!\!
\Hom_\cC(U_{g^{-1}(1)},V_{(gk)^{-1}(1)})\otimes\dots\otimes\Hom_\cC(U_{g^{-1}(d)},V_{(gk)^{-1}(d)})\Biggr)^{\fS_d}\\
&\simeq\!\!\bigoplus_{k\in\fS_p\backslash\fS_d/\fS_q}\!\!\!\!
(\Hom_\cC(U_1,V_{k^{-1}(1)})\otimes\dots\otimes\Hom_\cC(U_d,V_{k^{-1}(d)}))^{\fS_{p\wedge k(q)}}.\label{eq:hom}
\end{align*}
Here, for each direct summand
\begin{align*}
(\Hom_\cC(U_1,V_{k^{-1}(1)})\otimes\dots\otimes\Hom_\cC(U_d,V_{k^{-1}(d)}))^{\fS_{p\wedge k(q)}}
\end{align*}
in the right-hand side, its
embedding is induced from
\begin{align*}
\varphi_1\otimes\dots\otimes\varphi_d
\mapsto\!\!\sum_{g\in\fS_d/\fS_{p\wedge k(q)}}\!\!\!\!
\varphi_{g^{-1}(1)}\boxtimes\dots\boxtimes\varphi_{g^{-1}(d)}.
\end{align*}

If $\cC$ is a $\kk$-tensor category,
we can calculate tensor product of objects
in the same manner.
In the $\kk$-tensor category $\cC^{\barboxtimes d}$,
\begin{align*}
&\mspace{24mu}\Ind_p(U_1\boxtimes\dots\boxtimes U_d)\otimes\Ind_q(V_1\boxtimes\dots\boxtimes V_d)\\
&\simeq\bigoplus_{\substack{g\in\fS_d/\fS_p\\h\in\fS_d/\fS_q}}\!\!
(U_{g^{-1}(1)}\boxtimes\dots\boxtimes U_{g^{-1}(d)})\otimes(V_{h^{-1}(1)}\boxtimes\dots\boxtimes V_{h^{-1}(d)})\\
&\simeq\bigoplus_{\substack{g\in\fS_d/\fS_p\\h\in\fS_d/\fS_q}}\!\!
(U_{g^{-1}(1)}\otimes V_{h^{-1}(1)})\boxtimes\dots\boxtimes (U_{g^{-1}(d)}\otimes V_{h^{-1}(d)})\\
&\simeq\bigoplus_{k\in\fS_p\backslash\fS_d/\fS_q}\bigoplus_{g\in\fS_{p\wedge k(q)}}\!\!
(U_{g^{-1}(1)}\otimes V_{(gk)^{-1}(1)})\boxtimes\dots\boxtimes (U_{g^{-1}(d)}\otimes V_{(gk)^{-1}(d)})\\
&\simeq\bigoplus_{k\in\fS_p\backslash\fS_d/\fS_q}\!\!\!\!\!\!
\Ind_{p\wedge k(q)}((U_1\otimes V_{k^{-1}(1)})\boxtimes\dots\boxtimes(U_d\otimes V_{k^{-1}(d)})).
\end{align*}
This isomorphism is clearly $\fS_d$-invariant.
Moreover if $\cC$ has a braiding, the induced braiding
at these objects are the direct sum of the morphisms
\begin{gather*}
\Ind_{p\wedge k(q)}((U_1\otimes V_{k^{-1}(1)})\boxtimes\dots\boxtimes(U_d\otimes V_{k^{-1}(d)}))
\eqmap\Ind_{q\wedge k^{-1}(p)}((V_1\otimes U_{k(1)})\boxtimes\dots\boxtimes(V_d\otimes U_{k(d)})).
\end{gather*}
Thus the full subcategory $\pW_d(\cC)$ is closed under
the tensor product and the braiding of $\cW_d(\cC)$.

\subsection{Restriction and Induction}
\label{sec:res_ind}
Let $d_1,d_2\in\NN$ and put $d\coloneqq d_1+d_2$.
We write $i'\coloneqq d_1+i$ for $i=1,\dots,d_2$.
There is a natural embedding of groups
$\fS_{d_1}\times\fS_{d_2}\hookrightarrow\fS_d$
where $\fS_{d_1}$ and $\fS_{d_2}$
acts on $\{1,\dots,d_1\}$ and $\{1',\dots,d_2'\}$ respectively.
Since there is a fully faithful embedding of categories
\begin{gather*}
\cC^G\barboxtimes\cD^H
\to(\cC\barboxtimes\cD)^{G\times H}
\end{gather*}
when groups $G$ and $H$ acts on $\cC$ and $\cD$ respectively,
we have the induction functor
\begin{gather*}
\Ind_{\fS_{d_1}\times \fS_{d_2}}^{\fS_d}\colon
\cW_{d_1}(\cC)\barboxtimes\cW_{d_2}(\cC)\to\cW_d(\cC).
\end{gather*}
To keep notations simple, we also use the binary operator $*$ to denote this induction functor:
\begin{gather*}
U_1*U_2\coloneqq\Ind_{\fS_{d_1}\times \fS_{d_2}}^{\fS_d}(U_1\boxtimes U_2).
\end{gather*}
This operator is associative and commutative
up to canonical isomorphisms.
The direct sum category $\cW_\bullet(\cC)\coloneqq\bigoplus_{d\in\NN}\cW_d(\cC)$
forms a graded $\kk$-symmetric tensor category with respect to the product $*$.

In the other direction, we have no natural restriction functors
since the embedding of categories above is not invertible in general.
However, for an object of the form $\Ind_p(U_1\boxtimes\dots\boxtimes U_d)$
we can calculate its restriction to $\fS_{d_1}\times\fS_{d_2}$.
We omit the proof of the next lemma.
\begin{lemma}
For $U_1,\dots,U_d\in\cC$,
in the $\kk$-linear category $(\cC^{\barboxtimes d})^{\fS_{d_1}\times\fS_{d_2}}$
\begin{gather*}
\Ind_p(U_1\boxtimes\dots\boxtimes U_d)\simeq
\!\!\!\!\!\!\!\!\!\!\!\!
\bigoplus_{g\in{\fS_{d_1}\times\fS_{d_2}}\backslash\fS_{d}/\fS_p}
\!\!\!\!\!\!\!\!\!\!\!\!
\Ind_q(U_{g^{-1}(1)}\boxtimes\dots\boxtimes U_{g^{-1}(d_1)})\boxtimes
\Ind_{q'}(U_{g^{-1}(1')}\boxtimes\dots\boxtimes U_{g^{-1}(d_2')}).
\end{gather*}
Here $q\in P(d_1)$ and $q'\in P(d_2)$ are the restriction of the equivalent relation $g(p)\in P(d)$
to each components.
The notation $g\in{\fS_{d_1}\times\fS_{d_2}}\backslash\fS_{d}/\fS_p$ is same as the previous one.
\end{lemma}

Thus we can define the restriction functor on $\pW_d(\cC)$:
\begin{gather*}
\Res_{\fS_{d_1}\times \fS_{d_2}}^{\fS_d}\colon
\pW_d(\cC)\to\pW_{d_1}(\cC)\barboxtimes\pW_{d_2}(\cC).
\end{gather*}
It is both the left and the right adjoint of the restricted induction functor on $\pW_d(\cC)$.

On the other hand, let $d_1,d_2\in\NN$ and put $d\coloneqq d_1 d_2$.
Let us write $i^{(j)}\coloneqq(j-1)d_1+i$ for $i=1,\dots,d_1$ and $j=1,\dots,d_2$.
The wreath product of the symmetric group $\fS_{d_1}\wr\fS_{d_2}$
can also be embedded into $\fS_d$ naturally:
the $j$-th component of $(\fS_{d_1})^{d_2}$ correspond to
the permutations on $\{1^{(j)},\dots,d_1^{(j)}\}$
and $\fS_{d_2}$ shuffles the index $j$ of $i^{(j)}$
for all $i=1,\dots,d_1$ simultaneously.
This gives us the induction functor and the restriction functor again:
\begin{align*}
\Ind_{\fS_{d_1}\wr \fS_{d_2}}^{\fS_d}&
\colon\cW_{d_2}(\cW_{d_1}(\cC))\to\cW_d(\cC),\\
\Res_{\fS_{d_1}\wr \fS_{d_2}}^{\fS_d}&
\colon\pW_d(\cC)\to\pW_{d_2}(\pW_{d_1}(\cC)).
\end{align*}
Their calculations are same for $\fS_{d_1}\times\fS_{d_2}$
but using $\fS_{d_1}\wr\fS_{d_2}$.

\section{Wreath Product in Non-integral Rank}
In this section we introduce our main product in this paper, the category $\cS_t(\cC)$.
It interpolates $\cW_d(\cC)$, the categories of representations of wreath products
from $d\in\NN$ to $t\in\kk$.
The original idea of the arguments is due to Deligne~\cite{Deligne:2007}
who first consider the representation theory of symmetric group of non-integral rank.

\subsection{Definition of 2-functor \texorpdfstring{$\cS_t$}{St}}

To apply the 2-functor $\cS_t$,
we need the fixed ``unit object'' in the target category.
So we introduce the notion of ``category with unit'' as follows.

\begin{definition}
\begin{enumerate}
\item A \term{$\kk$-linear category with unit} is
a $\kk$-linear category $\cC$ equipped with a fixed
object $\unit_\cC\in\cC$ which satisfies
$\End_\cC(\unit_\cC)\simeq\kk$.
\item A \term{$\kk$-linear functor with unit}
from $\cC$ to $\cD$ is
a $\kk$-linear functor $F\colon\cC\to\cD$ along with an isomorphism $\iota_F\colon\unit_\cD\eqmap F(\unit_\cC)$.
\item A \term{$\kk$-linear transformation with unit}
from $F$ to $G$ is a $\kk$-linear transformation $\eta\colon F\to G$
which satisfies the same condition as $\kk$-tensor transformations, i.e.\ $\eta(\unit_\cC)=\id_{\unit_\cD}$.
See the diagram on the right in Definition~\ref{def:tenfun}~(2).
\end{enumerate}
\end{definition}

We denote by $\uCat_\kk$ the 2-category consisting of
$\kk$-linear categories, functors and transformations with unit.
Obviously there are forgetful 2-functors $\tCat_\kk\to\uCat_\kk\to\Cat_\kk$.
The reader should check that we can apply all the 2-functors we have defined
to categories with unit and create new categories with unit.

Now fix a $\kk$-linear category $\cC$ with unit
and $d\in\NN$.

\begin{definition}
\label{def:bracket}
Let $I$ be a finite set and
$U_I=(U_i)_{i\in I}$ be a family of objects in $\cC$ indexed by $I$.
Set $m=\#I$ and write $I=\{i_1,\dots,i_m\}$.
Let us define $\br{U_I}\in\cW_d(\cC)$ by
\begin{align*}
\br{U_I}\coloneqq
\begin{dcases*}
U_{i_1}*\dots*U_{i_m}*\unit_\cC^{\boxtimes(d-m)}, &if $m\leq d$,\\
0, &otherwise.
\end{dcases*}
\end{align*}
This object is well-defined because it does not depend on the order of $i_1,\dots,i_m$.
We also write $\br{U_I}$ as $\br{U_{i_1},\dots,U_{i_m}}$.
\end{definition}
Before studying these objects, we introduce some notations.

\begin{definition}
Let $I_1,\dots,I_l$ be finite sets.
A \term{recollement} (gluing) of $I_1,\dots,I_l$
is a partition $r\in P(I_1\sqcup\dots\sqcup I_l)$ such that
for any $a=1,\dots,l$ and $i,i'\in I_a$, $i\sim_r i'$ implies $i=i'$.
In other words, $r$ is a recollement if
each $I_a\to(I_1\sqcup\dots\sqcup I_l)/{\sim_r}$ is injective.
Let us denote by $R(I_1,\dots,I_l)$ the set of recollements of $I_1,\dots,I_l$.

For $\{a_1,\dots,a_p\}\subset\{1,\dots,l\}$,
let $\pi_{a_1,\dots,a_p}\colon R(I_1,\dots,I_l)\to R(I_{a_1},\dots,I_{a_p})$
be the map which takes restriction of equivalence relation
via $I_{a_1}\sqcup\dots\sqcup I_{a_p}\subset I_1\sqcup\dots\sqcup I_l$.
\end{definition}
\begin{notation}
For example, $R(I,J)$ is the set of partitions of the form
\begin{align*}
r=\{\{i,j\},\dots,\{i'\},\dots,\{j'\},\dots\}
\end{align*}
where $i,i',\dots\in I$ and $j,j',\dots\in J$.
For convenience, we represent such $r$ by
\begin{align*}
r=\{(i,j),\dots,(i',\varnothing),\dots,(\varnothing,j'),\dots\}
\end{align*}
where $\varnothing$ is another index different from any element of $I\sqcup J$
so that we can simply write recollement as $r=\{(i,j),\dots\}$.
For more than two sets $I_1,\dots,I_l$, we use the same notation
$r=\{(i_1,\dots,i_l),\dots\}\in R(I_1,\dots,I_l)$.
For any family of objects $U_I$ in $\cC$,
the symbol $U_i$ denotes the unit element $\unit_\cC$ if $i=\varnothing$.
\end{notation}

\begin{definition}
Let $U_I,V_J$ be finite families of objects in $\cC$.
For $r\in R(I,J)$, define the $\kk$-module
\begin{align*}
H(U_I;V_J)
&\coloneqq\!\!\bigoplus_{r\in R(I,J)}\!\!
H_r(U_I;V_J)
\intertext{where for each $r\in R(I,J)$,}
H_r(U_I;V_J)&\coloneqq
\bigotimes_{(i,j)\in r}\Hom_\cC(U_i,V_j).
\end{align*}
This $\kk$-module is $\NN$-graded by length of recollements.
Let
\begin{align*}
H^d(U_I;V_J)
&\coloneqq\!\!\bigoplus_{\substack{r\in R(I,J)\\\#r=d}}\!\!
H_r(U_I;V_J)
\end{align*}
and write $H^{\le d}(U_I;V_J)\coloneqq\bigoplus_{e\le d}H^e(U_I;V_J)$ and
$H^{>d}(U_I;V_J)\coloneqq\bigoplus_{e>d}H^e(U_I;V_J)$.
Obviously $H^d(U_I;V_J)=0$ unless $\#I,\#J\leq d\leq\#I+\#J$.
\end{definition}

Let $I,J$ be finite sets such that $\#I,\#J\leq d$.
Write $I=\{i_1,\dots,i_m\}$, $J=\{j_1,\dots,j_n\}$ and
let $p,q\in P(d)$ as
\begin{align*}
p&={\{\{1\},\dots,\{m\},\{m+1,\dots,d\}\}},&
q&={\{\{1\},\dots,\{n\},\{n+1,\dots,d\}\}}
\end{align*}
respectively. For each $g\in\fS_d$,
we can take a unique recollement $r\in R(I,J)$
which satisfies $i_k\sim_r j_l$ if and only if $g(k)=l$.
The correspondence $g\mapsto r$ induces a bijection
$\fS_p\!\backslash\fS_d/\fS_q\bijectivemap\{r\in R(I,J)\;|\;\#r\leq d\}$.
Thus the isomorphism in Section~\ref{sec:replica} gives
\begin{align*}
\Hom_{\cW_d(\cC)}(\br{U_I},\br{V_J})
\simeq H^{\leq d}(U_I;V_J).
\end{align*}
This is also true when $\#I>d$ or $\#J>d$ because both sides are zero.

For $\Phi\in H^{\leq d}(U_I;V_J)$, let us denote by $\br{\Phi}\colon\br{U_I}\to\br{V_J}$
the map corresponding to $\Phi$ via the isomorphism above.
We give an explicit description of it here.
\begin{definition}
Let $I,J$ be finite sets.
We say that a sequence
$((i_1,j_1),\dots,(i_d,j_d))$
for $i_1,\dots,i_d\in I\sqcup\{\varnothing\}$, $j_1,\dots,j_d\in J\sqcup\{\varnothing\}$
is \term{adapted} to a recollement $r\in R(I,J)$ if
the sequence obtained by removing all $(\varnothing,\varnothing)$'s from it
is equal to a permutation of all the elements in the set $r=\{(i,j),\dots\}$.
\end{definition}
Recall that $\br{U_I}$
is a direct sum of objects of the form $U_{i_1}\boxtimes\dots\boxtimes U_{i_d}$
($i_1,\dots,i_d\in I\sqcup\{\varnothing\}$)
in $\cC^{\barboxtimes d}$.
Let $r\in R(I,J)$ and $\Phi\in H_r(U_I,V_J)$.
Each matrix entry of $\br{\Phi}\colon\br{U_I}\to\br{V_J}$ at the cell
\begin{align*}
U_{i_1}\boxtimes\dots\boxtimes U_{i_d}\to V_{j_1}\boxtimes\dots\boxtimes V_{j_d}
\end{align*}
for $i_1,\dots,i_d\in I\sqcup\{\varnothing\}$, $j_1,\dots,j_d\in J\sqcup\{\varnothing\}$
is equal to $\Phi$ (after reordering the tensor terms) if
$((i_1,j_1),\dots,(i_d,j_d))$ is adapted to $r$
and otherwise zero.

We extend the usage of this symbol $\br{\Phi}$ for $\Phi\in H^{>d}(U_I;V_J)$ so that $\br{\Phi}=0$.
Thus we have a $\kk$-linear map $\br{\bullet}\colon H(U_I;V_J)\to\Hom_{\cW_d(\cC)}(\br{U_I},\br{V_J})$
which is surjective and whose kernel is $H^{>d}(U_I;V_J)$.

\begin{example}
Let us consider the case $d=3$.
Take objects $U,V\in\cC$ then we have
\begin{gather*}
\Hom_{\cW_3(\cC)}(\br[3]{U},\br[3]{V})\simeq
\Hom_\cC(U,V)\oplus(\Hom_\cC(U,\unit_\cC)\otimes\Hom_\cC(\unit_\cC,V)).
\end{gather*}
Let us confirm this directly as follows.
In $\cC^{\barboxtimes3}$,
\begin{align*}
\br[3]{U}&\simeq(U\boxtimes\unit_\cC\boxtimes\unit_\cC)
\oplus(\unit_\cC\boxtimes U\boxtimes\unit_\cC)
\oplus(\unit_\cC\boxtimes\unit_\cC\boxtimes U),\\
\br[3]{V}&\simeq(V\boxtimes\unit_\cC\boxtimes\unit_\cC)
\oplus(\unit_\cC\boxtimes V\boxtimes\unit_\cC)
\oplus(\unit_\cC\boxtimes\unit_\cC\boxtimes V).
\end{align*}
For morphisms $\varphi\colon U\to V$, $\psi\colon U\to\unit_\cC$ and $\xi\colon\unit_\cC\to V$
in $\cC$,
the morphisms $\br[3]{\varphi}\colon\br[3]{U}\to\br[3]{V}$ and
$\br[3]{\psi\otimes\xi}\colon\br[3]{U}\to\br[3]{V}$
in $\cC^{\barboxtimes3}$ are represented by the matrices
\begin{align*}
\br[3]{\varphi}&\coloneqq
\begin{pmatrix}
\varphi\boxtimes1\boxtimes1&0&0\\
0&1\boxtimes\varphi\boxtimes1&0\\
0&0&1\boxtimes1\boxtimes\varphi
\end{pmatrix},\\
\br[3]{\psi\otimes\xi}&\coloneqq
\begin{pmatrix}
0&\xi\boxtimes\psi\boxtimes1&\xi\boxtimes1\boxtimes\psi\\
\psi\boxtimes\xi\boxtimes1&0&1\boxtimes\xi\boxtimes\psi\\
\psi\boxtimes1\boxtimes\xi&1\boxtimes\psi\boxtimes\xi&0
\end{pmatrix}
\end{align*}
respectively where $1$ stands for $\id_{\unit_\cC}\colon\unit_\cC\to\unit_\cC$.
It is clear that the space of all $\fS_3$-invariant morphisms
in $\cC^{\barboxtimes3}$ are spanned by them.
It is true for all $d\geq2$ but when $d=0$ or $1$ the matrices become smaller
and some of the non-zero terms disappear.
\end{example}

What we have to do next is to compute the composition of these morphisms.

\begin{definition}
Let $r\in R(I,J)$, $s\in R(J,K)$ be two recollements.
We define the set
\begin{align*}
R(s\circ r)\coloneqq\{u\in R(I,J,K)\;|\;\pi_{1,2}(u)=r,\;\pi_{2,3}(u)=s\}.
\end{align*}
For $u\in R(s\circ r)$, $\Phi\in H_r(U_I;V_J)$ and
$\Psi\in H_s(V_J;W_K)$,
we denote by $\Psi\circ_u\Phi\in H_{\pi_{1,3}(u)}(U_I;W_K)$
the element obtained by composing terms of $\Phi\otimes\Psi$
using compositions
\begin{align*}
\Hom_\cC(U_i,V_j)\otimes\Hom_\cC(V_j,W_k)\to\Hom_\cC(U_i,W_k)
\end{align*}
for all $(i,j,k)\in u$.
If both $i$ and $k$ are $\varnothing$, the composite of $\unit_\cC\to V_j\to\unit_\cC$
is regarded as a scalar in $\kk\simeq\End_\cC(\unit_\cC)$.
\end{definition}

\begin{lemma}
\label{lem:complaw}
Let $\Phi\in H_r(U_I;V_J),\Psi\in H_s(V_J;W_K)$ be as above. Then
\begin{gather*}
\br{\Psi}\circ\br{\Phi}=\!\!\sum_{u\in R(s\circ r)}
\!\!P_u(d)\;\br{\Psi\circ_u\Phi}
\end{gather*}
where $P_u$ is the polynomial
\begin{gather*}
P_u(T)\coloneqq
\!\!\prod_{\#\pi_{1,3}(u)\leq a<\#u}\!\!\!\!\!\!(T-a)
=(T-\#\pi_{1,3}(u))\dotsm(T-\#u+1).
\end{gather*}
\end{lemma}

Note that the degree $\#u-\#\pi_{1,3}(u)$ of $P_u$ does not depend on the choice of $u\in R(s\circ r)$.
This is equal to the number of ``orphans'' $(\varnothing,j,\varnothing)\in u$ in $J$.

\begin{proof}
If $\#I>d$ or $\#K>d$, both sides above are zero and the equation clearly holds.
Otherwise the composite is a sum of morphisms of the form
$\br{\Psi\circ_u\Phi}$.
Since $\pi_{1,3}\colon R(s\circ r)\to R(I,K)$ is injective,
we can uniquely write
\begin{align*}
\br{\Psi}\circ\br{\Phi}=
\!\!\sum_{\substack{u\in R(s\circ r)\\\#\pi_{1,3}(u)\leq d}}\!\!a_u\br{\Psi\circ_u\Phi}
\end{align*}
for some $a_u\in\NN$ for each $u\in R(s\circ r)$.
So take $u\in R(s\circ r)$ with $\#\pi_{1,3}(u)\leq d$
and fix a sequence $((i_1,k_1),\dots,(i_d,k_d))$ adapted to $\pi_{1,3}(u)$.
Since the matrix entry of $\br{\Psi}\circ\br{\Phi}$ in $\cC^{\barboxtimes d}$ at the cell
\begin{gather*}
U_{i_1}\boxtimes\dots\boxtimes U_{i_d}\to W_{k_1}\boxtimes\dots\boxtimes W_{k_d}
\end{gather*}
coincides with $a_u(\Psi\circ_u\Phi)$,
$a_u$ is equal to the number of sequences $(j_1,\dots,j_d)$ such that
both $((i_1,j_1),\dots,(i_d,j_d))$ and $((j_1,k_1),\dots,(j_d,k_d))$ are adapted to
the recollements $r$ and $s$ respectively.
For $a=1,\dots,d$, $j_a\in J\sqcup\{\varnothing\}$ is uniquely determined
if at least one of $i_a$ and $k_a$ is not $\varnothing$.
Thus only we can choose is the positions of $j\in J$
correspond to orphans $(\varnothing,j,\varnothing)\in u$.
The number of them is $\#u-\#\pi_{1,3}(u)$ and we can place them in
$d-\#\pi_{1,3}(u)$ distinct positions.
So the number of choices is
$P_u(d)=(d-\#\pi_{1,3}(u))\dotsm(d-\#u+1)$.
\end{proof}

We remark that in the composition law above the rank $d$ only appears
as polynomials in the coefficients.
So we can change $d\in\NN$ into an arbitrary $t\in\kk$.
This is the definition of our category $\cS_t(\cC)$.

\begin{definition}
Let $\cC$ be a $\kk$-linear category with unit and $t\in\kk$.
We define the $\kk$-linear category $\cS_t(\cC)$
by taking the pseudo-abelian envelope of the category defined
as follows:
\begin{description}
\item[Object] A finite family of objects in $\cC$ written as $\ag{U_I}$ for $U_I=(U_i)_{i\in I}$.
We also write $\ag{U_I}=\ag{U_{i_1},\dots,U_{i_m}}$ when $I=\{i_1,\dots,i_m\}$.
\item[Morphism] For objects $\ag{U_I}$ and $\ag{V_J}$,
\begin{gather*}
\Hom_{\cS_t(\cC)}(\ag{U_I},\ag{V_J})\simeq H(U_I;V_J).
\end{gather*}
For each $\Phi\in H(U_I;V_J)$, we denote by $\ag{\Phi}$
the corresponding morphism in $\cS^0_t(\cC)$.
The composition of morphisms is given by
\begin{align*}
\ag{\Psi}\circ\ag{\Phi}\coloneqq\!\!\sum_{u\in R(s\circ r)}
\!\!P_u(t)\;\ag{\Psi\circ_u\Phi}
\end{align*}
for each $\Phi\in H_r(U_I;V_J)$, $\Psi\in H_s(V_J;W_K)$.
\end{description}
The unit object $\unit_{\cS_t(\cC)}$ of $\cS_t(\cC)$ is the object $\ag{}$
corresponding to the empty family.
\end{definition}

\begin{lemma}
The category $\cS_t(\cC)$ is well-defined;
that is, there are identity morphisms and the composition of morphisms
is associative.
\end{lemma}
\begin{proof}
The identity morphism of $\ag{U_I}$ is given by
\begin{align*}
\Bigag{\bigotimes_{i\in I}\id_{U_i}}
\in \bigag{H_{r_I}(U_I;U_I)}
\end{align*}
where $r_I=\{(i,i)\,|\,i\in I\}\in R(I,I)$.
To prove associativity, we first prove the case for
replacing $\kk$ with the polynomial ring $\kk[T]$ and $t$ with the indeterminate $T\in\kk[T]$.
Let $\Phi\in H(U_I;V_J)$, $\Psi\in H(V_J;W_K)$ and $\Theta\in H(W_K;X_L)$.
Set
\begin{align*}
\ag[T]{\Upsilon}=
\bigl(\ag[T]{\Theta}\circ\ag[T]{\Psi}\bigr)\circ\ag[T]{\Phi}-
\ag[T]{\Theta}\circ\bigl(\ag[T]{\Psi}\circ\ag[T]{\Phi}\bigr).
\end{align*}
For all $d\in\NN$, we have
$\br{\Upsilon|_{T=d}}=0$ by Lemma~\ref{lem:complaw}.
Since $\br{\bullet}$ is an isomorphism when $d\geq \#I+\#L$,
we have $\Upsilon|_{T=d}=0$ for such $d$.
Thus $\Upsilon=0$ and
we get the associativity for $t\in\kk$ by substituting $T=t$.
\end{proof}

\begin{definition}
For a functor $F\colon\cC\to\cD$ with unit,
let $\cS_t(F)$ be the functor $\cS_t(\cC)\to\cS_t(\cD)$ with unit which sends
$\ag{U_I}$ to $\ag{F(U_I)}$ and $\ag{\Phi}$ to $\ag{F(\Phi)}$.
For a $\kk$-linear transformation $\eta\colon F\to G$ with unit,
let us define the $\kk$-linear transformation $\cS_t(\eta)\colon\cS_t(F)\to\cS_t(G)$ with unit by
\begin{align*}
\cS_t(\eta)(\ag{U_I})\coloneqq\Bigag{\bigotimes_{i\in I}\eta(U_i)}
\in \bigag{H_{r_I}(F(U_I);G(U_I))}
\end{align*}
where $r_I=\{(i,i)\,|\,i\in I\}\in R(I,I)$.
These operations define a 2-functor $\cS_t\colon\uCat_\kk\to\uPsCat_\kk$
where $\uPsCat_\kk$ is defined as same as before.
\end{definition}

Now the following statements are obvious.
\begin{theorem}
Let $d\in\NN$.
For finite families $U_I,V_J$ of objects in $\cC$, the map
\begin{align*}
\Hom_{\cS_d(\cC)}(\ag[d]{U_I},\ag[d]{V_J})&
\to\Hom_{\cW_d(\cC)}(\br{U_I},\br{V_J})\\
\ag[d]{\Phi}&\mapsto\br{\Phi}
\end{align*}
is surjective and its kernel is $\bigag[d]{H^{>d}(U_I;V_J)}$.
In particular, it is an isomorphism when $d\geq \#I+\#J$.
This map induces a functor
$\cS_d(\cC)\to\cW_d(\cC);\ag[d]{U_I}\mapsto\br{U_I}$.
If $d!$ is invertible in $\kk$, this functor is also essentially surjective on objects.
\end{theorem}


\begin{remark}
Deligne's category $\mathrm{Rep}(S_t,\kk)$ in \cite{Deligne:2007}
is equal to $\cS_t(\Triv_\kk)$, in our language.
Since $\cS_t(\cC)\simeq\cS_t(\Ps(\cC))$,
this is also equivalent to $\cS_t(\Rep(\kk))$.
Its generalization $\mathrm{Rep}(G\wr S_t,\kk)$ for a finite group $G$
by Knop~\cite{Knop:2006,Knop:2007} is equivalent to the full subcategory
of $\cS_t(\Rep(\kk[G]))$ generated by $\ag{\kk[G]}^{\otimes m}$
where $\kk[G]$ is
the regular representation of $G$.
\end{remark}

\subsection{\texorpdfstring{$\cS_t$}{St} for Tensor categories}
When $\cC$ is a $\kk$-tensor category,
we can calculate the tensor product of objects of the form $\br{U_I}$
in the same manner as in the previous subsection.
It holds for families $U_I$ and $V_J$ that
\begin{align*}
\br{U_I}\otimes\br{V_J}
\simeq\!\bigoplus_{\substack{r\in R(I,J)\\\#r\leq d}}\!\bigbr{T_r(U_I,V_J)}
\simeq\!\bigoplus_{r\in R(I,J)}\!\bigbr{T_r(U_I,V_J)}.
\end{align*}
Here, for each $r\in R(I,J)$,
$T_r(U_I,V_J)$ is the family
\begin{align*}
T_r(U_I,V_J)\coloneqq
(U_i\otimes V_j)_{(i,j)\in r}
\end{align*}
indexed by the set $r=\{(i,j),\dots\}$.
Remark that there is a bijection
\begin{align*}
R(I,J,K,L)\bijection\!\!\bigsqcup_{\substack{r\in R(I,J)\\s\in R(K,L)}}\!\!R(r,s)
\end{align*}
where $R(r,s)$ denotes the set of recollements between the sets
$r=\{(i,j),\dots\}$ and $s=\{(k,l),\dots\}$.
Via this bijection
a recollement $u\in R(I,J,K,L)$ correspond to
$u'\in R(\pi_{1,2}(u),\pi_{3,4}(u))$ which satisfies
$((i,j),(k,l))\in u'$ if and only if $(i,j,k,l)\in u$.
So using this bijection
the morphisms between tensor products are given by
\begin{align*}
\Hom_{\cW_d(\cC)}(\br{U_I}\otimes\br{V_J},\br{W_K}\otimes\br{X_L})
\simeq\!\!\bigoplus_{\substack{u\in R(I,J,K,L)\\\#u\leq d}}\!\!\!\!H_u(U_I,V_J;W_K,X_L)
\end{align*}
where for each $u\in R(I,J,K,L)$,
\begin{align*}
H_u(U_I,V_J;W_K,X_L)
&\coloneqq
H_{u'}(T_{\pi_{1,2}(u)}(U_I,V_J);T_{\pi_{3,4}(u)}(W_K,X_L))
\\&\simeq
\!\!\bigotimes_{(i,j,k,l)\in u}\!\!
\Hom_\cC(U_i\otimes V_j,W_k\otimes X_l).
\end{align*}
The proof of the next lemma is same as that of Lemma~\ref{lem:complaw}.

\begin{lemma}
For $\Phi\in H_r(U_I;W_K)$ and $\Psi\in H_s(V_J;X_L)$,
\begin{gather*}
\br{\Phi}\otimes\br{\Psi}=
\!\!\sum_{u\in R(r\otimes s)}\!\!\br{\Phi\otimes_u\Psi}.
\end{gather*}
Here,
\begin{gather*}
R(r\otimes s)\coloneqq\{u\in R(I,J,K,L)\;|\;\pi_{1,3}(u)=r,\;\pi_{2,4}(u)=s\}
\end{gather*}
and $\Phi\otimes_u\Psi\in H_u(U_I,V_J;W_K,X_L)$ is obtained by composing terms of $\Phi\otimes\Psi$
using tensor products
\begin{gather*}
\Hom_\cC(U_i,W_k)\otimes\Hom_\cC(V_j,X_l)\to\Hom_\cC(U_i\otimes V_j,W_k\otimes X_l)
\end{gather*}
for all $(i,j,k,l)\in u$.
\end{lemma}

\begin{definition}
We define tensor products on $\cS_t(\cC)$ in the same manner as above:
for families $U_I$ and $V_J$ of objects in $\cC$,
\begin{align*}
\ag{U_I}\otimes\ag{V_J}
&\coloneqq\!\bigoplus_{r\in R(I,J)}\!\bigag{T_r(U_I,V_J)}
\intertext{and for morphisms $\Phi\in H_r(U_I;W_K)$ and $\Psi\in H_s(V_J;X_L)$,}
\ag{\Phi}\otimes\ag{\Psi}&\coloneqq
\!\sum_{u\in R(r\otimes s)}\!\ag{\Phi\otimes_u\Psi}.
\end{align*}
This tensor product induces a structure of $\kk$-tensor category
to $\cS_t(\cC)$ and
we have an enriched 2-functor $\cS_t\colon\tCat_\kk\to\tPsCat_\kk$.
For $d\in\NN$, $\cS_d(\cC)\to\cW_d(\cC)$ induces a $\kk$-tensor functor.
\end{definition}

The generalized formula for $m$-fold tensor products is as follows.
The symbols
\begin{align*}
T_r(U_{I_1},\dots,U_{I_m})&\quad\text{for}\quad r\in R(I_1,\dots,I_m),\\
H_r(U_{I_1},\dots,U_{I_m};V_{J_1},\dots,V_{J_n})&\quad\text{for}\quad r\in R(I_1,\dots,I_m,J_1,\dots,J_n)
\end{align*}
are defined in the same manner as in the case $m=n=2$.
\begin{lemma}
Let $U_{I_1},\dots,U_{I_m},V_{J_1},\dots,V_{J_n}$ be families of objects in $\cC$. Then
\begin{align*}
\ag{U_{I_1}}\otimes\dots\otimes\ag{U_{I_m}}
&\simeq\mspace{-15mu}\bigoplus_{r\in R(I_1,\dots,I_m)}\mspace{-12mu}\bigag{T_r(U_{I_1},\dots,U_{I_m})},\\
\Hom_{\cS_t(\cC)}(\ag{U_{I_1}}\otimes\dots\otimes\ag{U_{I_m}},\ag{V_{J_1}}\otimes\dots\otimes\ag{V_{J_n}})
&\simeq\mspace{-36mu}\bigoplus_{r\in R(I_1,\dots,I_m,J_1,\dots,J_n)}
\mspace{-33mu}\bigag{H_r(U_{I_1},\dots,U_{I_m};V_{J_1},\dots,V_{J_n})}.
\end{align*}
\end{lemma}
By specializing it to the case that the all families are of size one,
we get:
\begin{corollary}
\label{cor:partition_hom}
For $U_1,\dots,U_m,V_1,\dots,V_n\in\cC$,
\begin{align*}
\ag{U_1}\otimes\dots\otimes\ag{U_m}
&\simeq\bigoplus_{p\in P(m)}\ag{T_p(U_1,\dots,U_m)},\\
\Hom_{\cS_t(\cC)}(\ag{U_1}\otimes\dots\otimes\ag{U_m},\ag{V_1}\otimes\dots\otimes\ag{V_n})
&\simeq\mspace{-5mu}\bigoplus_{p\in P(m,n)}\mspace{-4mu}\bigag{H_p(U_1,\dots,U_m;V_1,\dots,V_n)}.
\end{align*}
Here $P(m)=P(\{1,\dots,m\})$ and $P(m,n)\coloneqq P(\{1,\dots,m\}\sqcup\{1',\dots,n'\})$.
\end{corollary}

Note that the object $\ag{U_1,\dots,U_m}$
is obtained as a direct summand of $\ag{U_1}\otimes\dots\otimes\ag{U_m}$
by the corollary above.
Thus $\cS_t(\cC)$ is also generated by objects of this form.


\subsection{Base change}
Let $r\in R(I,J)$ be a recollement
between finite sets $I$ and $J$.
As before, we regard $r$ as a set
$r=\{(i,j),\dots\}$.
This set is naturally identified with the pushout
$I\sqcup J/{\sim_r}$.
Conversely, for such $r$, let us denote by $\overline r$
the pullback
\begin{gather*}
\overline r\coloneqq\{(i,j)\in I\times J\;|\;i\sim_r j\}
=\{(i,j)\in r\;|\;i,j\neq\varnothing\}.
\end{gather*}
So $I$, $J$, $r$ and $\overline r$ form a cartesian and cocartesian square
\begin{gather*}
\xymatrix @C 16pt @R 16pt{
&\overline r\ar@{_(->}[ld]\ar@{^(->}[rd]&\\
I\ar@{^(->}[rd]&&J\ar@{_(->}[ld]\\
&r&
}
\end{gather*}
in the category of finite sets.
Remark that there are bijections
\begin{align*}
R(I,J)
&\bijection\{\text{set $r$ with injective maps }I\hookrightarrow r,J\hookrightarrow r\text{ such that }I\sqcup J\to r\text{ is surjective}\}/{\sim}\\
&\bijection\{\text{set $\overline r$ with injective maps }\overline r\hookrightarrow I,\overline r\hookrightarrow J\}/{\sim}.
\end{align*}

Let $U_I,V_J$ be families of objects in $\cC$.
Take a recollement $r\in R(I,J)$ and write
\begin{gather*}
r=\{(i,j),\dots,(i',\varnothing),\dots,(\varnothing,j'),\dots\}
\end{gather*}
where $i,i',\dotsc\in I$ and $j,j',\dotsc\in J$.
Using this representation, let us write
\begin{align*}
U_I&=(U_i,\dots,U_{i'},\dots),&
V_J&=(V_j,\dots,V_{j'},\dots).
\intertext{respectively. Let us introduce four families}
U_r&\coloneqq(U_i,\dots,U_{i'},\dots,V_{j'},\dots),&
U_{\overline r}&\coloneqq(U_i,\dots),\\
V_r&\coloneqq(V_j,\dots,U_{i'},\dots,V_{j'},\dots),&
V_{\overline r}&\coloneqq(V_j,\dots)
\end{align*}
indexed by the sets $r$ and $\overline r$ respectively.

Take an element $\Phi\in H_r(U_I;V_J)$ of the form
\begin{gather*}
\Phi=\varphi_{i,j}^{(1)}\otimes\dots\otimes\varphi_{i'}^{(2)}\otimes\dots\otimes\varphi_{j'}^{(3)}\otimes\dotsb
\end{gather*}
where $\varphi_{i,j}^{(1)}\colon U_i\to V_j$,
$\varphi_{i'}^{(2)}\colon U_{i'}\to\unit_\cC$ and
$\varphi_{j'}^{(3)}\colon\unit_\cC\to V_{j'}$.
By the composition law in $\cS_t(\cC)$, we have that the map $\ag{\Phi}\colon\ag{U_I}\to\ag{V_J}$
factors through $\ag{U_r}$ and $\ag{V_r}$; that is, the composite
\begin{gather*}
\xymatrix{
\ag{U_I} \ar[rd]_-{\ag{\dotsb\otimes\varphi_{j'}^{(3)}\otimes\dotsb}}&&&&
\ag{V_J} \\&
\ag{U_r} \ar[rr]_-{\ag{\dotsb\otimes\varphi_{i,j}^{(1)}\otimes\dotsb}}&&
\ag{V_r} \ar[ru]_-{\ag{\dotsb\otimes\varphi_{i'}^{(2)}\otimes\dotsb}}&
}
\end{gather*}
is equal to $\ag{\Phi}$.
Now let us consider another composite which goes through $\ag{U_{\overline r}}$ and $\ag{V_{\overline r}}$:
\begin{gather*}
\xymatrix{
&
\ag{U_{\overline r}} \ar[rr]^-{\ag{\dotsb\otimes\varphi_{i,j}^{(1)}\otimes\dotsb}}&&
\ag{V_{\overline r}} \ar[rd]^-{\ag{\dotsb\otimes\varphi_{j'}^{(3)}\otimes\dotsb}}&\\
\ag{U_I} \ar[ru]^-{\ag{\dotsb\otimes\varphi_{i'}^{(2)}\otimes\dotsb}}&&&&
\ag{V_J}. \\
}
\end{gather*}
We denote this morphism by the symbol $\aag{\Phi}$.
By the composition law, we get the formula
\begin{gather*}
\aag{\Phi}=\sum_{s\leq r}\ag{\Phi|_s}
\end{gather*}
immediately. Here, for each recollement $s\leq r$, $\Phi|_s\in H_s(U_I;V_J)$
is obtained by composing terms of $\Phi$ using
\begin{gather*}
\Hom_\cC(U_i,\unit_\cC)\otimes\Hom_\cC(\unit_\cC,V_j)\to\Hom_\cC(U_i,V_j)
\end{gather*}
for each $i\in I$ and $j\in J$ such that $i\not\sim_r j$ but $i\sim_s j$.
Thus we have another isomorphism $\aag{\bullet}\colon H(U_I;V_J)\to\Hom_{\cS_t(\cC)}(\ag{U_I},\ag{V_J})$
and morphisms of the form $\aag{\Phi}$ also form a basis of $\Hom_{\cS_t(\cC)}(\ag{U_I},\ag{V_J})$.

Conversely, we can explicitly represent a morphism of the form $\ag{\Phi}$
as a linear combination of morphisms $\aag{\Psi}$.
For each recollements $s\leq r$, their M\"obius function is given by
$\mu(s,r)=(-1)^{\#r-\#s}$
since the subset $\{u\in R(I,J)\;|\;s\leq u\leq r\}$ is isomorphic to
the power set of a set of order $\#r-\#s$
as partially ordered set.
Thus we have the inverse formula
\begin{gather*}
\ag{\Phi}=\sum_{s\leq r}(-1)^{\#r-\#s}\aag{\Phi|_s}.
\end{gather*}

Now let us take two morphisms $\aag{\Phi}\colon\ag{U_I}\to\ag{V_J}$ and
$\aag{\Psi}\colon\ag{V_J}\to\ag{W_K}$
and calculate the composite of them.
Let $\Phi\in H_r(U_I;V_J)$ and $\Psi\in H_s(V_J;W_K)$ be
\begin{align*}
\Phi&=\varphi_{i,j}^{(1)}\otimes\dots\otimes\varphi_{i'}^{(2)}\otimes\dots\otimes\varphi_{j'}^{(3)}\otimes\dots\\
\Psi&=\psi_{j,k}^{(1)}\otimes\dots\otimes\psi_{j'}^{(2)}\otimes\dots\otimes\psi_{k'}^{(3)}\otimes\dots
\end{align*}
as same as before.
Let $J_1\subset J$ be the union of images $\overline r\hookrightarrow J$ and $\overline s\hookrightarrow J$
and denote by $V_{J_1}$ the subfamily of $V_J$ indexed by $J_1$.
By the composition law, the composite $\ag{V_{\overline r}}\to\ag{V_J}\to\ag{V_{\overline s}}$
is equal to the scalar multiple of the composite $\ag{V_{\overline r}}\to\ag{V_{J_1}}\to\ag{V_{\overline s}}$.
Here, its scalar coefficient is given by
\begin{gather*}
P_{r,s}(t)\!\!\prod_{j'\in J\setminus J_1}\!\!\psi_{j'}^{(2)}\circ\varphi_{j'}^{(3)}
\end{gather*}
where $P_{r,s}$ is the polynomial
\begin{gather*}
P_{r,s}(T)\coloneqq
\!\!\prod_{\#J_1\leq a<\#J}\!\!\!\!(T-a)
=(T-\#J_1)\dotsm(T-\#J+1)
\end{gather*}
and we regard
each $\psi_{j'}^{(2)}\circ\varphi_{j'}^{(3)}\colon\unit_\cC\to V_{j'}\to\unit_\cC$
as scalar via $\End_\cC(\unit_\cC)\simeq\kk$.
Then we can complete the square
\begin{gather*}
\xymatrix @C 14pt @R 12pt{
&&\ag{V_{\overline u}}\ar@{-->}[rd]&\\
\dotsb\ar[r]&
\ag{V_{\overline r}}\ar[rd]\ar@{-->}[ru]&&
\ag{V_{\overline s}}\ar[r]&\dotsb\\
&&\ag{V_{J_1}}\ar[ru]&&\\
}
\end{gather*}
using the base change formula.
To apply the formula, we regard $J_1$ as a recollement $J_1\in R(\overline r,\overline s)$
via the injective maps $\overline r\hookrightarrow J_1$ and $\overline s\hookrightarrow J_1$.
The sum is taken over all recollements $u\in R(\overline r,\overline s)$ such that
$u\leq J_1$.
Taken together, we obtain the formula in the next proposition.

For $u\in R(\overline r,\overline s)$, let us denote by $u'\in R(I,K)$ the induced recollement on $I$ and $K$ by
the injective maps $\overline u\hookrightarrow\overline r\hookrightarrow I$
and $\overline u\hookrightarrow\overline s\hookrightarrow K$.
Let $s\circ r$ be the maximal element of $R(s\circ r)$,
i.e.\ the equivalent relation on $I\sqcup J\sqcup K$ generated by $r$ and $s$,
so $J_1'=\pi_{1,3}(s\circ r)$.

\begin{proposition}
Let $r\in R(I,J)$, $s\in R(J,K)$,
$\Phi\in H_r(U_I;V_J)$ and $\Psi\in H_s(V_J,W_K)$ as above.
Put $\Xi\coloneqq\Psi\circ_{(s\circ r)}\Phi\in H_{J_1'}(H_I,W_K)$.
Then
\begin{gather*}
\aag{\Psi}\circ\aag{\Phi}=
P_{r,s}(t)
\sum_{u\leq J_1}(-1)^{\#J_1-\#u}
\aag{\Xi|_{u'}}.
\end{gather*}
\end{proposition}

The inequality
$\#\overline r,\#\overline s\geq\#\overline u=\#\overline{u'}$
for $r$, $s$ and $u$ above gives us the next corollary.
\begin{corollary}
\label{cor:ideal}
Let $U_I$, $V_J$ and $W_K$ be families of objects in $\cC$.
Take $d,e\in\NN$ and let $f\coloneqq\max\{d+\#K,e+\#I\}-\#J$. Then
\begin{gather*}
\bigaag{H^{\geq e}(V_J,W_K)}\;\circ\;\bigaag{H^{\geq d}(U_I,V_J)}\;\subset\;
\bigaag{H^{\geq f}(U_I,W_K)}.
\end{gather*}
In particular, $\bigaag{H^{\geq d}(U_I,U_I)}$ is a two-sided ideal of $\End_{\cS_t(\cC)}(\ag{U_I})$ for any $d$.
\end{corollary}

\subsection{Restriction and Induction}
We also interpolate the restriction functors defined in Section~\ref{sec:res_ind}
to arbitrary ranks.
\begin{definition}
Let $\cC$ be a $\kk$-linear category with unit and $t_1,t_2\in\kk$.
Put $t=t_1+t_2$.
We define the functor
$\Res^{\fS_t}_{\fS_{t_1}\times\fS_{t_2}}\colon\cS_t(\cC)\to\cS_{t_1}(\cC)\boxtimes\cS_{t_2}(\cC)$
by
\begin{align*}
\Res^{\fS_t}_{\fS_{t_1}\times\fS_{t_2}}(\ag{U_I})\coloneqq\bigoplus_{I'\subset I}
\ag[t_1]{U_{I'}}\boxtimes\ag[t_2]{U_{I\setminus I'}}.
\end{align*}
The map for morphisms is defined as follows.
Fix subsets $I'\subset I$ and $J'\subset J$ and 
take $r\in R(I,J)$.
Let $r'\in R(I',J')$ and $r''\in R(I\setminus I',J\setminus J')$
be the restricted recollements of $r$ to each subsets.
Then
\begin{gather*}
H_{r'}(U_{I'};V_{J'})\otimes H_{r''}(U_{I\setminus I'};V_{J\setminus J'})
\simeq H_{r'\sqcup r''}(U_I;V_J).
\end{gather*}
Here $r\sqcup r'\in R(I,J)$ is the equivalence relation generated by $r$ and $r'$.
For each $\Phi\in H_r(U_I,V_J)$,
the matrix entry of $\Res^{\fS_t}_{\fS_{t_1}\times\fS_{t_2}}(\ag{\Phi})$ at the cell
\begin{gather*}
\ag[t_1]{U_{I'}}\boxtimes\ag[t_2]{U_{I\setminus I'}}\to
\ag[t_1]{V_{J'}}\boxtimes\ag[t_2]{V_{J\setminus J'}}
\end{gather*}
is defined to be zero if $r\neq r'\sqcup r''$;
otherwise $\sum\ag[t_1]{\Phi'}\boxtimes\ag[t_2]{\Phi''}$
when we write $\Phi=\sum\Phi'\otimes\Phi''$ using
$\Phi'\in H_{r'}(U_{I'};V_{J'})$ and $\Phi''\in H_{r''}(U_{I\setminus I'};V_{J\setminus J'})$.
\end{definition}

\begin{definition}
Let $\cC$ be a $\kk$-linear category with unit, $t_1,t_2\in\kk$
and put $t=t_1 t_2$.
We define the functor $\Res^{\fS_t}_{\fS_{t_1}\wr\fS_{t_2}}\colon\cS_t(\cC)\to\cS_{t_2}(\cS_{t_1}(\cC))$ by
\begin{align*}
\Res^{\fS_t}_{\fS_{t_1}\wr\fS_{t_2}}(\ag{U_I})\coloneqq\bigoplus_{p\in P(I)}
\bigag[t_2]{\ag[t_1]{U_p}}.
\end{align*}
Here, $p$ runs over all partitions of $I$
and $\ag[t_1]{U_p}$ is the family of objects in $\cS_{t_1}(\cC)$
indexed by $p=\{I_1,\dots,I_l\}$:
\begin{gather*}
\ag[t_1]{U_p}\coloneqq(\ag[t_1]{U_{I_1}},\dots,\ag[t_1]{U_{I_l}}).
\end{gather*}
The map for morphisms is defined in the same manner;
the matrix entry of $\Res^{\fS_t}_{\fS_{t_1}\wr\fS_{t_2}}(\ag{\Phi})$
for $\Phi\in H_r(U_I,V_J)$
at the cell
\begin{gather*}
\bigag[t_2]{\ag[t_1]{U_p}}\to\bigag[t_2]{\ag[t_1]{V_q}}
\end{gather*}
is induced from $\Phi$ if $r$ is compatible with $p,q$ and otherwise zero.
\end{definition}

The well-definedness of these functors is proved by the same argument as the previous one:
consider the case for the indeterminate rank $T\in\kk[T]$ and
check that equations hold for all $T=d\gg0$ in $\cW_d(\cC)$.
Note that for a $\kk$-braided tensor category $\cC$, it is easier to define them
using the universality of $\cS_t(\cC)$, see Theorem~\ref{thm:universality}.

On the other hand, it does not seem possible to interpolate
the induction functors to general $t_1,t_2\in\kk$.
For example, if the functor $\Ind^{\fS_t}_{\fS_{t_1}\times\fS_{t_2}}$ exists
it should multiply ``dimensions''
of objects by the binomial coefficient
$t!/(t_1!t_2!)$, which is not a polynomial in $t_1,t_2$.
However, in the special case where
one of the parameters $t_2=d_2\in\NN$ is a natural number
and $d_2!$ is invertible in $\kk$,
we can define associative $*$-product by
\begin{align*}
\cS_{t_1}(\cC)\barboxtimes\cW_{d_2}(\cC)
&\to\cS_{t_1+d_2}(\cC)\\
\ag[t_1]{U_1,\dots,U_m}\boxtimes\br[d_2]{V_1,\dots,V_{d_2}}
&\mapsto\ag[t_1+d_2]{U_1,\dots,U_m,V_1,\dots,V_{d_2}}
\end{align*}
since $\cW_{d_2}(\cC)$ is generated by objects of this form.
This defines the action of $\kk$-tensor category $\cW_{\bullet}(\cC)$
on $\cS_{\bullet}(\cC)\coloneqq\bigoplus_{t\in\kk}\cS_t(\cC)$.

\subsection{\texorpdfstring{$\cS_t$}{St} for Braided Tensor Categories}

If a $\kk$-tensor category $\cC$ has a braiding $\sigma_\cC$
then the 2-functor $\cS_t$ naturally induces a braiding $\sigma_{\cS_t(\cC)}$ of $\cS_t(\cC)$.
Here its component
$\ag{U_I}\otimes\ag{V_J}\eqmap\ag{V_J}\otimes\ag{U_I}$
is the direct sum of isomorphisms
\begin{align*}
\Bigag{\bigotimes_{(i,j)\in r}\!\sigma_\cC(U_i,V_j)}
\colon\bigag{T_r(U_I,V_J)}\eqmap\bigag{T_{\tilde r}(V_J,U_I)}
\end{align*}
for all $r\in R(I,J)$ where $\tilde r\in R(J,I)$ is the corresponding recollement
to $r$ via $I\sqcup J\bijection J\sqcup I$.
Clearly if the braiding $\sigma_\cC$ is symmetric then so is $\sigma_{\cS_t(\cC)}$.

As we have seen, it is too complicated to describe the morphisms in $\cS_t(\cC)$.
But if a braiding $\sigma_\cC$ of the category $\cC$ is given,
we can use a very powerful tool: the graphical representation of morphisms.
First we represent object $\ag{U_1}\otimes\dots\otimes\ag{U_m}$
by labeled points placed side-by-side:
\begin{align*}
\ag{U_1}\otimes\dots\otimes\ag{U_m}=
\begin{xy}
(0,0)*{\bullet}+(0,3)*{U_1};
(8,0)*{\bullet}+(0,3)*{U_2};
(16,0)*{\cdots};
(24,0)*{\bullet}+(0,3)*{U_m};
\end{xy}
.
\end{align*}
When $m=0$, ``no points'' denotes the unit object $\unit_{\cS_t(\cC)}$.
Recall that objects of this form generate the pseudo-abelian category $\cS_t(\cC)$;
so to describe $\cS_t(\cC)$ it suffices to consider morphisms between them.
We represent such morphisms by strings which connect points from top to bottom.

For each morphism $\varphi\colon U\to V$ in $\cC$, we have
$\ag{\varphi}\colon\ag{U}\to\ag{V}$.
We represent it by a string with a label $\inbox{\varphi}$.
If $\varphi=\id_U\colon U\to U$, the label may be omitted: 
\begin{align*}
\ag{\varphi}&=\;
\begin{xy}
(0,9)*={\bullet}="u"+(0,3)*{U},
(0,-9)*={\bullet}="u'"+(0,-3)*{V},
(0,0)*+{\varphi}="p"*\frm{-};
"u";"p" **\dir{-},
"u'";"p" **\dir{-},
\end{xy}\;,&
\begin{xy}
(0,9)*={\bullet}="u"+(0,3)*{U},
(0,-9)*={\bullet}="u'"+(0,-3)*{U},
(0,0)*+{\id_U}="p"*\frm{-};
"u";"p" **\dir{-},
"u'";"p" **\dir{-},
\end{xy}\;&=\;
\begin{xy}
(0,9)*={\bullet}="u"+(0,3)*{U},
(0,-9)*={\bullet}="u'"+(0,-3)*{U},
"u";"u'" **\dir{-},
\end{xy}
\;=\id_{\ag{U}}.
\end{align*}

By definition,
the spaces of morphisms
$\unit_{\cS_t(\cC)}\to\ag{\unit_\cC}$
and $\ag{\unit_\cC}\to\unit_{\cS_t(\cC)}$ are
both isomorphic to $\End_\cC(\unit_\cC)$.
Take morphisms $\iota_\cC$ and
$\epsilon_\cC$ from them respectively
which correspond to $\id_{\unit_\cC}$.
We represent them by broken strings:
\begin{align*}
\iota_\cC&=
\begin{xy}
(0,-7)*={\bullet}="uv"+(0,-3)*{\unit_\cC},
"uv" \ar @{-x} (0,5),
\end{xy},&
\epsilon_\cC&=
\begin{xy}
(0,7)*={\bullet}="uv"+(0,3)*{\unit_\cC},
"uv" \ar @{-x} (0,-5),
\end{xy}.
\end{align*}

As we have seen, $\ag{U\otimes V}$ is a direct summand of $\ag{U}\otimes\ag{V}$.
We denote its retraction by $\mu_\cC(U,V)\colon\ag{U}\otimes\ag{V}\to\ag{U\otimes V}$
and section $\Delta_\cC(U,V)\colon\ag{U\otimes V}\to\ag{U}\otimes\ag{V}$.
We represent them by ramifications of strings:
\begin{align*}
\mu_\cC(U,V)&=
\begin{xy}
(-5,7)*={\bullet}="u"+(0,3)*{U},
(5,7)*={\bullet}="v"+(0,3)*{V},
(0,-7)*={\bullet}="uv"+(0,-3)*{U\otimes V},
"u";0 **\crv{(-5,2)},
"v";0 **\crv{(5,2)},
"uv";0 **\dir{-},
\end{xy},&
\Delta_\cC(U,V)&=
\begin{xy}
(-5,-7)*={\bullet}="u"+(0,-3)*{U},
(5,-7)*={\bullet}="v"+(0,-3)*{V},
(0,7)*={\bullet}="uv"+(0,3)*{U\otimes V},
"u";0 **\crv{(-5,-2)},
"v";0 **\crv{(5,-2)},
"uv";0 **\dir{-},
\end{xy}.
\end{align*}

Let us denote by $\tau_\cC(U,V)$
the braiding $\sigma_{\cS_t(\cC)}(\ag{U},\ag{V})\colon\ag{U}\otimes\ag{V}\eqmap\ag{V}\otimes\ag{U}$
for short.
This morphism is represented by crossing strings.
We distinguish the braiding from its inverse by the sign of the crossing,
the overpass and the underpass:
\begin{align*}
\tau_\cC(U,V)&=\;
\begin{xy}
(0,6)*={\bullet}="u"+(0,3)*{U},
(10,6)*={\bullet}="v"+(0,3)*{V},
(0,-6)*={\bullet}="v'"+(0,-3)*{V},
(10,-6)*={\bullet}="u'"+(0,-3)*{U},
(5,0)*{\hole}="p",
"u";"u'" **\dir{-},
"v";"p" **\dir{-},
"v'";"p" **\dir{-},
\end{xy},&
\tau_\cC^{-1}(U,V)&=
\begin{xy}
(0,6)*={\bullet}="u"+(0,3)*{U},
(10,6)*={\bullet}="v"+(0,3)*{V},
(0,-6)*={\bullet}="v'"+(0,-3)*{V},
(10,-6)*={\bullet}="u'"+(0,-3)*{U},
(5,0)*{\hole}="p",
"v";"v'" **\dir{-},
"u";"p" **\dir{-},
"u'";"p" **\dir{-},
\end{xy}.
\end{align*}

We represent the tensor product of these morphisms by
placing corresponding diagrams side-by-side.
Finally we connect these diagrams from top to bottom to represent the composite of them.

\begin{example}
\label{ex:diagram}
The diagram in the introduction
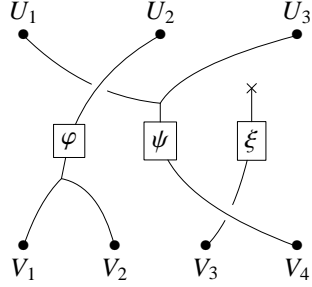
\begin{align*}
\begin{xy}
(-12,0)*+{\varphi}*\frm{-}="p",
(0,0)*+{\psi}*\frm{-}="q",
(12,0)*+{\xi}*\frm{-}="r",
(-18,14)*={\bullet}="u1"+(0,3)*{U_1},
(0,14)*={\bullet}="u2"+(0,3)*{U_2},
(18,14)*={\bullet}="u3"+(0,3)*{U_3},
(-18,-14)*={\bullet}="v1"+(0,-3)*{V_1},
(-6,-14)*={\bullet}="v2"+(0,-3)*{V_2},
(6,-14)*={\bullet}="v3"+(0,-3)*{V_3},
(18,-14)*={\bullet}="v4"+(0,-3)*{V_4},
(-8,7)*{\hole}="x",
(8.5,-9.5)*{\hole}="y",
"p"+(-1,-5)="p-",
"q"+(0,5)="q+",
"r"+(0,7)="r+",
"p";"p-" **\dir{-},
"q";"q+" **\dir{-},
"r";"r+" **\dir{-} ?>*\dir{x},
"p";"u2" **\crv{(-9,9)},
"q+"."x";"u1" **\crv{(-10,6)}, "q+";"u1"."x" **\crv{(-10,6)},
"q+";"u3" **\crv{(3,10)},
"p-";"v1" **\crv{(-16,-8)},
"p-";"v2" **\crv{(-8,-6)},
"r"."y";"v3" **\crv{(10,-10)}, "r";"v3"."y" **\crv{(10,-10)},
"q";"v4" **\crv{(2,-8)},
\end{xy}
\end{align*}
denotes the composite of morphisms
\begin{align*}
\ag{U_1}\otimes\ag{U_2}\otimes\ag{U_3}
&\xrightarrow{\makebox[102pt]{$\scriptstyle\tau_\cC^{-1}(U_1,U_2)\otimes\id_{\ag{U_3}}$}}
\ag{U_2}\otimes\ag{U_1}\otimes\ag{U_3}\\
&\xrightarrow{\makebox[102pt]{$\scriptstyle\id_{\ag{U_2}}\otimes\mu_\cC(U_1,U_3)\otimes\iota_\cC$}}
\ag{U_2}\otimes\ag{U_1\otimes U_3}\otimes\ag{\unit_\cC}\\
&\xrightarrow{\makebox[102pt]{$\scriptstyle\ag{\varphi}\otimes\ag{\psi}\otimes\ag{\xi}$}}
\ag{V_1\otimes V_2}\otimes\ag{V_4}\otimes\ag{V_3}\\
&\xrightarrow{\makebox[102pt]{$\scriptstyle\Delta_\cC(V_1,V_2)\otimes\id_{\ag{V_4}}\otimes\id_{\ag{V_3}}$}}
\ag{V_1}\otimes\ag{V_2}\otimes\ag{V_4}\otimes\ag{V_3}\\
&\xrightarrow{\makebox[102pt]{$\scriptstyle\id_{\ag{V_1}}\otimes\id_{\ag{V_2}}\otimes\tau_\cC(V_4,V_3)$}}
\ag{V_1}\otimes\ag{V_2}\otimes\ag{V_3}\otimes\ag{V_4}
\end{align*}
for $\varphi\colon U_2\to V_1\otimes V_2$,
$\psi\colon U_1\otimes U_3\to V_4$ and
$\xi\colon\unit_\cC\to V_3$.
\end{example}

Recall that we can decompose the space of morphisms
\begin{gather*}
\ag{U_1}\otimes\dots\otimes\ag{U_m}\to\ag{V_1}\otimes\dots\otimes\ag{V_n}
\end{gather*}
by partitions $P(m,n)$ as in Corollary~\ref{cor:partition_hom}.
It is easy to show that if we take the morphism
represented by the diagram above,
this morphism is decomposed as
\begin{gather*}
\sum_{q\leq p}\ag{\Theta_q}
\end{gather*}
using suitable $\Theta_q\in H_q(U_1,\dots,U_3;V_1,\dots,V_4)$
for each $q\leq p$ where $p\in P(3,4)$ is a partition
$\{\{2,1',2'\},\{1,3,4'\},\{3'\}\}$.
Moreover, the top component $\Theta_p$ is equal to $\varphi\otimes\psi\otimes\xi$.

To apply this argument globally,
we have to fix a ``shape''
of each partition.
For example,
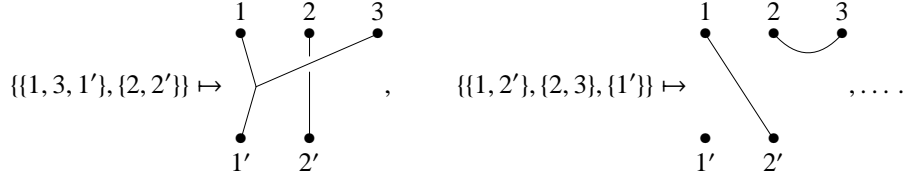
\begin{align*}
\{\{1,3,1'\},\{2,2'\}\}&\mapsto
\begin{xy}
(0,7)*={\bullet}="1"+(0,3)*{1},
(9,7)*={\bullet}="2"+(0,3)*{2},
(18,7)*={\bullet}="3"+(0,3)*{3},
(0,-7)*={\bullet}="1'"+(0,-3)*{1'},
(9,-7)*={\bullet}="2'"+(0,-3)*{2'},
(2,0)="p",
(9,3)*{\hole}="q",
"1";"p" **\dir{-},
"3";"p" **\dir{-},
"1'";"p" **\dir{-},
"2";"q" **\dir{-},
"2'";"q" **\dir{-},
\end{xy},&
\{\{1,2'\},\{2,3\},\{1'\}\}&\mapsto
\begin{xy}
(0,7)*={\bullet}="1"+(0,3)*{1},
(9,7)*={\bullet}="2"+(0,3)*{2},
(18,7)*={\bullet}="3"+(0,3)*{3},
(0,-7)*={\bullet}="1'"+(0,-3)*{1'},
(9,-7)*={\bullet}="2'"+(0,-3)*{2'},
"1";"2'" **\dir{-},
"2";"3" **\crv{(13.5,2)},
\end{xy},\dotsc.
\end{align*}
Let us describe it more precisely. For each partition $p\in P(m,n)$,
first we fix an order of the components
$p=\{I_1,\dots,I_l\}$.
For each $k=1,\dots,l$, write
\begin{align*}
I_k=\{i_{k,1},i_{k,2},\dots,i_{k,a(k)},j'_{k,1},j'_{k,2},\dots,j'_{k,b(k)}\}
\end{align*}
so that $i_{k,1}<i_{k,2}<\dots<i_{k,a(k)}$ and $j_{k,1}<j_{k,2}<\dots<j_{k,b(k)}$.
Next we choose braid group elements $g\in\fB_m$ and $h\in\fB_n$ which satisfy
\begin{align*}
(g^{-1}(1),\dots,g^{-1}(m))&=(i_{1,1},i_{1,2},\dots,i_{1,a(1)},\;\dotsc\;,i_{l,1},i_{l,2},\dots,i_{l,a(l)}),\\
(h^{-1}(1),\dots,h^{-1}(n))&=(j_{1,1},j_{1,2},\dots,j_{1,b(1)},\;\dotsc\;,j_{l,1},j_{l,2},\dots,j_{l,b(l)}).
\end{align*}
These are what we called the \term{shape of $p$}. Using these data,
we define a ``diagram labeling'' map
\begin{gather*}
f_p\colon H_p(U_1,\dots,U_m;V_1,\dots,V_n)\longrightarrow
\Hom_{\cS_t(\cC)}(\ag{U_1}\otimes\dots\otimes\ag{U_m},\ag{V_1}\otimes\dots\otimes\ag{V_n})
\end{gather*}
for each $p\in P(m,n)$ as follows.
Put
\begin{align*}
\tilde U_k&\coloneqq U_{i_{k,1}}\otimes U_{i_{k,2}}\otimes\dots\otimes U_{i_{k,a(k)}},&
\tilde V_k&\coloneqq V_{j_{k,1}}\otimes V_{j_{k,2}}\otimes\dots\otimes V_{j_{k,b(k)}}.
\end{align*}
For $\varphi_k\colon\tilde U_k\to\tilde V_k$
($k=1,\dots,l$),
the corresponding morphism
$f_p(\varphi_1\otimes\dots\otimes\varphi_l)$
is defined to be
\begin{gather*}
f_p(\varphi_1\otimes\dots\otimes\varphi_l)\coloneqq
(\tau_\cC^h)^{-1}\circ\Delta_\cC^p\circ(\ag{\varphi_1}\otimes\dots\otimes\ag{\varphi_l})
\circ\mu_\cC^p\circ\tau_\cC^g
\end{gather*}
where $\tau_\cC^g$ and $\tau_\cC^h$ are braidings along $g$ and $h$ respectively
and
\begin{align*}
\mu_\cC^p&\colon
\ag{U_{g^{-1}(1)}}\otimes\dots\otimes\ag{U_{g^{-1}(m)}}\to
\ag{\tilde U_1}\otimes\dots\otimes\ag{\tilde U_l}\\
\Delta_\cC^p&\colon
\ag{\tilde V_1}\otimes\dots\otimes\ag{\tilde V_l}\to
\ag{V_{h^{-1}(1)}}\otimes\dots\otimes\ag{V_{h^{-1}(n)}}
\end{align*}
are suitable composites of $\mu_\cC,\iota_\cC$ and $\Delta_\cC,\epsilon_\cC$ respectively
(this notion is well-defined since $\mu_\cC$ is associative and $\Delta_\cC$ is coassociative;
see Proposition~\ref{prop:diagram_transform}).
So the morphism in Example~\ref{ex:diagram}
is written as $f_p(\varphi\otimes\psi\otimes\xi)$
if we choose a suitable shape of $p$.
It is easy to check that this map also satisfies unitriangularity
\begin{gather*}
f_p(\Phi)=\ag{\Phi}+\sum_{q\lneq p}\ag{\Theta_q}.\qquad\text{($\Theta_q\in H_q(U_1,\dots,U_m;V_1,\dots,V_n)$)}
\end{gather*}
Thus by the induction on the partial order of the partitions,
we have another isomorphism
\begin{gather*}
\Hom_{\cS_t(\cC)}(\ag{U_1}\otimes\dots\otimes\ag{U_m},\ag{V_1}\otimes\dots\otimes\ag{V_n})
\simeq\!\bigoplus_{p\in P(m,n)}\!f_p(H_p(U_1,\dots,U_m;V_1,\dots,V_n)).
\end{gather*}
Notice that this isomorphism depends on the shapes of the partitions we have chosen.

We say that a diagram is \term{of standard form}
if it represents a composite
\begin{gather*}
(\tau_\cC^h)^{-1}\circ\Delta_\cC^p\circ(\ag{\varphi_1}\otimes\dots\otimes\ag{\varphi_l})
\circ\mu_\cC^p\circ\tau_\cC^g
\end{gather*}
for some $p\in P(m,n)$.
Of course this notion also depends on the shapes we have chosen.
Bring these arguments all together, we have the next proposition.

\begin{proposition}
\label{prop:standardform}
Every morphism
$\ag{U_1}\otimes\dots\otimes\ag{U_m}\to\ag{V_1}\otimes\dots\otimes\ag{V_n}$
can be represented by a linear combination of diagrams of standard form.
In such a representation, the corresponding component of
$H_p(U_1,\dots,U_m;V_1,\dots,V_n)$ at each $p\in P(m,n)$ is uniquely determined.
\end{proposition}

\begin{remark}
Several known algebras are appeared as the endomorphism ring
of an object of the form $\ag{U}^{\otimes m}\in\cS_t(\cC)$.
For Deligne's case $\cC=\Rep(\kk)$,
$\End_{\cS_t(\cC)}(\ag{\unit_\kk}^{\otimes m})$
is the \term{partition algebra} introduced by
Jones~\cite{Jones:1994} and Martin~\cite{Martin:1994}.
More generally, fix $r\in\NN$ and
let $\cC\coloneqq\Rep(\kk)^{\ZZ/r\ZZ}$ the category of $(\ZZ/r\ZZ)$-graded $\kk$-modules
(Deligne's case is when $r=1$).
Let $U=\unit_\kk[-1]$ which has a component $\unit_\kk$ at degree $1$, so
\begin{align*}
\Hom_\cC(U^{\otimes m},U^{\otimes n})\simeq
\begin{dcases*}
\kk, &if $m\equiv n\pmod{r}$,\\
0, &otherwise.
\end{dcases*}
\end{align*}
The endomorphism ring
$\End_{\cS_t(\cC)}(\ag{U}^{\otimes m})$
is called the \term{$r$-modular party algebra}~\cite{Kosuda:2009}.
It is spanned by diagrams whose number of input legs
and that of output legs are congruent modulo $r$
at each its connected component.

Another example is Knop's case, $\cC=\Rep(\kk[G])$
for a finite group $G$.
The endomorphism ring
$\End_{\cS_t(\cC)}(\ag{\kk[G]}^{\otimes m})$
is the \term{{$G$}-colored partition algebra} of Bloss~\cite{Bloss:2003}.
To represent morphisms he uses little different diagrams from ours but
we can easily translate them into our form using
the following morphisms:
the right multiplication $\kk[G]\to\kk[G]$ by $g\in G$,
the diagonal embedding $\kk[G]\to\kk[G]\otimes\kk[G]$
and projection $\kk[G]\otimes\kk[G]\to\kk[G]$.
Note that in either case objects of the form $\ag{U}^{\otimes m}$
generate the whole pseudo-abelian category $\cS_t(\cC)$.
\end{remark}

\subsection{Universality of \texorpdfstring{$\cS_t(\cC)$}{St(C)}}
The last proposition tells us
that $\cS_t(\cC)$ is generated by
the morphisms $\ag{\varphi}$, $\mu_\cC(U,V)$, $\iota_\cC$, $\Delta_\cC(U,V)$ and $\epsilon_\cC$
as pseudo-abelian $\kk$-braided tensor category.
Next we study the relations between them.
Note that functoriality of the braiding implies that 
any diagram can pass under and jump over a string
(including the Reidemeister move of type~III):
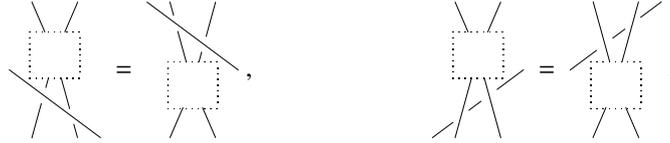
\begin{align*}
\begin{xy}
(0,2)*+{\phantom{\bigoplus}}*\frm{.}="p",
(-6,0);(6,-9) **\dir{-} ?!{(-3,-9);"p"} *{\hole}="x" ?!{(3,-9);"p"} *{\hole}="y",
(-3,9);"p" **\dir{-},
(3,9);"p" **\dir{-},
(-3,-9);"x" **\dir{-}, "p";"x" **\dir{-},
(3,-9);"y" **\dir{-}, "p";"y" **\dir{-},
\end{xy}
\;&=\;
\begin{xy}
(0,-2)*+{\phantom{\bigoplus}}*\frm{.}="p",
(-6,9);(6,0) **\dir{-} ?!{(-3,9);"p"} *{\hole}="x" ?!{(3,9);"p"} *{\hole}="y",
(-3,-9);"p" **\dir{-},
(3,-9);"p" **\dir{-},
(-3,9);"x" **\dir{-}, "p";"x" **\dir{-},
(3,9);"y" **\dir{-}, "p";"y" **\dir{-},
\end{xy}
\;,&
\begin{xy}
(0,2)*+{\phantom{\bigoplus}}*\frm{.}="p",
(-3,9);"p" **\dir{-},
(3,9);"p" **\dir{-},
(-3,-9);"p" **\dir{-} ?!{(6,0);(-6,-9)} *{\hole}="x",
(3,-9);"p" **\dir{-} ?!{(6,0);(-6,-9)} *{\hole}="y",
(6,0);"y" **\dir{-}, "x";"y" **\dir{-}, "x";(-6,-9) **\dir{-},
\end{xy}
\;&=\;
\begin{xy}
(0,-2)*+{\phantom{\bigoplus}}*\frm{.}="p",
(-3,9);"p" **\dir{-} ?!{(6,9);(-6,0)} *{\hole}="x",
(3,9);"p" **\dir{-} ?!{(6,9);(-6,0)} *{\hole}="y",
(-3,-9);"p" **\dir{-},
(3,-9);"p" **\dir{-},
(6,9);"y" **\dir{-}, "x";"y" **\dir{-}, "x";(-6,0) **\dir{-},
\end{xy}
\;,
\end{align*}
and of course we can also apply the Reidemeister move of type~II:
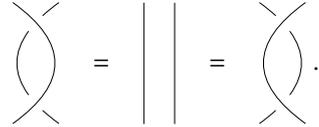
\begin{gather*}
\begin{xy}
(-3,8)="u1",(-3,-8)="v1",(3,8)="u2",(3,-8)="v2",(-8,0)="p1",(8,0)="p2",
(0,5.25)*{\hole}="x",(0,-5.25)*{\hole}="y",
"u1";"v1" **\crv{"p2"},
"u2"."x"."y".(-3,0);"v2" **\crv{"p1"},
"u2"."x";"v2"."y" **\crv{"p1"},
"u2";"v2"."x"."y".(-3,0) **\crv{"p1"},
\end{xy}
\quad=\quad
\begin{xy}
(-2,-8);(-2,8) **\dir{-},
(2,-8);(2,8) **\dir{-},
\end{xy}
\quad=\quad
\begin{xy}
(-3,8)="u1",(-3,-8)="v1",(3,8)="u2",(3,-8)="v2",(-8,0)="p1",(8,0)="p2",
(0,5.25)*{\hole}="x",(0,-5.25)*{\hole}="y",
"u2";"v2" **\crv{"p1"},
"u1"."x"."y".(3,0);"v1" **\crv{"p2"},
"u1"."x";"v1"."y" **\crv{"p2"},
"u1";"v1"."x"."y".(3,0) **\crv{"p2"},
\end{xy}\;.
\end{gather*}

In addition, we can transform diagrams along the local moves listed
in the next proposition.
The proof is easy and straightforward.
We prove later that these equations are enough
to define $\cS_t(\cC)$
by generators and relations.

\begin{proposition}
\label{prop:diagram_transform}
In $\cS_t(\cC)$, the morphisms
$\ag{\varphi}$, $\mu_\cC(U,V)$, $\iota_\cC$, $\Delta_\cC(U,V)$ and $\epsilon_\cC$
satisfy the equations below.
\begin{enumerate}
\item $\ag{\bullet}\colon\cC\to\cS_t(\cC)$ is a $\kk$-linear functor:
\begin{align*}
\begin{xy}
(0,8)="u",
(0,-8)="u'",
(0,0)*+{\id}="p"*\frm{-},
"u";"p" **\dir{-},
"u'";"p" **\dir{-},
\end{xy}\;&=\;
\begin{xy}
(0,8)="u",
(0,-8)="u'",
"u";"u'" **\dir{-},
\end{xy}\;,&\;
\begin{xy}
(0,8)="u",
(0,-8)="u'",
(0,3.5)*+{\varphi}="p"*\frm{-},
(0,-3.5)*+{\psi}="q"*\frm{-},
"u";"p" **\dir{-},
"p";"q" **\dir{-},
"u'";"q" **\dir{-},
\end{xy}\;&=\;
\begin{xy}
(0,8)="u",
(0,-8)="u'",
(0,0)*+{\psi\circ\varphi}="p"*\frm{-},
"u";"p" **\dir{-},
"u'";"p" **\dir{-},
\end{xy}\;,&\;
\begin{xy}
(0,8)="u",
(0,-8)="u'",
(0,0)*+{a\varphi+b\psi}="p"*\frm{-},
"u";"p" **\dir{-},
"u'";"p" **\dir{-},
\end{xy}\;&=\;
a\;\begin{xy}
(0,8)="u",
(0,-8)="u'",
(0,0)*+{\varphi}="p"*\frm{-},
"u";"p" **\dir{-},
"u'";"p" **\dir{-},
\end{xy}\;+\;
b\;\begin{xy}
(0,8)="u",
(0,-8)="u'",
(0,0)*+{\psi}="p"*\frm{-},
"u";"p" **\dir{-},
"u'";"p" **\dir{-},
\end{xy}\;.
\end{align*}
\item $\mu_\cC\colon\ag{\bullet}\otimes\ag{\bullet}\to\ag{\bullet\otimes\bullet}$
and $\Delta_\cC\colon\ag{\bullet\otimes\bullet}\to\ag{\bullet}\otimes\ag{\bullet}$
are both $\kk$-linear transformations:
\begin{align*}
\begin{xy}
(-5,7)="u",
(5,7)="v",
(-5,2)*+{\varphi}="p"*\frm{-},
(5,2)*+{\psi}="q"*\frm{-},
(0,-7)="uv",
(0,-4)="z",
"p";"u" **\dir{-},
"q";"v" **\dir{-},
"p";"z" **\crv{(-5,-3)},
"q";"z" **\crv{(5,-3)},
"uv";"z" **\dir{-},
\end{xy}&\;=\;
\begin{xy}
(-5,7)="u",
(5,7)="v",
(0,-2)*+{\varphi\otimes\psi}="pq"*\frm{-},
(0,-7)="uv",
(0,3)="z",
"u";"z" **\crv{(-5,4)},
"v";"z" **\crv{(5,4)},
"uv";"pq" **\dir{-},
"pq";"z" **\dir{-},
\end{xy}\;,&
\begin{xy}
(-5,-7)="u",
(5,-7)="v",
(-5,-2)*+{\varphi}="p"*\frm{-},
(5,-2)*+{\psi}="q"*\frm{-},
(0,7)="uv",
(0,4)="z",
"p";"u" **\dir{-},
"q";"v" **\dir{-},
"p";"z" **\crv{(-5,3)},
"q";"z" **\crv{(5,3)},
"uv";"z" **\dir{-},
\end{xy}&\;=\;
\begin{xy}
(-5,-7)="u",
(5,-7)="v",
(0,2)*+{\varphi\otimes\psi}="pq"*\frm{-},
(0,7)="uv",
(0,-3)="z",
"u";"z" **\crv{(-5,-4)},
"v";"z" **\crv{(5,-4)},
"uv";"pq" **\dir{-},
"pq";"z" **\dir{-},
\end{xy}\;.
\end{align*}
\item Associativity and coassociativity:
\begin{align*}
\begin{xy}
(0,6);(0,-6) **\dir{-},
(-5,6);(0,1) **\crv{(-5,2)},
(5,6);(0,-2) **\crv{(5,-1)},
\end{xy}&\;=\;
\begin{xy}
(0,6);(0,-6) **\dir{-},
(-5,6);(0,-2) **\crv{(-5,-1)},
(5,6);(0,1) **\crv{(5,2)},
\end{xy}\;,&
\begin{xy}
(0,-6);(0,6) **\dir{-},
(-5,-6);(0,-1) **\crv{(-5,-2)},
(5,-6);(0,2) **\crv{(5,1)},
\end{xy}&\;=\;
\begin{xy}
(0,-6);(0,6) **\dir{-},
(-5,-6);(0,2) **\crv{(-5,1)},
(5,-6);(0,-1) **\crv{(5,-2)},
\end{xy}\;.
\end{align*}
\item Unitality and counitality:
\begin{align*}
\begin{xy}
(0,6);(0,-6) **\dir{-},
(0,-2);(-5,3) **\crv{(-5,-1)} ?>*\dir{x},
\end{xy}\;\;=\;\;
\begin{xy}
(0,6);(0,-6) **\dir{-},
\end{xy}&\;\;=\;\;
\begin{xy}
(0,6);(0,-6) **\dir{-},
(0,-2);(5,4) **\crv{(5,-1)} ?>*\dir{x},
\end{xy}\;,&
\begin{xy}
(0,6);(0,-6) **\dir{-},
(0,2);(-5,-4) **\crv{(-5,1)} ?>*\dir{x},
\end{xy}\;\;=\;\;
\begin{xy}
(0,6);(0,-6) **\dir{-},
\end{xy}&\;\;=\;\;
\begin{xy}
(0,6);(0,-6) **\dir{-},
(0,2);(5,-4) **\crv{(5,1)} ?>*\dir{x},
\end{xy}\;.
\end{align*}
\item $\mu_\cC$ and $\Delta_\cC$ commute with braidings:
\begin{align*}
\begin{xy}
(-5,6)="u",
(5,6)="v",
(0,-6)="uv",
(0,-3)="z",
(0,2)="x",
"x";"u" **\dir{-},
"x"*{\hole};"v" **\dir{-},
"x"*{\hole};"z" **\crv{(-5,-2)},
"x";"z" **\crv{(5,-2)},
"uv";"z" **\dir{-},
\end{xy}&\;=\;
\begin{xy}
(-5,6)="u",
(5,6)="v",
(0,-2)*+{\sigma_\cC}="pq"*\frm{-},
(0,-6)="uv",
(0,2)="z",
"u";"z" **\crv{(-5,3)},
"v";"z" **\crv{(5,3)},
"uv";"pq" **\dir{-},
"pq";"z" **\dir{-},
\end{xy}\;,&
\begin{xy}
(-5,-6)="u",
(5,-6)="v",
(0,6)="uv",
(0,3)="z",
(0,-2)="x",
"x"*{\hole};"u" **\dir{-},
"x";"v" **\dir{-},
"x";"z" **\crv{(-5,2)},
"x"*{\hole};"z" **\crv{(5,2)},
"uv";"z" **\dir{-},
\end{xy}&\;=\;
\begin{xy}
(-5,-6)="u",
(5,-6)="v",
(0,2)*+{\sigma_\cC}="pq"*\frm{-},
(0,6)="uv",
(0,-2)="z",
"u";"z" **\crv{(-5,-3)},
"v";"z" **\crv{(5,-3)},
"uv";"pq" **\dir{-},
"pq";"z" **\dir{-},
\end{xy}\;.
\end{align*}
\item Compatibility between $\mu_\cC$ and $\Delta_\cC$:
\begin{align*}
\begin{xy}
(-6,7);(-2,-3) **\crv{(-6,-1)},
(2,7);(2,3) **\dir{-},
(2,3);(-2,-3) **\dir{-},
(-2,-7);(-2,-3) **\dir{-},
(6,-7);(2,3) **\crv{(6,1)},
\end{xy}\;&=\;
\begin{xy}
(-6,7);(0,2) **\crv{(-6,3)},
(2,7);(0,2) **\crv{(2,4)},
(0,2);(0,-2) **\dir{-},
(6,-7);(0,-2) **\crv{(6,-3)},
(-2,-7);(0,-2) **\crv{(-2,-4)},
\end{xy}\;,&\;
\begin{xy}
(6,7);(2,-3) **\crv{(6,-1)},
(-2,7);(-2,3) **\dir{-},
(-2,3);(2,-3) **\dir{-},
(2,-7);(2,-3) **\dir{-},
(-6,-7);(-2,3) **\crv{(-6,1)},
\end{xy}\;&=\;
\begin{xy}
(6,7);(0,2) **\crv{(6,3)},
(-2,7);(0,2) **\crv{(-2,4)},
(0,2);(0,-2) **\dir{-},
(-6,-7);(0,-2) **\crv{(-6,-3)},
(2,-7);(0,-2) **\crv{(2,-4)},
\end{xy}\;.
\end{align*}
\item $\mu_\cC$ is a retraction and $\Delta_\cC$ is a section:
\begin{gather*}
\begin{xy}
(0,6);(0,3) **\dir{-},
(0,3);(0,-3) **\crv{(5,0)} *\crv{(-5,0)},
(0,-6);(0,-3) **\dir{-},
\end{xy}\;=\;
\begin{xy}
(0,6);(0,-6) **\dir{-},
\end{xy}\;.
\end{gather*}
\item Quadratic relation on braidings:
\begin{align*}
\begin{xy}
0*{\hole}="p",
(-5,7);(5,-7) **\dir{-}, (5,7);"p" **\dir{-}, (-5,-7);"p" **\dir{-},
\end{xy}
\;-\;
\begin{xy}
0*{\hole}="p",
(5,7);(-5,-7) **\dir{-}, (-5,7);"p" **\dir{-}, (5,-7);"p" **\dir{-},
\end{xy}
\;=\;
\begin{xy}
0*+{\sigma_\cC-\sigma_\cC^{-1}}*\frm{-}="p",
"p"+(0,4)="p+",
"p"+(0,-4)="p-",
"p";"p+" **\dir{-},
"p";"p-" **\dir{-},
(5,7);"p+" **\dir{-}, (-5,7);"p+" **\dir{-}, (-5,-7);"p-" **\dir{-}, (5,-7);"p-" **\dir{-},
\end{xy}
\;.
\end{align*}
\item The object $\ag{\unit_\cC}$ is of dimension $t$:
\begin{align*}
\begin{xy}
(0,-2);(0,2) **\dir{-} ?<*\dir{x} ?>*\dir{x},
\end{xy}\;=t\id_{\unit_{\cS_t(\cC)}}.
\end{align*}
\end{enumerate}
\end{proposition}

Using these equations, we can easily calculate composites of morphisms.
Calculating tensor products is easier: it is nothing but arranging diagrams horizontally.
Note that the rank $t$ appears only when we remove isolated components
from diagrams using the last equation (9).

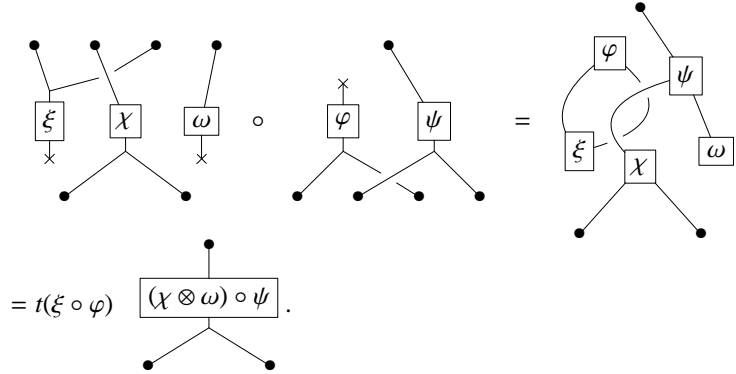
\begin{example}
\begin{align*}
&\mspace{21mu}
\begin{xy}
(-12,10)*={\bullet}="v1",
(-4,10)*={\bullet}="v2",
(4,10)*={\bullet}="v3",
(12,10)*={\bullet}="v4",
(-8,-10)*={\bullet}="w1",
(8,-10)*={\bullet}="w2",
(-10,0)*+{\xi}*\frm{-}="r",
(0,0)*+{\chi}*\frm{-}="s",
(10,0)*+{\omega}*\frm{-}="t",
"r"+(0,4)="r+",
"s"+(0,-4)="s-",
(-2.5,6)*{\hole}="y",
"r";"r+" **\dir{-},
"r";"r"+(0,-5) **\dir{-} ?>*\dir{x},
"s";"s-" **\dir{-},
"t";"t"+(0,-5) **\dir{-} ?>*\dir{x},
"v1";"r+" **\dir{-},
"v3";"y" **\dir{-}, "y";"r+" **\dir{-},
"v2";"s" **\dir{-},
"v4";"t" **\dir{-},
"w1";"s-" **\dir{-},
"w2";"s-" **\dir{-},
\end{xy}
\quad\circ\quad
\begin{xy}
(0,10)*={\bullet}="u1",
(-12,-10)*={\bullet}="v1",
(-4,-10)*={\bullet}="v2",
(4,-10)*={\bullet}="v3",
(12,-10)*={\bullet}="v4",
(-6,0)*+{\varphi}*\frm{-}="p",
(6,0)*+{\psi}*\frm{-}="q",
"p"+(0,-4)="p-",
"q"+(0,-4)="q-",
(0,-7.5)*{\hole}="x",
"p";"p"+(0,5) **\dir{-} ?>*\dir{x},
"p";"p-" **\dir{-},
"q";"q-" **\dir{-},
"u1";"q" **\dir{-},
"v2";"q-" **\dir{-},
"v4";"q-" **\dir{-},
"v1";"p-" **\dir{-},
"v3";"x" **\dir{-}, "x";"p-" **\dir{-},
\end{xy}
\quad=\quad
\begin{xy}
(0,15)*={\bullet}="u1",
(-8,-15)*={\bullet}="w1",
(8,-15)*={\bullet}="w2",
(-4,9)*+{\varphi}*\frm{-}="p",
(6,6)*+{\psi}*\frm{-}="q",
(-8,-4)*+{\xi}*\frm{-}="r",
(0,-6)*+{\chi}*\frm{-}="s",
(10,-4)*+{\omega}*\frm{-}="t",
(0.3,4.15)*{\hole}="x",
(-3.3,-2.5)*{\hole}="y",
"u1";"q" **\dir{-},
"q";"t" **\dir{-},
"p";"r"."x" **\crv{(8,0)}, "p"."x";"r"."y" **\crv{(8,0)},
"p"*={}."y".(3,0);"r" **\crv{(8,0)},
"p";"r" **\crv{(-14,2)},
"q";"s" **\crv{(-10,2)},
"w1";"s" **\dir{-},
"w2";"s" **\dir{-},
\end{xy}\\
&=t(\xi\circ\varphi)\quad
\begin{xy}
(0,8)*={\bullet}="u1",
(-8,-8)*={\bullet}="w1",
(8,-8)*={\bullet}="w2",
(0,1)*+{(\chi\otimes\omega)\circ\psi}*\frm{-}="c",
"c"+(0,-4)="c-",
"c";"c-" **\dir{-},
"u1";"c" **\dir{-},
"w1";"c-" **\dir{-},
"w2";"c-" **\dir{-},
\end{xy}\;.
\end{align*}
\end{example}

Notice that $\ag{\bullet}\colon\cC\to\cS_t(\cC)$ is
a $\kk$-linear functor between $\kk$-braided tensor categories
but not a $\kk$-braided tensor functor.
In fact, the conditions (1)-(5) is almost same as the definition of braided tensor functor
but the only difference is that
they do not require that $\mu_\cC$ and $\Delta_\cC$, $\iota_\cC$ and $\epsilon_\cC$ are inverse to each other.
With this fact in mind,
we define weaker notions of tensor functors and transformations.
\begin{definition}
Let $\cC$ and $\cD$ be $\kk$-tensor categories.
\begin{enumerate}
\item A $\kk$-linear functor $F\colon\cC\to\cD$ is called a \term{$\kk$-Frobenius functor}
if it is endowed with $\kk$-linear transformations
\begin{align*}
\mu_F&\colon F(\bullet)\otimes F(\bullet)\to F(\bullet\otimes\bullet),&
\Delta_F&\colon F(\bullet\otimes\bullet)\to F(\bullet)\otimes F(\bullet),\\
\iota_F&\colon\unit_\cD\to F(\unit_\cC),&
\epsilon_F&\colon F(\unit_\cC)\to\unit_\cD
\end{align*}
which are associative, unital, coassociative and counital
(see Definition~\ref{def:tenfun}~(1)),
and satisfies the compatibility conditions
in Proposition~\ref{prop:diagram_transform}~(6), i.e.\
\begin{align*}
\Delta_F(U\otimes V,W)\circ\mu_F(U,V\otimes W)&=
(\mu_F(U,V)\otimes\id_W)\circ(\id_U\otimes\Delta_F(V,W)),\\
\Delta_F(U,V\otimes W)\circ\mu_F(U\otimes V,W)&=
(\id_U\otimes\mu_F(V,W))\circ(\Delta_F(U,V)\otimes\id_W).
\end{align*}
The scalar $\epsilon_F\circ\iota_F\in\End_\cC(\unit_\cC)\simeq\kk$
is called the \term{dimension} of $F$ and denoted by $\dim F$.
\item A $\kk$-Frobenius functor $F$ is called
\term{separable} if
$\mu_F$ is a retraction and $\Delta_F$ is a section, i.e.\
$\mu_F(U,V)\circ\Delta_F(U,V)=\id_{U\otimes V}$.
\item A \term{$\kk$-Frobenius transformation} $\eta\colon F\to G$
between two Frobenius functors
is a $\kk$-linear transformation
such that both the diagrams in Definition~\ref{def:tenfun}~(2)
and their dual commute.
\end{enumerate}
\end{definition}
\begin{definition}
Let $\cC$ and $\cD$ be $\kk$-braided tensor categories.
\begin{enumerate}
\item A \term{$\kk$-braided Frobenius functor} $F\colon\cC\to\cD$ is
a $\kk$-Frobenius functor
such that $\mu_F$ and $\Delta_F$ commute with braidings.
See Definition~\ref{def:braid}~(3).
\item A $\kk$-braided Frobenius functor $F$ is called \term{quadratic}
if it satisfies the quadratic relation
\begin{gather*}
\sigma_\cD(F(U),F(V))-\sigma^{-1}_\cD(F(U),F(V))
=\Delta_F(V,U)\circ(\sigma_\cC(U,V)-\sigma^{-1}_\cC(U,V))\circ\mu_F(U,V).
\end{gather*}
\item A \term{$\kk$-braided Frobenius transformation}
is just a $\kk$-Frobenius transformation
between two $\kk$-braided Frobenius functors.
\end{enumerate}
\end{definition}

Thus Proposition~\ref{prop:diagram_transform} just says that
$\ag{\bullet}\colon\cC\to\cS_t(\cC)$ is a
$\kk$-braided Frobenius functor which is
separable, quadratic, and of dimension $t$.
Obviously an usual $\kk$-braided tensor functor
is also but of dimension $1$.
Note that $\kk$-braided Frobenius functors
are closed under composition and
the properties listed above are preserved.
In addition, $\dim(G\circ F)=\dim F\dim G$.

\begin{remark}
Frobenius functors, usually called Frobenius monoidal functors,
were introduced and studied
by Szlach\'anyi~\cite{Szlachanyi:2001,Szlachanyi:2005}, Day and Pastro~\cite{DayPastro:2008}.
Notice that McCurdy and Street~\cite{McCurdyStreet:2010} require
a stronger relation
\begin{gather*}
\sigma_\cD(F(U),F(V))
=\Delta_F(V,U)\circ\sigma_\cC(U,V)\circ\mu_F(U,V).
\end{gather*}
in their definition of the term ``braided'' on separable Frobenius functors than ours.
\end{remark}

Now we state the universal property of $\cS_t(\cC)$.
That is, $\cS_t(\cC)$ is the smallest category
which has generators and satisfies relations
as in Proposition~\ref{prop:diagram_transform}.
Let us denote by $\HOM_\kk^\mathbf{B}(\cC,\cD)$
(resp.\ $\HOM_\kk^\mathbf{BF}(\cC,\cD)$)
the category of $\kk$-braided tensor (resp.\ Frobenius) functors
and transformations.

\begin{theorem}
\label{thm:universality}
Let $\cC,\cD$ be $\kk$-braided tensor categories
and assume that $\cD$ is pseudo-abelian.
\begin{enumerate}
\item The natural functor
\begin{gather*}
\HOM_\kk^\mathbf{B}(\cS_t(\cC),\cD)
\xrightarrow{\circ\ag{\bullet}}
\HOM_\kk^\mathbf{BF}(\cC,\cD)
\end{gather*}
is fully faithful.
\item For $F\in\HOM_\kk^\mathbf{BF}(\cC,\cD)$,
there exists $\tilde F\in\HOM_\kk^\mathbf{B}(\cS_t(\cC),\cD)$
such that $F\simeq\tilde F\circ\ag{\bullet}$ as $\kk$-braided Frobenius functors
if and only if $F$ is separable, quadratic, and of dimension $t$.
\end{enumerate}
\end{theorem}
\begin{proof}
(1) Let $\tilde F,\tilde G\colon\cS_t(\cC)\to\cD$ be $\kk$-braided tensor functors
and put $F\coloneqq\tilde F\circ\ag{\bullet}$, $G\coloneqq\tilde G\circ\ag{\bullet}$.
We have to show that the map between the sets of transformations
\begin{align*}
\Hom_{\HOM_\kk^\mathbf{B}(\cS_t(\cC),\cD)}(\tilde F,\tilde G)
&\to\Hom_{\HOM_\kk^\mathbf{BF}(\cC,\cD)}(F,G)\\
\tilde\eta&\mapsto\eta
\end{align*}
defined by $\eta(U)=\tilde\eta(\ag{U})$ is bijective.

By the definition of $\kk$-tensor transformation,
the map $\tilde\eta(\ag{U_1}\otimes\dots\otimes\ag{U_m})$
is determined by each $\tilde\eta(\ag{U_i})=\eta(U_i)$.
Thus this map is injective.
Conversely, for each $\kk$-braided Frobenius transformation $\eta\colon F\to G$,
we can define $\tilde\eta\colon\tilde F\to\tilde G$ at each objects in $\cS_t(\cC)$ as above.
We can show easily that $\tilde\eta$ commute with all the morphisms in $\cS_t(\cC)$;
so $\tilde\eta$ is actually a transformation whose restriction is equal to $\eta$.
Thus this map is also surjective.

(2) The ``only if'' part is obvious, so we prove the ``if'' part.
Let us take a $\kk$-braided Frobenius functor $F\colon\cC\to\cD$
which is separable, quadratic and of dimension $t$.
First we define $\tilde F$ for objects $\ag{U_1}\otimes\dots\otimes\ag{U_m}$ by
\begin{gather*}
\tilde F(\ag{U_1}\otimes\dots\otimes\ag{U_m})\coloneqq F(U_1)\otimes\dots\otimes F(U_m).
\end{gather*}
The map for morphisms is determined by
$\tilde F(\mu_\cC)\coloneqq\mu_F$, $\tilde F(\iota_\cC)\coloneqq\iota_F$ etc;
since all morphisms in $\cS_t(\cC)$ are generated by them.
By taking its pseudo-abelian envelope, we can extend its domain to the whole objects in $\cS_t(\cC)$.

To prove its well-definedness, we have to show that
a linear combination of diagrams which represents
a zero morphism in $\cS_t(\cC)$ is also zero in $\cD$.
Here we also use diagrams to denote morphisms in $\cD$ which are came from $\cC$ via $F$.
By Proposition~\ref{prop:standardform}
it suffices to show that
every diagram can be transformed into a linear combination of diagrams of standard form
using the relations listed in Proposition~\ref{prop:diagram_transform} only.
First we state the next lemma.
\begin{lemma}
If two strings in left-hand sides below are connected,
\begin{align*}
\begin{xy}
(0,0)*{\hole}="p",
(-5,7);(5,-7) **\dir{-}, (5,7);"p" **\dir{-}, (-5,-7);"p" **\dir{-},
\end{xy}
\;&=\;
\begin{xy}
0*+{\sigma_\cC}*\frm{-}="p",
"p"+(0,4)="p+",
"p"+(0,-4)="p-",
"p";"p+" **\dir{-},
"p";"p-" **\dir{-},
(5,7);"p+" **\dir{-}, (-5,7);"p+" **\dir{-}, (-5,-7);"p-" **\dir{-}, (5,-7);"p-" **\dir{-},
\end{xy}
\;,&
\begin{xy}
0*{\hole}="p",
(5,7);(-5,-7) **\dir{-}, (-5,7);"p" **\dir{-}, (5,-7);"p" **\dir{-},
\end{xy}
\;&=\;
\begin{xy}
0*+{\sigma_\cC^{-1}}*\frm{-}="p",
"p"+(0,4)="p+",
"p"+(0,-4)="p-",
"p";"p+" **\dir{-},
"p";"p-" **\dir{-},
(5,7);"p+" **\dir{-}, (-5,7);"p+" **\dir{-}, (-5,-7);"p-" **\dir{-}, (5,-7);"p-" **\dir{-},
\end{xy}
\;.
\end{align*}
\end{lemma}
\begin{proof}
It suffices to prove the first equation.
By the assumption we can find a loop connecting the two strings.
The shape of the loop looks like either of the diagrams below
depending on whether the loop contains the other side of the crossing or not:
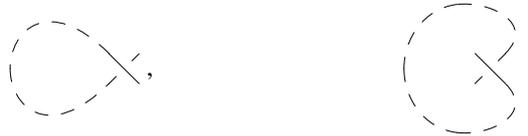
\begin{align*}
\relax{\begin{xy}
(2,-2);(-2,2) **\dir{-}, (-2,-2);0*{\hole} **\dir{-}, (2,2);0*{\hole} **\dir{-},
(-2,2);(-2,-2) **\crv{~**\dir{--}(-10,10)&(-20,0)&(-10,-10)}
\end{xy}}\;,&&
\relax{\begin{xy}
(2,-2);(-2,2) **\dir{-}, (-2,-2);0*{\hole} **\dir{-}, (2,2);0*{\hole} **\dir{-},
(2,2);(2,-2) **\crv{~**\dir{--}(6,6)&(-6,12)&(-14,0)&(-6,-12)&(6,-6)}
\end{xy}}\;.
\end{align*}
We prove the equation by the induction on sizes of loops.
So we may assume that the loop has no short circuits
and other self-crossings.
To prove the equation we can reverse crossings in the loop freely since
the right-hand side of the relation (8) makes smaller loops.
So we can remove all unconnected strings from the diagram.
In addition, the crossing in the loop of second type above
can be moved to the outside of the loop
since the strings in the other side of the crossing are not connected to the loop:
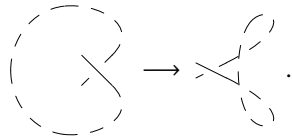
\begin{gather*}
\relax{\begin{xy}
(2,-2);(-2,2) **\dir{-}, (-2,-2);0*{\hole} **\dir{-}, (2,2);0*{\hole} **\dir{-},
(2,2);(2,-2) **\crv{~**\dir{--}(6,6)&(-6,12)&(-14,0)&(-6,-12)&(6,-6)}
\end{xy}}\;\longrightarrow\;
\relax{\begin{xy}
(2,-1);(-2,1) **\dir{-}, (-2,-1);0*{\hole} **\dir{-}, (2,1);0*{\hole} **\dir{-},
(2,1);(2,-1) **\crv{~**\dir{--}(12,6)&(3,12)&(3,-12)&(12,-6)}
\end{xy}}.
\end{gather*}
Thus we may assume that the loop is of first type.

If there is a string in the loop,
by the assumptions the string is connected to the loop at only one point.
If this string has a crossing with the loop,
by the hypothesis of the induction
we can apply the lemma to this crossing and we get a smaller loop.
Otherwise we can flip it to the outside using (5):
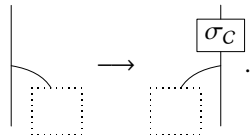
\begin{gather*}
\begin{xy}
(0,8);(0,-8) **\dir{-},
(0,0);(6,-6)*+{\phantom{\bigoplus}}*\frm{.} **\crv{(6,-1)},
\end{xy}\;\longrightarrow\;
\begin{xy}
(0,4)*+{\sigma_\cC}*\frm{-}="p",
(0,8);"p" **\dir{-},
(0,-8);"p" **\dir{-},
(0,0);(-6,-6)*+{\phantom{\bigoplus}}*\frm{.} **\crv{(-6,-1)},
\end{xy}.
\end{gather*}
Thus we may assume that there is no strings in the loop.
We can remove extra parts on the loop by using (2), (3) and (6).
So it suffices to prove the equation in two special cases below:
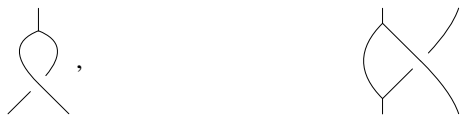
\begin{align*}
\begin{xy}
(0,-3)="p",
(4,-7);"p" **\dir{-}, (-4,-7);"p"*{\hole} **\dir{-},
"p"*{\hole};(0,4) **\crv{(5,2)},
"p";(0,4) **\crv{(-5,2)},
(0,7);(0,4) **\dir{-},
\end{xy}\;,&&
\begin{xy}
(-1,0)="p",
(-6,5);"p" **\dir{-}, (4,-7);"p" **\crv{(3,-4)},
(-6,-5);"p"*{\hole} **\dir{-}, (4,7);"p"*{\hole} **\crv{(3,4)},
(-6,5);(-6,-5) **\crv{(-11,0)},
(-6,5);(-6,7) **\dir{-},
(-6,-5);(-6,-7) **\dir{-},
\end{xy}\;.
\end{align*}
The proof is easy and we left it to the reader.

\end{proof}

Let us continue the proof of the theorem.
First take an arbitrary connected diagram.
Using this lemma, we can remove all crossings from the diagram
and we get a planar diagram.
If the diagram has extra $\iota_\cC$'s and $\epsilon_\cC$'s
we can put them together to other strings
using the lemma and the relation (4).
By removing all bubbles using (6),
we get a tree diagram which has no extra endpoints.
If the diagram represents a morphism $\unit_{\cS_t(\cC)}\to\unit_{\cS_t(\cC)}$,
we can transform it into a scalar by (9).
Otherwise we can move all $\mu_\cC$'s to the top of the diagram
and $\Delta_\cC$'s to the bottom;
then we obtain a diagram of standard form.

Next we prove this for any diagram which has more than two connected components
by the induction on the number of them.
For such a diagram, first we reverse some crossings using (8)
so that the connected components are
totally ordered from the back of the paper to the front.
Because the number of the connected components of right-hand side of (8) is
less than that of left-hand side,
we can apply the hypothesis of the induction to
the difference between them.
Then we can transform each connected component to standard form in the manner described above.
Reversing some crossings again, we get a diagram of standard form.
\end{proof}

\begin{remark}
Let $\cC$ be a $\kk$-tensor category
and consider the subcategory $\TL_t(\cC)$
of $\cS_t(\cC)$
whose objects are generated by $\ag{U_1}\otimes\dots\otimes\ag{U_m}$
for all $U_1,\dots,U_m\in\cC$
and morphisms between them are $\kk$-linear combinations of ``non-crossing'' diagrams,
i.e.\ composites of $\ag{\varphi}$, $\mu_\cC$, $\iota_\cC$, $\Delta_\cC$ and $\epsilon_\cC$.
This $\kk$-tensor category is a ``$\cC$-colored'' version
of so-called Temperley--Lieb category~\cite{Freedman:2003}
and satisfies the same universality as in Theorem~\ref{thm:universality}
with respect to separable $\kk$-Frobenius functors of dimension $t$.
The important difference between $\cS_t$ and $\TL_t$
is that we can naturally apply $\TL_t$ to any $\kk$-linear bicategories,
in other words, $\kk$-tensor categories with several 0-cells.
\end{remark}

\section{Classification of Indecomposable Objects}
In this section we assume that
$k$ is a field of characteristic zero.
The purpose of this section is to explain
the structure of our category $\cS_t(\cC)$.

\subsection{For Deligne's category}
Let us denote Deligne's category $\cS_t(\Rep(\kk))$ by $\cD_t$.
We review here the result of
Comes and Ostrik~\cite{ComesOstrik:2011}
which describes the complete classification of indecomposable objects
in $\cD_t$.

For $m\in\NN$, we use the same symbol $m$ to denote the family of objects $(\unit_\kk)_{i=1}^m$
which contains the trivial representation $\unit_\kk$ by multiplicity $m$
so that we can write an object in $\cD_t$ as $\ag{m}$.
Let us denote by $E_{t,m}$ the $\kk$-algebra $\End_{\cD_t}(\ag{m})$.
It is the direct sum of $\bigaag{H_r(m;m)}$
for all recollements $r\in R(m,m)$ and each $H_r(m;m)$ is one-dimensional.
\begin{lemma}
Let $m\in\NN$ and put
$A\coloneqq\bigaag{H^m(m;m)}$,
$I\coloneqq\bigaag{H^{>m}(m;m)}$.
Then
\begin{enumerate}
\item $E_{t,m}=A\oplus I$ as a $\kk$-module,
\item $A$ is a $\kk$-subalgebra of $E_{t,m}$ isomorphic to $\kk[\fS_m]$,
\item $I$ is a two-sided ideal of $E_{t,m}$.
\end{enumerate}
Thus $E_{t,m}/I\simeq\kk[\fS_m]$ as a $\kk$-algebra.
\end{lemma}
\begin{proof}
(1) and (2) are obvious. (3) follows from Corollary~\ref{cor:ideal}.
\end{proof}

We recall here some facts about representations of symmetric groups
in characteristic zero. For details, see e.g.\ \cite{Fulton:1997}.
A Young diagram $\lambda=(\lambda_1,\lambda_2,\dots)$ is a non-increasing
sequence of natural numbers such that all but finitely many entries are zero.
We call $|\lambda|\coloneqq\sum_i\lambda_i$ the size of $\lambda$
and denote by $\varnothing=(0,0,\dots)$ the unique Young diagram of size zero.
We denote by $\cP$ the set of all Young diagrams
and by $\cP_m$ the set of those with size $m$.
There is a one to one correspondence
\begin{gather*}
\cP_m\bijection\{\text{irreducible representations of }\fS_m\}
\end{gather*}
and we denote by $S_\lambda$ the
irreducible representation of $\fS_m$
corresponding to $\lambda\in\cP_m$.

For each $\lambda\in\cP_m$, the $\kk[\fS_m]$-module $S_\lambda$
can be regarded as an $E_{t,m}$-module via the map $E_{t,m}\twoheadrightarrow\kk[\fS_m]$.
Its projective cover $P(S_\lambda)$ is isomorphic to $E_{t,m}$-module of the form
$E_{t,m} e_{t,\lambda}$ where $e_{t,\lambda}\in E_{t,m}$ is some primitive idempotent.
Then its image  $L_{t,\lambda}\coloneqq e_{t,\lambda}\ag{m}\in\cD_t$ is indecomposable and 
well-defined up to isomorphism.

\begin{remark}
In \cite{ComesOstrik:2011}, $L_{t,\lambda}$ is defined
as a direct summand of $\ag{1}^{\otimes m}$,
not $\ag{m}$.
\end{remark}

For a Krull--Schmidt $\kk$-linear category $\cC$,
we denote by $I(\cC)$ the set of isomorphism classes
of indecomposable objects in $\cC$.
For $U,V\in I(\cC)$, we say $U$ and $V$ \term{are in the same block} if
there exists a chain of indecomposable objects $U=U_0,U_1,\dots,U_m=V\in I(\cC)$ such that
either $\Hom_\cC(U_{i-1},U_i)$ or $\Hom_\cC(U_i,U_{i-1})$ is non-zero for each $i=1,\dots,m$.
We also use the term \term{block} to refer each pseudo-abelian full subcategory of
$\cC$ generated by all indecomposable objects in a same block.
A block is called \term{trivial} if it is equivalent to $\Rep(\kk)$.
Note that such a category is equivalent to the direct sum of all its blocks.

\begin{theorem}[Deligne~\cite{Deligne:2007}, Comes--Ostrik~\cite{ComesOstrik:2011}]
\label{thm:comes_ostrik}
\begin{enumerate}
\item $\lambda\mapsto L_{t,\lambda}$ gives a bijection
$\cP\bijectivemap I(\cD_t)$.
\item If $t\not\in\NN$ then all blocks in $\cD_t$ are trivial.
\item For $d\in\NN$,
non-trivial blocks in $\cD_d$
are parameterized by Young diagrams of size $d$.
For $\lambda\in\cP_d$, let us define
$\lambda^{(j)}=(\lambda^{(j)}_1,\lambda^{(j)}_2,\dots)\in\cP$ by
\begin{align*}
\lambda^{(j)}_i=\begin{dcases*}
\lambda_i+1,&if $1\leq i\leq j$,\\
\lambda_{i+1},&otherwise.
\end{dcases*}
\end{align*}
Then $L_{d,\lambda^{(0)}},L_{d,\lambda^{(1)}},\dots$ generate a block in $\cD_d$ and
all non-trivial blocks are obtained by this construction.
Morphisms between them are spanned by
\begin{align*}
\xymatrix{
L_{d,\lambda^{(0)}}\ar@(ul,ur)[]^{\id}\ar@<2pt>[r]^{\alpha_0}&
L_{d,\lambda^{(1)}}\ar@(ul,ur)[]^{\id}\ar@(dr,dl)[]^{\gamma_1}\ar@<2pt>[r]^{\alpha_1}\ar@<2pt>[l]^{\beta_0}&
L_{d,\lambda^{(2)}}\ar@(ul,ur)[]^{\id}\ar@(dr,dl)[]^{\gamma_2}\ar@<2pt>[r]^{\alpha_2}\ar@<2pt>[l]^{\beta_1}&
\cdots\ar@<2pt>[l]^{\beta_2}
}
\end{align*}
where $\beta_n\alpha_n=\alpha_{n-1}\beta_{n-1}=\gamma_n$ for $n\geq1$
and other non-trivial composites are zero.
The canonical functor $\cD_d\to\Rep(\kk[\fS_d])$ sends $L_{d,\lambda^{(0)}}$ to
$S_\lambda$ for each $\lambda\in\cP_m$
and the other indecomposable objects to the zero object.
\end{enumerate}
\end{theorem}

\subsection{Direct sum of Categories}

Let $\cC$ be a pseudo-abelian $\kk$-linear category with unit.
Assume that $\cC$ admits a direct sum decomposition
$\cC\simeq\bigoplus_{x\in X}\cC_x$
with index set $X$ (e.g.\ by blocks).
There is a unique $\cC_x$ which contains the unit object $\unit_\cC$
so let us denote its index by $x=0$
and put $X'\coloneqq X\setminus\{0\}$.

Recall that we have two kinds of $*$-product
\begin{align*}
\cS_{t_1}(\cC)\barboxtimes\cW_{d_2}(\cC)&\to\cS_{t_1+d_2}(\cC),&
\cW_{d_1}(\cC)\barboxtimes\cW_{d_2}(\cC)&\to\cW_{d_1+d_2}(\cC)\\
\ag[t_1]{U_I}*\br[d_2]{W_K}&\coloneqq\ag[t_1+d_2]{U_I\sqcup W_K},&
\br[d_1]{V_J}*\br[d_2]{W_K}&\coloneqq\br[d_1+d_2]{V_J\sqcup W_K}
\end{align*}
defined for objects which satisfy $\#J=d_1$ and $\#K=d_2$.
By definition, as pseudo-abelian $\kk$-linear category,
$\cS_t(\cC)$ is generated by objects of the form $\ag{U_I}$
where for each $i\in I$ its component $U_i$ is in some $\cC_{x_i}$.
For such a family we write $I_x\coloneqq\{i\in I\,|\,x_i=x\}$
and denote by $U_{I_x}$ the subfamily of $U_I$ indexed by $I_x$.
Then we can write
\begin{gather*}
\ag{U_I}\simeq\ag[t_0]{U_{I_0}}*\prod_{x\in X'}\br[d_x]{U_{I_x}}
\end{gather*}
using the $*$-product.
Here $d_x\coloneqq\#I_x$, $t_0\coloneqq t-\sum_{x\in X'}d_x$ and
$\prod$ denotes the $*$-product of finite terms for $x\in X'$ with $I_x\neq\varnothing$.

Let $U_I$ and $V_J$ be families of objects
of such form.
By the assumptions $\Hom_\cC(\unit_\cC,W)\simeq0\simeq\Hom_\cC(W,\unit_\cC)$
for all $W\in\cC_x$ when $x\neq0$.
So in the direct sum
\begin{gather*}
\Hom_{\cS_t(\cC)}(\ag{U_I},\ag{V_J})\simeq
\bigoplus_{r\in R(I,J)}
H_r(U_I;V_J)
\end{gather*}
we only need recollements $r\in R(I,J)$
all whose components $(i,j)\in r$
satisfy one of the conditions below:
\begin{gather*}
\begin{dcases*}
i,j\neq\varnothing\text{ and }x_i=x_j,\\
i=\varnothing\text{ and }x_j=0,\\
j=\varnothing\text{ and }x_i=0.
\end{dcases*}
\end{gather*}
Thus $\Hom_{\cS_t(\cC)}(\ag{U_I},\ag{V_J})=0$
unless $\#I_x=\#J_x$ for all $x\in X'$.
Otherwise
\begin{gather*}
\Hom_{\cS_t(\cC)}(\ag{U_I},\ag{V_J})\simeq
H(U_{I_0};V_{J_0})\otimes
\bigotimes_{x\in X'}H'(U_{I_x};V_{J_x})
\end{gather*}
where for each $U_{I'}=(U_1,\dots,U_d)$ and $V_{J'}=(V_1,\dots,V_d)$,
\begin{gather*}
H'(U_{I'};V_{J'})\coloneqq\bigoplus_{g\in\fS_d}
\Hom_\cC(U_1,V_{g(1)})\otimes\dots\otimes\Hom_\cC(U_d,V_{g(d)})
\end{gather*}
which is isomorphic to $\Hom_{\cW_d(\cC)}(\br{U_{I'}},\br{V_{J'}})$.
The same arguments also hold for $\cW_d(\cC)$ and we have following equivalences of $\kk$-linear category.

\begin{proposition}
\label{prop:St_for_direct_sum}
Let $\cC$ be a $\kk$-linear category which admits
a decomposition $\cC\simeq\bigoplus_{x\in X}\cC_x$.
Then the $*$-product induces a category equivalence
\begin{gather*}
\bigoplus_{\substack{d_x\in\NN\\d=\sum_{x\in X}d_x}}
\Bigl(\bigbarboxtimes_{x\in X}\cW_{d_x}(\cC_x)
\Bigr)\eqmap\cW_d(\cC).
\end{gather*}
In addition, assume that $\cC$ has the unit $\unit_\cC\in\cC_0$.
Put $X'\coloneqq X\setminus\{0\}$. Then
we have another equivalence
\begin{gather*}
\bigoplus_{\substack{t_0\in\kk,d_x\in\NN\\t=t_0+\sum_{x\in X'}d_x}}
\Bigl(\cS_{t_0}(\cC_0)\barboxtimes
\bigbarboxtimes_{x\in X'}\cW_{d_x}(\cC_x)
\Bigr)\eqmap\cS_t(\cC)
\end{gather*}
also induced by $*$-product.
\end{proposition}

For example,
let us consider the case when $\cC$ is a
hom-finite pseudo-abelian $\kk$-linear category
whose unit object $\unit_\cC\in\cC$ has no extension,
i.e.\ is in a trivial block $\Rep(\kk)\subset\cC$.
So there is a pseudo-abelian full subcategory $\cC'\subset\cC$ such that
$\cC\simeq\Rep(\kk)\oplus\cC'$.
By applying the proposition, we have
\begin{gather*}
\cS_t(\cC)\simeq
\bigoplus_{d\in\NN}\,\Bigl(
\cD_{t-d}\barboxtimes
\cW_d(\cC')
\Bigr).
\end{gather*}
Let us take indecomposable objects $L\in\cD_{t-d}$
and $U\in\cW_d(\cC')$ respectively
and consider $L*U\in\cS_t(\cC)$.
By Theorem~\ref{thm:comes_ostrik}, $\End_{\cD_{t-d}}(L)$ is
isomorphic to either $\kk$ or $\kk[\gamma]/(\gamma^2)$.
Thus its endomorphism ring
\begin{gather*}
\End_{\cS_t(\cC)}(L*U)\simeq\End_{\cD_{t-d}}(L)\otimes\End_{\cW_d(\cC')}(U)
\end{gather*}
is still local and $L*U$ is also an indecomposable object.
By Theorem~\ref{thm:krull_schmidt},
all indecomposable objects in $\cS_t(\cC)$ is of this form
and each block in $\cS_t(\cC)$ is therefore
equivalent to a tensor product of
two blocks in $\cD_{t-d}$ and $\cW_d(\cC')$ respectively.

\subsection{For Semisimple category}

A hom-finite pseudo-abelian $\kk$-linear category $\cC$
is called \term{semisimple} if every non-zero morphism
between indecomposable objects in $\cC$ is
an isomorphism, or equivalently,
if the endomorphism ring of each object in $\cC$
is a finite dimensional semisimple $\kk$-algebra.
We state a simple criterion
for semisimplicity of $\cS_t(\cC)$.

\begin{proposition}
Let $\cC$ be a hom-finite pseudo-abelian $\kk$-linear category
with unit.
Then $\cS_t(\cC)$ is semisimple if and only if
$t\not\in\NN$ and $\cC$ is semisimple.
\end{proposition}
\begin{proof}
If $t\in\NN$, $\cS_t(\cC)$ contains a non-semisimple full subcategory
$\cD_t$ so $\cS_t(\cC)$ itself is not semisimple.
If $\cC$ is not semisimple,
there are indecomposable objects $U_1,U_2\in I(\cC)$ 
and non-zero morphism $\varphi\colon U_1\to U_2$
which is not invertible.
For $i=1,2$, we have a $\kk$-algebra homomorphism
$\End_{\cS_t(\cC)}(\ag{U_i})\twoheadrightarrow\End_\cC(U_i)$.
By taking its projective cover,
we obtain an idempotent $e_i\in\End_{\cS_t(\cC)}(\ag{U_i})$
such that its image $e_i\ag{U_i}$ is indecomposable
and $e_2\ag{\varphi}e_1\colon e_1\ag{U_1}\to e_2\ag{U_2}$
is not zero or invertible.
Thus $\cS_t(\cC)$ is not semisimple either in this case.

Conversely assume that $t\not\in\NN$ and $\cC$ is semisimple.
Then $\cC\simeq\Rep(\kk)\oplus\cC'$
for some semisimple full subcategory $\cC'\subset\cC$.
Since semisimplicity of $\kk$-algebra is preserved under
tensor products and wreath products in characteristic zero,
we have that $\cS_t(\cC)$ is also semisimple
by Proposition~\ref{prop:St_for_direct_sum}
and Theorem~\ref{thm:comes_ostrik}~(2).
\end{proof}

Now assume that $\cC$ is semisimple
and all blocks are trivial, i.e.,
every indecomposable object $U\in I(\cC)$ satisfies
$\End_\cC(U)\simeq\kk$.
We give a complete description of the $\kk$-linear category $\cS_t(\cC)$
for this case parallel to Theorem~\ref{thm:comes_ostrik}.

Let $\cP^\cC$ be the set
\begin{gather*}
\cP^\cC\coloneqq\{\lambda\colon I(\cC)\to\cP\;|\;\lambda(U)=\varnothing\text{ for all but finitely many }U\}.
\end{gather*}
For each $\lambda\in\cP^\cC$, we write $|\lambda|\coloneqq\sum_U|\lambda(U)|$
and $|\lambda|'\coloneqq\sum_{U\neq\unit_\cC}|\lambda(U)|$.
For each $d\in\NN$, put
$\cP^\cC_d\coloneqq\{\lambda\in\cP^\cC\;|\;|\lambda|=d\}$.

Take an idempotent $f_\lambda\in\kk[\fS_d]$
for each $\lambda\in\cP_d$
which satisfies $S_\lambda\simeq\kk[\fS_d]f_\lambda$.
For $U\in I(\cC)$, since $\End_{\cW_d(\cC)}(U^{*d})\simeq\kk[\fS_d]$,
we can define the object $U^{\boxtimes\lambda}\in I(\cW_d(\cC))$ by
$U^{\boxtimes\lambda}\coloneqq f_\lambda U^{*d}$.
Let
\begin{align*}
\overline{L}_{t,\lambda}&\coloneqq L_{t-|\lambda|',\lambda(\unit_\cC)}
*\prod_{U\neq\unit_\cC}U^{\boxtimes\lambda(U)}\in\cS_t(\cC)
\intertext{for $\lambda\in\cP^\cC$ and}
\overline{S}_\lambda&\coloneqq S_{\lambda(\unit_\cC)}
*\prod_{U\neq\unit_\cC}U^{\boxtimes\lambda(U)}\in\cW_d(\cC)
\end{align*}
for $\lambda\in\cP^\cC_d$.
Applying Proposition~\ref{prop:St_for_direct_sum}
to the block decomposition of $\cC$, we have
$\overline{L}_{t,\lambda}$ (resp.\ $\overline{S}_\lambda$)
is indecomposable and
all indecomposable objects in $\cS_t(\cC)$ (resp.\ $\cW_d(\cC)$) are of such form.
We can now extend Theorem~\ref{thm:comes_ostrik}, the result of Comes and Ostrik.

\begin{theorem}
\begin{enumerate}
\label{thm:blocks}
\item $\lambda\mapsto \overline{L}_{t,\lambda}$ gives a bijection $\cP^\cC\bijectivemap I(\cS_t(\cC))$.
\item If $t\not\in\NN$ then all blocks in $\cS_t(\cC)$ are trivial.
\item
For $d\in\NN$, non-trivial blocks in $\cS_d(\cC)$
are parameterize by $\cP^\cC_d$.
The non-trivial block corresponding to $\lambda\in\cP^\cC_d$
is generated by indecomposable objects
$\overline{L}_{d,\lambda^{(0)}},\overline{L}_{d,\lambda^{(1)}},\dotsc$.
Here $\lambda^{(0)},\lambda^{(1)},\dotsc\in\cP^\cC$ is given by
\begin{align*}
\lambda^{(j)}(U)\coloneqq\begin{dcases*}
\lambda(\unit_\cC)^{(j)},&if $U=\unit_\cC$,\\
\lambda(U),&otherwise.
\end{dcases*}
\end{align*}
This block is equivalent to a non-trivial block in $\cD_d$
which is described in Theorem~{\rm \ref{thm:comes_ostrik}~(3)}.
The canonical functor $\cS_d(\cC)\to\cW_d(\cC)$ sends $\overline{L}_{d,\lambda^{(0)}}$ to $\overline{S}_\lambda$ and
the other indecomposable objects to the zero object.
\end{enumerate}
\end{theorem}

\appendix

\section{Tensor categories with additional structures}
\def\appendixname{}
\label{sec:cat_with_str}

There are various kinds of additional structures on tensor categories
which are introduced in many literature (e.g.\ see \cite{Selinger:2011})
and used in various fields of mathematics, physics and even computer science.
It is straightforward to show that these structures
are compatible with standard operations on categories:
taking an envelope, a tensor product or
a category of invariants under group action.
In this appendix we introduce that
our 2-functor $\cS_t$ also respects many of them.

\subsection{Duals}

\begin{definition}
\label{def:dual}
Let $\cC$ be a tensor category.
A \term{left dual} of an object $U\in\cC$ is
an object $U^*\in\cC$ along with morphisms
$\ev_U\colon U^*\otimes U\to\unit_\cC$ and $\coev_U\colon\unit_\cC\to U\otimes U^*$
such that the composites
\begin{gather*}
\xymatrix @R 5pt @C 50pt{
U \ar[r]^-{\coev_U\otimes\id_U}&
U\otimes U^*\otimes U \ar[r]^-{\id_U\otimes\ev_U}&
U\rlap{,}\\
U^* \ar[r]^-{\id_{U^*}\otimes\coev_U}&
U^*\otimes U\otimes U^* \ar[r]^-{\ev_U\otimes\id_{U^*}}&
U^*
}
\end{gather*}
are both identities.
Such triple $(U^*,\ev_U,\coev_U)$ is unique up to unique isomorphism when it exists.
For a morphism $\varphi\colon U\to V$ between objects which have left duals,
its left dual $\varphi^*\colon V^*\to U^*$ is defined as the composite
\begin{gather*}
V^*\xrightarrow{\id_{V^*}\otimes\coev_U}
V^*\otimes U\otimes U^*\xrightarrow{\id_{V^*}\otimes\varphi\otimes\id_{U^*}}
V^*\otimes V\otimes U^*\xrightarrow{\ev_V\otimes\id_{U^*}}
U^*.
\end{gather*}

The \term{right dual} ${}^*U$ is defined similarly
with the reversed tensor product so ${}^*(U^*)\simeq U\simeq ({}^*U)^*$.
The tensor category $\cC$ is called \term{rigid} (or \term{autonomous})
if every its object has both left and right duals.
\end{definition}

By definition $\unit_\cC^*\simeq\unit_\cC$ and
there is a functorial isomorphism $(U\otimes V)^*\simeq V^*\otimes U^*$ when they exist.
In addition, if $\sigma_\cC$ is a braiding in $\cC$, $\sigma_\cC(U,V)^*=\sigma_\cC(U^*,V^*)$ via the isomorphism.
So a rigid (braided) tensor category $\cC$ is (braided) tensor equivalent to its opposite category $\cC^\op$
via the functor $U\mapsto U^*$ if we define the suitable structure on $\cC^\op$.

Note that the left dual $U^*$ of $U$ need not to be isomorphic to its right dual ${}^*U$.
In a rigid tensor category every tensor transformation $\bullet^*\to{}^*\bullet$
is automatically invertible and such a functorial isomorphism is called a \term{pivot}.

\begin{example}
When $A$ is a Hopf algebra over $\kk$,
the $\kk$-tensor category $\Rep(A)$ is rigid.
For $U\in\Rep(A)$, its left dual $U^*$ and right dual ${}^*U$ are both defined as an $A^\op$-module $\Hom_\kk(U,\kk)$
and $A$ acts on them via the antipode $\gamma_A\colon A\to A^\op$ and its inverse $\gamma_A^{-1}$ respectively.
Note that $\Mod(A)$ is not rigid since 
we can not define a suitable map $\unit_A\to U\otimes U^*$
for an arbitrary $U\in\Mod(A)$.
\end{example}

If $U\in\cC$ has a left dual $U^*$,
$\ag{U}\in\cS_t(\cC)$ also has a left dual $\ag{U^*}$.
The equipped morphisms are the composites
\begin{gather*}
\ag{U^*}\otimes\ag{U}\xrightarrow{\mu_\cC(U^*,U)}
\ag{U^*\otimes U}\xrightarrow{\ag{\ev_U}}
\ag{\unit_\cC}\xrightarrow{\epsilon_\cC}
\unit_{\cS_t(\cC)},\\
\unit_{\cS_t(\cC)}\xrightarrow{\iota_\cC}
\ag{\unit_\cC}\xrightarrow{\ag{\coev_U}}
\ag{U\otimes U^*}\xrightarrow{\Delta_\cC(U,U^*)}
\ag{U}\otimes\ag{U^*}
\end{gather*}
illustrated as
\begin{align*}
\begin{xy}
(-6,6)*={\bullet}="v"+(0,3)*{U^*};
(6,6)*={\bullet}="u"+(0,3)*{U},
(0,-4)*+{\ev_U}="p"*\frm{-},
(0,1)="z",
"v";"z" **\crv{(-5,2)},
"u";"z" **\crv{(5,2)},
"z";"p" **\dir{-},
"p" \ar @{-x} (0,-9),
\end{xy},&&
\begin{xy}
(-6,-6)*={\bullet}="v"+(0,-3)*{U};
(6,-6)*={\bullet}="u"+(0,-3)*{U^*},
(0,4)*+{\coev_U}="p"*\frm{-},
(0,-1)="z",
"v";"z" **\crv{(-5,-2)},
"u";"z" **\crv{(5,-2)},
"z";"p" **\dir{-},
"p" \ar @{-x} (0,9),
\end{xy}.
\end{align*}

Conversely, suppose that $\ag{U}$ has a left dual $\ag{U}^*$.
The equation
$\id_{\ag{U}^*}=(\ev_{\ag{U}}\otimes\id_{\ag{U}^*})\circ(\id_{\ag{U}^*}\otimes\coev_{\ag{U}})$
implies that $\id_{\ag{U}^*}$ factors through some $\ag{V}$
so $\ag{U}^*$ is isomorphic to the image of an idempotent
$f\colon\ag{V}\to\ag{V}$.
Now $f$ can be decomposed as
\begin{gather*}
f=\ag{e}+\sum_i\ag{\varphi_i}\otimes\ag{\psi_i}
\end{gather*}
by $e\colon V\to V$ and $\varphi_i\colon V\to\unit_\cC$, $\psi_i\colon\unit_\cC\to V$.
Then $e$ is also idempotent and its image $eV$ is a left dual of $U$.
The same holds for right duals and thus
$\cS_t(\cC)$ is rigid if and only if $\cC$ is rigid.

\subsection{Traces}

\begin{definition}
A \term{(right) trace} on a $\kk$-tensor category $\cC$
is a family $\{\Tr_X\}_{X\in\cC}$ of $\kk$-linear transformations
$\Tr_X\colon\Hom_\cC(\bullet\otimes X,\bullet\otimes X)\to\Hom_\cC(\bullet,\bullet)$
which satisfies
\begin{enumerate}
\item $\Tr_X(\varphi\circ(\id_U\otimes\psi))=\Tr_Y((\id_V\otimes\psi)\circ\varphi)$
for each $\varphi\colon U\otimes Y\to V\otimes X$ and $\psi\colon X\to Y$,
\item $\Tr_X(\varphi\otimes\psi)=\varphi\otimes\Tr_X(\psi)$,
\item $\Tr_{\unit_\cC}(\varphi)=\varphi$ and $\Tr_{X\otimes Y}(\varphi)=\Tr_X(\Tr_Y(\varphi))$.
\end{enumerate}
\end{definition}

We remark that if the category is rigid there is a one to one correspondence
between traces and pivots.
For a given trace we can define a pivot
$p_\cC(U)\coloneqq\Tr_U(U^*\otimes U\xrightarrow{\ev_U}\unit_\cC\xrightarrow{\coev_{{}^*U}}{}^*U\otimes U)$.
Conversely, each pivot $p_\cC\colon\bullet^*\to{}^*\bullet$ induces a trace defined by
\begin{gather*}
\Tr_X(\varphi)\coloneqq
(U\xrightarrow{\id_U\otimes\coev_X}
U\otimes X\otimes X^*\xrightarrow{\varphi\otimes p_\cC(X)}
V\otimes X\otimes{}^*X\xrightarrow{\id_V\otimes\ev_{{}^*X}}
V)
\end{gather*}
for $\varphi\colon U\otimes X\to V\otimes X$.

For each trace on $\cC$ there is a unique trace on $\cS_t(\cC)$ which satisfies
\begin{gather*}
\ag{\Tr_X(\varphi)}=
\Tr_{\ag{X}}(
\ag{U}\otimes\ag{X}\xrightarrow{\mu_\cC(U,X)}
\ag{U\otimes X}\xrightarrow{\ag{\varphi}}
\ag{V\otimes X}\xrightarrow{\Delta_\cC(V,X)}
\ag{V}\otimes\ag{X})
\end{gather*}
for every $\varphi\colon U\otimes X\to V\otimes X$.
To construct this trace it suffices to define transformations
$\Tr_{\ag{X}}$ for each $X\in\cC$.
First let $f\mapsto\bar f$ be an idempotent endomorphism on
$\Hom_{\cS_t(\cC)}(A\otimes\ag{X},B\otimes\ag{X})$
defined by
\begin{gather*}
\bar f\coloneqq(
A\otimes\ag{X}\xrightarrow{\id_A\otimes\Delta_\cC(X,\unit_\cC)}
A\otimes\ag{X}\otimes\ag{\unit_\cC}\xrightarrow{f\otimes\id_{\ag{\unit_\cC}}}
B\otimes\ag{X}\otimes\ag{\unit_\cC}\xrightarrow{\id_B\otimes\mu_\cC(X,\unit_\cC)}
B\otimes\ag{X}).
\end{gather*}
By the axioms of trace it must satisfy $\Tr_{\ag{X}}(\bar f)=\Tr_{\ag{X}}(f)$.
Now let $U_I$ and $V_J$ be families of objects in $\cC$.
The image of $\bar\bullet$ on $\Hom_{\cS_t(\cC)}(\ag{U_I}\otimes\ag{X},\ag{V_J}\otimes\ag{X})$
is the direct sum 
\begin{gather*}
\bigoplus_{\substack{i\in I\sqcup\{\varnothing\}\\j\in J\sqcup\{\varnothing\}}}
\ag{H(U_{I\setminus\{i\}};V_{J\setminus\{j\}})\otimes\Hom_\cC(U_i\otimes X,V_j\otimes X)}.
\end{gather*}
For $\Phi\in H_r(U_{I\setminus\{i\}};V_{J\setminus\{j\}})$ and
$\psi\colon U_i\otimes X\to V_j\otimes X$, the trace of $\ag{\Phi\otimes\psi}$
is defined by and must be
\begin{gather*}
\Tr_{\ag{X}}(\ag{\Phi\otimes\psi})\coloneqq
\begin{dcases*}
(t-\#r)\Tr_X(\psi)\cdot\ag{\Phi},&if $i=j=\varnothing$,\\
\ag{\Phi\otimes\Tr_X(\psi)},&otherwise.
\end{dcases*}
\end{gather*}
Then these transformations satisfy
the axioms of trace.
It is easy to prove that every trace on $\cS_t(\cC)$ is
obtained by this construction.
Note that in a braided tensor category the trace we defined satisfies the equation
$\Tr_{\ag{X}}(\tau_\cC(X,X))=\ag{\Tr_X(\sigma_\cC(X,X))}$.

\subsection{Twists}

\begin{definition}
A \term{twist} on a braided tensor category $\cC$
is a functorial isomorphism $\theta_\cC(U)\colon U\to U$ such that
$\theta_\cC(\unit_\cC)=\id_{\unit_\cC}$ and
$\theta_\cC(U\otimes V)=\sigma_\cC(V,U)\circ(\theta_\cC(V)\otimes\theta_\cC(U))\circ\sigma_\cC(U,V)$.
A \term{balanced tensor category} is a braided tensor category equipped with a twist.
It is called a \term{ribbon category} (or a \term {tortile category}) if it
is rigid and satisfies $\theta_\cC(U^*)=\theta_\cC(U)^*$.
\end{definition}

For example, each trace in $\cC$ induces a twist $\theta_\cC(U)\coloneqq\Tr_U(\sigma_\cC(U,U))$.
When $\cC$ is rigid, this trace can be recovered from the pivot
\begin{gather*}
U^*\xrightarrow{\id_{U^*}\otimes\coev_{{}^*U}}
U^*\otimes{}^*U\otimes U\xrightarrow{\sigma^{-1}_\cC(U^*,{}^*U)\otimes\theta_\cC(U)}
{}^*U\otimes U^*\otimes U\xrightarrow{\id_{{}^*U}\otimes\ev_U}
{}^*U
\end{gather*}
so pivots, traces and twists are the same things in a rigid braided tensor category.

Similarly as traces, twists on a braided tensor category $\cC$ and
those on $\cS_t(\cC)$ are in one to one correspondence
via the 2-functor $\cS_t$ for transformations with unit.
In particular, $\cS_t$ also sends a balanced tensor category to a balanced tensor category
and a ribbon category to a ribbon category.
One of the most interesting application of tensor category theory
is that a ribbon category induces an oriented link invariant
such as (a constant multiple of) the Jones polynomial
or the HOMFLY-PT polynomial.
Now let $J$ and $J_t$ be link invariants
induced by ribbon categories $\cC$ and $\cS_t(\cC)$ respectively.
One can prove that the new invariant $J_t$ only depends on $J$;
\def\trefoil{\begin{xy}\end{xy}}
for example, $J_t(\text{a knot})=t\cdot J(\text{a knot})$
and $J_t(\text{a Hopf link})=(t^2-t)\cdot J(\text{a Hopf link})+t\cdot J(\text{two trivial knots})$.



\bibliographystyle{model1b-num-names}
\bibliography{2012}







\end{document}